\documentclass[10pt]{article}
\usepackage[a4paper, footskip=15mm,footnotesep=9mm,nohead,dvips]{geometry}
\usepackage{amsmath}
\usepackage{amsfonts}
\usepackage{amssymb}
\usepackage{amsthm}
\usepackage{enumerate}
\usepackage{setspace}
\usepackage[all,cmtip]{xy}

\def\modo#1{\left| #1 \right|}

\numberwithin{equation}{section}
\theoremstyle{plain}
\newtheorem{theorem}{Theorem}[section]
\newtheorem{proposition}[theorem]{Proposition}
\newtheorem{lemma}[theorem]{Lemma}
\newtheorem{corollary}[theorem]{Corollary}
\newtheorem{definition}[theorem]{Definition}
\newtheorem{rem}[theorem]{Remark}
\newtheorem{exa}[theorem]{Example}  
\newtheorem{conj}[theorem]{Conjecture} 
\newtheorem{fact}[theorem]{Fact}
\newtheorem{notation}[theorem]{Notation}

\begin{document}

\title{New Permutation Representations \\ of The Braid Group}
\date{July 15, 2008}
\author{Amiel Ferman\footnote{This article was submitted as part of the first author's PhD Thesis at Bar-Ilan University, Ramat-Gan Israel, 
written under the supervision of Mina Teicher and Tahl Nowik}, Tahl Nowik, Robert Schwartz, Mina Teicher \\ \\ 
Department of Mathematics, Bar-Ilan University \\ Ramat-Gan, Israel }
\maketitle

\begin{abstract}

We give a new infinite family of group homomorphisms from
the braid group $B_k$ to the symmetric group $S_{mk}$ for all $k$ and $m \geq 2$.
Most known permutation representations of braids are included in this family.
We prove that the homomorphisms in this family are non-cyclic and transitive.
For any divisor $l$ of $m$, $1\leq l < m$, we prove in particular that if $\frac{m}{l}$ is odd
then there are $1 + \frac{m}{l}$ non-conjugate homomorphisms included in our family. 
We define a certain natural restriction on homomorphisms $B_k\to S_n$,
common to all homomorphisms in our family,
which we term {\em good}, and of which there are two types.
We prove that all good homomorphisms $B_k\to S_{mk}$ of type $1$ are
included in the infinite family of homomorphisms we gave.
For $m=3$, we prove that all good homomorphisms $B_k\to S_{3k}$ of type $2$ are
also included in this family. Finally, we refute a conjecture made in \cite{MaSu05}
regarding permutation representations of braids and give an updated conjecture.

\end{abstract}

\section{Introduction}

The study of permutation representations of braids, i.e., homomorphisms from the braid group $B_k$ on $k$ strings
to the symmetric group $S_n$, is closely related to a variety of fields such as
algebraic equations with function coefficients, algebraic functions, configuration spaces and their
holomorphic maps and more. For more information on these relations see \cite{Arn70a,Arn70b,Arn70c}, \cite{GorLin69,GorLin74},
\cite{Kal75,Kal76,Kal93}, \cite{Kurth97},
\cite{Lin71,Lin72b,Lin79,Lin96a,Lin96b,Lin03,Lin04a}.

Most results known to date on permutation representations of braids were obtained by Artin (\cite{Artin}) and by Lin (\cite{Lin04b}).
Some results were also obtained by Matei and Suciu in \cite{MaSu05} where permutation representations of $B_3$ 
and $B_4$ are discussed and computed (we shall address a conjecture made in that paper later on). Let us also
mention the work of Arouch in \cite{Ab07} who has applied an algorithmic point of view.

Lin (\cite{Lin04b}) gives an explicit description of all non-cyclic transitive homomorphisms $B_k\to S_{2k}$ for $k>6$ up to conjugation. 
He proves that there are exactly three such homomorphisms up to conjugation. 
In our work we give a previously unknown infinite family of group homomorphisms $B_k \to S_{mk}$ for all $k>8$ and $m \geq 2$.
In fact, our description of these homomorphisms (up to conjugation), can be seen to be a generalization
of Lin's result. Furthermore, we show that this family can be described essentially by a single formula.
Interestingly enough, this single formula includes almost all permutation representations known to date.
We prove that all homomorphisms in this family are non-cyclic and transitive.
Furthermore, for any divisor $l$ of $m$, $1\leq l < m$, we prove in particular that if $\frac{m}{l}$ is odd
then there are $1 + \frac{m}{l}$ non-conjugate homomorphisms included in our family 
(in fact, the condition on $m$ and $l$ which guarantees $1 + \frac{m}{l}$ 
non-conjugate homomorphisms is more general than that - see proposition~\ref{prop m l cond}).
We define a certain natural restriction on homomorphisms $B_k\to S_n$, common to all
homomorphisms in our family, defined as follows

We call a homomorphism $\omega \colon B_k\to S_n$ \textbf{good} if one of the following holds

\begin{enumerate}
\item $supp(\widehat \sigma_i) \cap supp(\widehat \sigma_j) = \emptyset$ for $1\leq i,j \leq k-1$ such that $|i-j| > 1$.\\
\item $supp(\widehat \sigma_1) = \cdots = supp(\widehat \sigma_{k-1})$
\end{enumerate}

We prove that all transitive non-cyclic good homomorphisms $B_k\to S_{mk}$ of type $1$ are
included in the infinite family of homomorphisms we gave.
For $m=3$, we prove that all good homomorphisms $B_k\to S_{3k}$ of type $2$ are
also included in this family. 

We state our conjecture that all non-cyclic transitive homomorphisms $B_k\to S_n$ for $n>2k$ and $k>8$ are good.
This conjecture is based on our findings as well as those of Lin(\cite{Lin04b})~: There is only
a finite number of known examples for non-cyclic transitive homomorphisms $B_k\to S_n$ which are not good.
These examples are given for small values of $k$ and $n$ (see section~$4$ in \cite{Lin04b}).

Our paper is organized as follows : 

In \textbf{section~\ref{sec defs}} we give the definitions and notations which we shall use throughout the article. We define
the important notion of (co)retraction and (co)reduction of a homomorphism $B_k\to S_m$ and prove a relevant key result.
In \textbf{section~\ref{sec intro chap perm rep}} we present our new family of permutation representations of braids and
prove various results regarding it as follows : In \textbf{subsection~\ref{subsec model perm rep}} we give a new family of 
homomorphisms $B_k\to S_{mk}$ and prove these are transitive and non-cyclic. We also exhibit how this family generalizes 
the known family of homomorphisms $B_k\to S_{2k}$ found by Lin in \cite{Lin04b}.
In \textbf{subsection~\ref{subsec new perm rep}} we define our new family of homomorphisms $B_k\to S_{mk}$ and prove
that these homomorphisms are well defined, non-cyclic and transitive.
In \textbf{subsection~\ref{subsec good perm rep}} we define the notion of {\it good} permutation representations of braids, common to the homomorphisms
in our family, of which there are two types and we prove some of its properties.
In \textbf{subsection~\ref{subsec conj model homs}} we prove that there are non-trivial conjugacies between the homomorphisms defined in
our new family of homomorphisms $B_k\to S_{mk}$. In particular, for any divisor $l$ of $m$, $1\leq l < m$, 
we prove that if $\frac{m}{l}$ is odd then there are $1 + \frac{m}{l}$ non-conjugate homomorphisms 
included in our family. In \textbf{section~\ref{section B_k S_mk supp=2m}} we prove that 
all type $1$ good transitive non-cyclic homomorphisms $B_k \to S_{mk}$
are included in the family of homomorphisms we gave in section~\ref{sec intro chap perm rep}. 
In \textbf{section~\ref{section B_k S_3k supp=mk}} we prove that all type $2$ good transitive non-cyclic homomorphisms $B_k \to S_{3k}$
are included in the family of homomorphisms we gave in section~\ref{sec intro chap perm rep}. 
Finally, in \textbf{section~\ref{section conjecture}} we give an updated conjecture (updating the conjecture made in \cite{MaSu05}) 
regarding permutation representations of braid groups and we explain a way to construct homomorphisms $B_k\to S_n$
for $n\neq mk$. We give an infinite family of non-cyclic transitive homomorphisms $B_k\to S_{\frac{k(k-1)}{2}}$
thus refuting a conjecture made in \cite{MaSu05}.

\section{Definitions and Basic Properties}\label{sec defs}

Let us fix a few definitions and notations (most of our notations are taken from \cite{Lin04b}).

\begin{definition}
The \textbf{canonical presentation} of the braid group $B_k$ on $k$ strings consists of $k-1$ generators $\sigma_1,\ldots,\sigma_{k-1}$
and relations

\begin{eqnarray}
&&\sigma_i\sigma_j
=\sigma_j\sigma_i\qquad\qquad\qquad \  (\modo{i-j}\ge 2)\,,\\
&&\sigma_i\sigma_{i+1}\sigma_i
=\sigma_{i+1}\sigma_i\sigma_{i+1}\qquad (1\le i\le k-2)\,.                                             
\end{eqnarray}

\hfill $\bullet$

\end{definition}

\begin{notation}\label{notation braid pair}
We denote a conjugate pair of elements $g,h$ of a group $G$ as $g\sim h$. We denote conjugation 
as usual by $a^b = b^{-1} a b $ and we denote the commutator as usual $[a,b] = aba^{-1}b^{-1} $. The pair $g,h$ 
is called {\bf braid-like} if $ghg = hgh$; in this case we write $g\infty h$.
Clearly, $g\infty h$ implies $g\sim h$. \hfill $\bullet$
\end{notation}
\vskip0.2cm

\begin{definition}[\textbf{Conjugate homomorphisms}]
Two group homomorphisms $\phi,\psi\colon G\to H$ are called {\bf conjugate}
if there exists an element $h\in H$ such that
$\psi(g)=h^{-1}\phi(g)h$ for all $g\in G$; in this case
we write $\phi\sim\psi$ and $\psi = \phi^h$; ``$\sim$" is an equivalence
relation in $Hom(G,H)$.\hfill $\bullet$
\end{definition}
\vskip0.2cm

\begin{definition}[\textbf{Fixed points, Supports and Invariant sets}]\label{def fixed points}
We regard the symmetric group on a set $\Gamma$, $\mathbf{S}(\Gamma)$, as acting on $\Gamma$ from the left.
For any element $s$ of $\mathbf{S}(\Gamma)$ let $\mathbf{Fix(s)}$ denote the set $\{\gamma\in\Gamma|\, s(\gamma)=\gamma\}$
of all fixed points of $s$; the complement $ \mathbf{supp(s) =\Gamma\setminus Fix(s)}$
is called the {\bf support} of $s$. If $\mathfrak C = \{ s_1,\ldots,s_n\} $ is a set of permutations
then we denote $supp(\mathfrak C) = \bigcup_{i=1,\ldots,n} supp(s_i)$. 
For $\Sigma\subseteq\Gamma$ we
identify ${\mathbf S}(\Sigma)$ with the subgroup 
$\{s\in{\mathbf S}(\Gamma)|\, supp(s)\subseteq\Sigma\}
\subseteq{\mathbf S}(\Gamma)$. $s,t\in{\mathbf S}(\Gamma)$ are
{\bf disjoint} if $ supp(s)\cap supp(t)=\emptyset$. For $H\subseteq{\mathbf S}(\Gamma)$ a subset
$\Sigma\subseteq\Gamma$ is $H$-{\em invariant} if
$s(\Sigma)=\Sigma$ for all $s\in H$; we denote by $\mathbf{Inv (H)}$ the
family of all {\sl non-trivial} (i. e., $\ne\emptyset$ and
$\ne\Gamma$) $H$-invariant subsets of $\Gamma$. \hfill $\bullet$
\end{definition}

\begin{rem}\label{rem convention dis perms}
If $\Gamma$ and $\Gamma'$ are two disjoint sets, and $A \in \mathbf{S}(\Gamma), B\in \mathbf{S}(\Gamma')$ are two 
permutations on them, then it will be our convention in this work that the multiplication $AB$ should mean
a permutation in $\mathbf{S}(\Gamma \cup \Gamma')$. \hfill $\bullet$
\end{rem}

\begin{notation}\label{notation S_n}
${\mathbf S_n}$ denotes the symmetric
group of degree $n$, that is, the permutation group
${\mathbf S}({\boldsymbol\Delta}_n)$ of the $n$ point set
${\boldsymbol\Delta}_n=\{1,2,\ldots ,n\}$.
The {\em alternating subgroup} ${\mathbf A_n}\subset{\mathbf S_n}$
consists of all even permutations $S\in{\mathbf S_n}$ and coincides
with the commutator subgroup ${\mathbf S'_n}$. \hfill $\bullet$
\end{notation}

\begin{definition}\label{def canonical}
The \textbf{canonical epimorphism} $\mu = \mu_k \colon B_k \to {\mathbf S_k}$ is defined by 

$$\mu(\sigma_i) = (i,i+1)\in {\mathbf S_k}$$

\hfill $\bullet$
\end{definition}

\begin{definition}
We say that a group homomorphism $\psi \colon G \to {\mathbf S_n}$ is \textbf{transitive} if its image $\psi(G)$ is transitive,
i.e. if $\psi(G)$ as a subgroup of $S_n \cong \mathbf{S}(\Delta_n)$ acts transitively on $\Delta_n$ (see Definition~\ref{def fixed points}
and Notation~\ref{notation S_n}). $\psi$ is called \textbf{cyclic} if its image $\psi(G)$ is a cyclic subgroup of $S_n$.

\hfill $\bullet$
\end{definition}

Let us record the following

\begin{fact}\label{fact braid commute}
Let $a,b$ be any two elements in a group which satisfy $a \infty b$ (i.e. $a$ and $b$ are a braid-like pair) and $[a,b]=1$,
then we have that $b=a^{ba}=a^a = a $. \hfill $\bullet$
\end{fact}

This fact implies the following

\begin{lemma}\label{lemma cyclic homomorphism}
A homomorphism $\omega\colon B_k\to S_n$ is cyclic iff $\varphi(\sigma_i) = \varphi(\sigma_{i+1})$
for some $i\in\{1,\ldots,k-2 \}$ iff $\varphi(\sigma_1) = \cdots \varphi(\sigma_{k-1})$.

\end{lemma}

\begin{proof}

If $\omega$ is cyclic then in particular we have that for each $i=1,\ldots,k-2$

$$  [\omega(\sigma_i),\omega(\sigma_{i+1})]=1  $$

but since $\omega$ is a homomorphism we also have

$$ \omega(\sigma_i) \infty \omega(\sigma_{i+1}) $$

and this implies by Fact~\ref{fact braid commute} that $\omega(\sigma_i)=\omega(\sigma_{i+1})$
for all $i=1,\ldots,k-2$.

Conversely, if $\omega(\sigma_i) = \omega(\sigma_{i+1})$ for some $i\in \{1,\ldots,k-2\}$, 
then $\omega(\sigma_{i+2}) \infty \omega(\sigma_{i+1})$ 
and $1 = [\omega(\sigma_{i+2}),\omega(\sigma_i)] = 
[\omega(\sigma_{i+2}),\omega(\sigma_{i+1})]$ implies by Fact~\ref{fact braid commute} 
that $\omega(\sigma_{i+1}) = \omega(\sigma_{i+2})$. Applying the same argument for each consecutive index gives 
that $\omega(\sigma_1) = \cdots = \omega(\sigma_{k-1})$ and hence $Im(\omega)$ is a cyclic subgroup of $S_n$
generated by $\omega(\sigma_1)$. 
\end{proof}

\begin{definition}[\textbf{Cyclic types, Components}]
Let $A = C_1\cdots C_q$ be the cyclic
decomposition of $A\in{\mathbf S_n}$ and
$r_i\ge 2$ be the length of the cycle
$C_i$ \ ($1\le i\le q$); the {\sl unordered} $q$-tuple of
the natural numbers $[r_1,\ldots ,r_q]$ is called
the {\bf cyclic type} of $A$ and is denoted by $[A]$
(each $r_i$ occurs in $[A]$ as many times
as $r_i$-cycles occur in the cyclic decomposition of $A$).
Clearly, $ord(A) = LCM(r_1,\ldots, r_q)$
(the least common multiple of $r_1,...,r_q$). For $A\in{\mathbf S_n}$ and a natural $r\ge 2$ we denote by
\index{r-component of a permutation\hfill}
\index{${\mathfrak C}_r(A)$\hfill}
${\mathfrak C}_r(A)$ the set of all $r$-cycles that occur in the cyclic
decomposition of $A$; we call this set the ${\mathbf r}$ {\bf - component}
of $A$. The {\bf length} of an $r$-component ${\mathfrak C}_r(A)$ is the number $|{\mathfrak C}_r(A)|$.
The set $\ Fix(A)$ is called the {\bf degenerate component}
of $A$.  \hfill $\bullet$
\end{definition}
\vskip0.2cm

\begin{notation}
For a permutation representation $\omega\colon B_k\to S_n$ we shall sometimes denote $\omega(\sigma_i)$ as $\widehat \sigma_i$
when $\omega$ is understood from the context. \hfill $\bullet$
\end{notation}

\begin{definition}[\textbf{Disjoint product of homomorphisms}]\label{Disjoint product}
Given a disjoint decomposition 
${\boldsymbol\Delta}_n=D_1\cup\cdots\cup D_q$, \ $ |D_j|=n_j$,
we have the corresponding embedding
${\mathbf S}(D_1)\times\cdots\times{\mathbf S}(D_q)
\hookrightarrow{\mathbf S_n}$.
For a family of group homomorphisms
$\psi_j\colon G\to{\mathbf S}(D_j)\cong{\mathbf S_{n_j}}$, \
$j=1,...,q$, we define the {\bf disjoint product}
$$
\psi=\psi_1\times\cdots\times\psi_q\colon G\to
{\mathbf S}(D_1)\times\cdots\times{\mathbf S}(D_q)
\hookrightarrow{\mathbf S_n}
$$ 
by $\psi(g)=\psi_1(g)\cdots\psi_q(g)\in{\mathbf S_n}$, \ $g\in G$.\hfill $\bullet$
\end{definition}

\begin{definition}[\textbf{Reductions of homomorphisms to $\mathbf{S_n}$}]
Let $\omega\colon G\to{\mathbf S_n}$ be a group homomorphism, $H = Im(\omega)$
and $\Sigma\in Inv(H)$ (for instance, $\Sigma$ may be an $H$-orbit).
Define the homomorphisms
$$
\phi_\Sigma\colon H\ni s\mapsto s|\Sigma\in{\mathbf S}(\Sigma)
\quad\text{and}\quad
\omega_{\Sigma}=\phi_\Sigma\circ\omega\colon G\stackrel{\omega}{\longrightarrow} H 
\stackrel{\phi_\Sigma}{\longrightarrow} {\mathbf S}(\Sigma)\,;
$$
the composition $\omega_{\Sigma}=\phi_\Sigma\circ\omega$
\index{Reduction of homomorphism\hfill}
is called the {\bf reduction of $\mathbf{\omega}$ to $\mathbf{\Sigma}$}; it is
transitive if and only if $\Sigma$ is
an $Im(\omega)$-orbit. Any homomorphism $\omega$ is the disjoint
product of its reductions to all $Im(\omega)$-orbits
(this is just the decomposition of the representation $\omega$
into the direct sum of irreducible representations). \hfill $\bullet$
\end{definition}

\begin{definition}[\textbf{Retractions and Coretractions}]\label{def retractions}
Let $\omega\colon B_k \to S_n$ be a homomorphism.
Let $ \mathfrak D_r = \{ D_1,\ldots,D_t\} $ be an ordered set of $r$-cycles in the cycle decomposition of $\widehat \sigma_1$ for some $r\ge 2$. 
We call $\mathfrak D_r$ an $\mathbf{r}$-\textbf{subcomponent} of $\widehat \sigma_1$. Since $\widehat\sigma_3,...,\widehat\sigma_{k-1}$ 
commute with $\widehat\sigma_1$,
the conjugation of $\widehat\sigma_1$ by any of the elements 
$\widehat\sigma_i$ ($3\le i\le k-1$) induces
a certain permutation of the $r$-cycles in the cycle decomposition of $\widehat \sigma_1$; suppose this permutation
permutes the cycles of the $r$-subcomponent $\mathfrak D_r$ among themselves, then this gives rise to a homomorphism
$$
\Omega_{\mathfrak D_r}\colon B_{k-2}\to{\mathbf S}(\mathfrak D_r)\cong {\mathbf S_t},
$$
where $B_{k-2}$ is the subgroup of $B_k$ generated
by $\sigma_3,...,\sigma_{k-1}$ and each $\sigma_i$ is sent to the permutation it induces on the $t$ $r$-cycles in $\mathfrak D_r$.

The homomorphism $\Omega_{\mathfrak D_r}$ is called the \textbf{retraction} of $\omega$ to the $r$-subcomponent $\mathfrak D_r$. 

Furthermore, since 

$$ supp(\mathfrak D_r) = \bigcup_{i=1,\ldots,t} supp(D_i) \in Inv(\{ \widehat \sigma_3,\ldots,\widehat \sigma_{k-1} \}) $$

(since, according to our assumption, the conjugation of $\widehat \sigma_1$ by 
each permutation in $\{ \widehat \sigma_3,\ldots,\widehat \sigma_{k-1} \}$ 
induces a permutation of the cycles in $\mathfrak D_r$ among themselves), we can define
the reduction of $\omega$ to $supp(\mathfrak D_r)$~:

$$ \omega_{supp(\mathfrak D_r)}\colon B_{k-2} \to \mathbf{S}(supp(\mathfrak D_r)) $$

where $B_{k-2}$ is the subgroup of $B_k$ generated by $\sigma_3,...,\sigma_{k-1}$. The homomorphism $\omega_{supp(\mathfrak D_r)}$
thus defined is called the \textbf{reduction of $\mathbf{\omega}$ corresponding to the $\bf r$-subcomponent $\bf \mathfrak D_r$}.

Similarly, suppose that $\mathfrak E_r = \{ E_1,\ldots,E_t \}$ is an $r$-subcomponent of $\widehat \sigma_{k-1}$ and
suppose that $\widehat\sigma_1,...,\widehat\sigma_{k-3}$ permute the cycles in $\mathfrak E_r$ among themselves then
we can define

$$
\Omega_{\mathfrak E_r}\colon B^*_{k-2}\to{\mathbf S}(\mathfrak E_r)\cong {\mathbf S_t},
$$
where $B^*_{k-2}$ is the subgroup of $B_k$ generated
by $\sigma_1,...,\sigma_{k-3}$ and each $\sigma_i$ is sent to the permutation it induces on the $t$ $r$-cycles in $\mathfrak E_r$.

The homomorphism $\Omega_{\mathfrak E_r}$ is called the \textbf{coretraction} of $\omega$ to the $r$-subcomponent $\mathfrak E_r$.

Furthermore, since 

$$ supp(\mathfrak E_r) = \bigcup_{i=1,\ldots,t} supp(E_i) \in Inv(\{ \widehat \sigma_1,\ldots,\widehat \sigma_{k-3} \}) $$

(since, according to our assumption, the conjugation of $\widehat \sigma_{k-1}$ with 
each permutation in $\{ \widehat \sigma_1,\ldots,\widehat \sigma_{k-3} \}$ 
induces a permutation of the cycles in $\mathfrak E_r$ among themselves), 
we can define the reduction of $\omega$ to $supp(\mathfrak E_r)$~:

$$ \omega_{supp(\mathfrak E_r)}\colon B^*_{k-2} \to \mathbf{S}(supp(\mathfrak E_r)) $$

where $B^{*}_{k-2}$ is the subgroup of $B_k$ generated by $\sigma_1,...,\sigma_{k-3}$. The homomorphism $\omega_{supp(\mathfrak E_r)}$
thus defined is called the \textbf{coreduction of $\mathbf{\omega}$ corresponding to the $\bf r$-subcomponent $\bf \mathfrak E_r$}

\hfill $\bullet$

\vskip0.2cm
\end{definition}

We will use the following results (based on sections~$5$ and $6$ in~\cite{Lin04b})

%(see Thm. 6.17 there) :

% \begin{lemma}[Lemma 5.3 in \cite{Lin04b}]
% \label{lemma conj omega}
% Let $\omega\colon B_k \to S_n$ be a homomorphism and let $\mathfrak D_r$ be the $r$-component
% of $\widehat\sigma_1$ for some $r\ge 2$. Then the retraction and the
% coretraction of $\omega$ to $\mathfrak D_r$ are conjugate. i.e.

% $$ \Omega_{\mathfrak D_r} \sim \Omega_{\mathfrak D^*_r} $$

% \hfill $\bullet$
% \end{lemma}

\begin{proposition}\label{prop commut retract}
Let $\omega\colon B_k \to S_n$ be a homomorphism. Let $\mathfrak D_r$ be an $r$-subcomponent of $\widehat \sigma_1$ with $t$ cycles. %$\{D_1,\ldots,D_t \} $. 
Let $\Omega_{\mathfrak D_r}$ be the retraction of $\omega$ to the $r$-subcomponent 
$\mathfrak D_r$ (assuming it is well defined) 
and let $\omega_{supp(\mathfrak D_r)}$ be the corresponding reduction as defined in Definition~\ref{def retractions}.
Then there exists a commuting diagram as follows

\begin{equation}\label{eq: CD with Omega lemma}
\xymatrix{
&& B_{k-2} \ar[d]_{\omega_{supp(\mathfrak D_r)}} \ar[rrd]^{\Omega_{\mathfrak D_r}} \\
H\ar[rr]^{i} && G\ar[rr]^{\pi} && \mathbf{S}(\mathfrak D_r)\cong S_t}
\end{equation}

where $B_{k-2}$ is the subgroup of $B_k$ generated by $\{\sigma_3,\ldots,\sigma_{k-1} \}$, $G$ is a finite group, $H$ is a finite abelian group 
and the horizontal line in \eqref{eq: CD with Omega lemma} is an exact sequence.

Similarly, let $\mathfrak E_r$ be an $r$-subcomponent of $\widehat \sigma_{k-1}$ with $t$ cycles. %$\{E_1,\ldots,E_t \} $. 
Let $\Omega_{\mathfrak E_r}$ be the coretraction of $\omega$ to the $r$-subcomponent 
$\mathfrak E_r$ (assuming it is well defined) 
and let $\omega_{supp(\mathfrak E_r)}$ be the corresponding coreduction as defined in Definition~\ref{def retractions}.
Then there exists a commuting diagram as follows

\begin{equation}\label{eq: CD with Omega lemma*}
\xymatrix{
&& B^*_{k-2} \ar[d]_{\omega_{supp(\mathfrak E_r)}} \ar[rrd]^{\Omega_{\mathfrak E_r}} \\
H^*\ar[rr]^{i} && G^*\ar[rr]^{\pi} && \mathbf{S}(\mathfrak E_r)\cong S_t}
\end{equation}

where $B^*_{k-2}$ is the subgroup of $B_k$ generated by $\{\sigma_1,\ldots,\sigma_{k-3} \}$, $G^*$ is a finite group, $H^*$ is a finite abelian group 
and the horizontal line in \eqref{eq: CD with Omega lemma*} is an exact sequence.

\end{proposition}

\begin{proof}
Let us first prove the claim for diagram \eqref{eq: CD with Omega lemma}.

Denote the $t$ $r$-cycles in $\mathfrak D_r$ as $\{ D_1,\ldots, D_t \}$ and denote $\Sigma = supp(\mathfrak D_r)$.
As usual, we denote $\omega(\sigma_i)$ as $\widehat \sigma_i$ for each $i$, $1\leq i \leq k-1$.

Let $G$ be the centralizer of the element $\mathcal D = D_1\cdots D_t$ in 
$\mathbf{S}(\Sigma)$.Since any element $g\in G$ commutes with the element
$\mathcal D = D_1\cdots D_t$, we have

\begin{equation}\label{eq def pi HGS}
D_1\cdots D_t = g^{-1}\cdot D_1\cdots D_t \cdot g = g^{-1}D_1g \cdots g^{-1}D_tg\qquad\text{for any} \ \ g\in G,
\end{equation}

in \eqref{eq def pi HGS} we have that $D_1,\ldots ,D_t$ are disjoint $r$-cycles
as well as $g^{-1}D_1g,\ldots , g^{-1}D_tg$.
Hence there is a unique permutation $\tau_g\in{\mathbf S}(\mathfrak D_r)$ such that
$g^{-1} D_j g = \tau_g(D_j)$ for all $j$.
Let us define $\pi\colon G\to \mathbf{S}(\mathfrak D_r)\cong S_t$ as follows

%From \eqref{eq def pi HGS} we can see that conjugation by $g$ of $\mathcal D$ induces a permutation on the cycles of $\mathfrak D_r$. 

\begin{equation}\label{eq def pi HGS'}
\pi(g) = \tau_g \quad \text{where} \quad \tau_g(D_i) = g^{-1}D_i g. 
\end{equation}

Let $H \cong (\mathbb Z_r)^t$ be the abelian subgroup of $\mathbf{S}(\Sigma)$ generated by the $r$-cycles 
$\{ D_1,\ldots, D_t\}$. Note that since each $D_i$ commutes with $\mathcal D$ we have that $H\subset G$. 
We will now prove that $H=Ker(\pi)$~: 

Suppose $\pi(g) = id$ for some $g\in G$. 
Then it follows from \eqref{eq def pi HGS'} that each $D_i$ commutes with $g$ and according to Lemma~\ref{Lemma commuting permutations}
this means that $g|_{supp(D_i)}$ is equal to some power of $D_i$, i.e., $g\in H$.

Conversely, if $h=\prod_j D_j^{c_j}\in H$ where $c_j\in \mathbb N$, $0\leq c_j\leq r-1$, then for each $i=1,\ldots,k-1$ we have

$$ h^{-1} D_i h = (\prod_j D_j^{c_j})^{-1} D_i \prod_j D_j^{c_j} = D_i $$   

so $h\in Ker(\pi)$. Hence $H = Ker(\pi)$ is an abelian normal subgroup of $G$.

Thereby, we obtain the exact sequence

\begin{equation}\label{exact HGS}
\xymatrix{
H\ar[r]^{i} & G\ar[r]^-{\pi} & \mathbf{S}(\mathfrak D_r)\cong S_t}
\end{equation}

where $i$ is the embedding of $H$ in $G$ and $\pi$ is as defined in \eqref{eq def pi HGS'}.

Consider the following diagram, which includes \eqref{exact HGS}

\begin{equation}\label{eq: CD with Omega''}
\xymatrix{
&& B_{k-2} \ar[d]_{\omega_{supp(\mathfrak D_r)}} \ar[rrd]^{\Omega_{\mathfrak D_r}} \\
H\ar[rr]^{i} && G\ar[rr]^{\pi} && \mathbf{S}(\mathfrak D_r)\cong S_t}
\end{equation}

First note that the image of $\omega_{supp(\mathfrak D_r)}$ indeed lies in $G$ : For each $i$, $3\leq i \leq k-1$, we have
by Definition~\ref{def retractions} that $\omega_{supp(\mathfrak D_r)}(\sigma_i)$ is equal to $\widehat \sigma_i|_{\Sigma}$
and that conjugation of $\widehat \sigma_1$ by $\widehat \sigma_i|_{\Sigma}$ induces a permutation of the $r$-cycles of $\mathfrak D_r$
which means in particular that each  $\widehat \sigma_i|_{\Sigma}$ commutes with $\mathcal D = D_1\cdots D_t$, hence
$Im(\omega_{supp(\mathfrak D_r)})\subseteq G$.

Furthermore, by Definition~\ref{def retractions}, $\Omega_{\mathfrak D_r}$ send $\sigma_i$, $3\leq i\leq k-1$, to the
permutation that is induced on the $r$-cycles in $\mathfrak D_r$ by conjugating $\widehat \sigma_1$ with $\widehat \sigma_i$
or, what is equivalent, by conjugating $\widehat \sigma_1$ with $\widehat \sigma_i|_{\Sigma}$.
But this is exactly $\pi \circ \omega_{supp(\mathfrak D_r)}(\sigma_i)$ according to the definition of $\pi$ in \eqref{eq def pi HGS'}.
This means that the diagram \eqref{eq: CD with Omega''} is commutative.

Now if $\mathfrak E_r = \{E_1,\ldots,E_t \}$ is an $r$-subcomponent of $\widehat \sigma_{k-1}$ with $t$ cycles and %$\{E_1,\ldots,E_t \} $. 
$\Omega_{\mathfrak E_r}$ is the coretraction of $\omega$ to the $r$-subcomponent 
$\mathfrak E_r$ and $\omega_{supp(\mathfrak E_r)}$ is the corresponding coreduction as defined in Definition~\ref{def retractions},
then the analogous claim about \eqref{eq: CD with Omega lemma*} is proved in a 
completely similar way where in this case $H^*\cong (\mathbb Z_r)^t$ is the
subgroup of $\mathbf{S}(supp(\mathfrak E_r))$ generated by the cycles of $\mathfrak E_r$ and $G^*$ is the centralizer
of the element $\mathcal E = E_1\cdots E_t$ in $\mathbf{S}(supp(\mathfrak E_r))$.

\end{proof}

\begin{corollary}
\label{cor cyclic omega}
Let $\omega\colon B_k \to S_n$ be a homomorphism. If the (co)retraction of $\omega$ to an $r$-subcomponent for some $r \ge 2$ is cyclic
then the (co)reduction of $\omega$ to that $r$-subcomponent is cyclic as well. \hfill $\bullet$ 
\end{corollary}
\begin{proof}

Let $\mathfrak D_r$ be an $r$-subcomponent of $\widehat \sigma_1$ with $t$ cycles. Denote the $t$ $r$-cycles in $\mathfrak D_r$
as $\{ D_1,\ldots, D_t \}$. Let $\Omega_{\mathfrak D_r}$ be the retraction of $\omega$ to the $r$-subcomponent 
$\mathfrak D_r$ (assuming it is well defined) 
and let $\omega_{supp(\mathfrak D_r)}$ be the corresponding reduction as defined in Definition~\ref{def retractions}.

For the rest of the proof let us denote $\Sigma = supp(\mathfrak D_r)$. According to Proposition~\ref{prop commut retract}
we have the following commuting diagram

\begin{equation}\label{eq: CD with Omega cor}
\xymatrix{
&& B_{k-2} \ar[d]_{\omega_{\Sigma}} \ar[rrd]^{\Omega_{\mathfrak D_r}} \\
H\ar[rr]^{i} && G\ar[rr]^{\pi} && \mathbf{S}(\mathfrak D_r)\cong S_t}
\end{equation}

where $B_{k-2}$ is the subgroup of $B_k$ generated by $\{\sigma_3,\ldots,\sigma_{k-1} \}$, $G$ is a finite group, $H$ is a finite abelian group 
and the horizontal line in \eqref{eq: CD with Omega cor} is an exact sequence.

Denote the commutator subgroup of $B_{k-2}$ as $B'_{k-2}$. Suppose $\Omega_{\mathfrak D_r}$ is cyclic. In particular, the image
of $\Omega_{\mathfrak D_r}$ is abelian. Hence, since the homomorphic image of any commutator
in an abelian group is trivial, we have that  $\Omega_{\mathfrak D_r}(B'_{k-2}) = \{1\}$. 
By Proposition~\ref{prop commut retract} $\Omega_{\mathfrak D_r} = \pi\circ \omega_{\Sigma}$ and so
 $\pi\circ \omega_{\Sigma}(B'_{k-2}) = \{1\}$. Hence $\omega_{\Sigma}(B'_{k-2})$ is contained in the 
abelian group $Ker(\pi)=H$ and again since the homomorphic image of any commutator
in an abelian group is trivial, we have that $\omega_{\Sigma}(B'_{k-2})=\{1\}$. But since $B_{k-2}/B'_{k-2}\cong \mathbb Z$
(see section~$1.3$ in \cite{Lin04b}), this means that $\omega_{\Sigma}$ factors 
through $B_{k-2}/B'_{k-2}\cong \mathbb Z$, i.e. $\omega_{\Sigma}$ is cyclic. 

The dual claim regarding the coretractions and their corresponding coreduction follows by the same arguments
applied to an $r$-subcomponent of $\widehat \sigma_{k-1}$ and using \eqref{eq: CD with Omega lemma*} of
Proposition~\ref{prop commut retract}.

\end{proof}

\section{New Permutation Representations of the Braid Group}\label{sec intro chap perm rep}

\subsection{Model and Standard Permutation Representations}\label{subsec model perm rep}

Permutation representations of braids were first studied by Artin in \cite{Artin2} where he proves

\noindent{ARTIN THEOREM}
\index{Artin Theorem\hfill|phantom\hfill}
\index{Theorem!of Artin\hfill}
\index{Theorem!on homomorphisms!$B_k\to{\mathbf S_k}$\hfill}
\index{Homomorphisms!$B_k\to{\mathbf S_k}$\hfill}
{\sl Let $\psi\colon B_k\to{\mathbf S_k}$ be a non-cyclic
transitive homomorphism. Denote $\alpha = \sigma_1\sigma_2\cdots \sigma_{k-1}$. 
($\alpha$ and $\sigma_1$ generate $B_k$, see \cite{Lin04b})
\vskip0.2cm

$a)$ If $k\ne 4$ and $k\ne 6$ then $\psi$ is conjugate to the
canonical epimorphism $\mu=\mu_k$.
\vskip0.2cm

$b)$ If $k = 6$ and $\psi\not\sim\mu_6$
then $\psi$ is conjugate to the homomorphism $\nu_6$ 
\index{Homomorphism $\nu_6$\hfill}
\index{Artin homomorphism $\nu_6$\hfill}
\index{$\nu_6$\hfill}
defined by
%&
\begin{equation}\label{Artin homomorphism vu6}
\nu_6(\sigma_1) = (1,2)(3,4)(5,6),
\qquad \nu_6(\alpha) = (1,2,3)(4,5).
\end{equation}
%&

$c)$ If $k = 4$ and $\psi\not\sim\mu_4$ 
then $\psi$ is conjugate to one of the 
three homomorphisms $\nu_{4,1}$, $\nu_{4,2}$ and $\nu_{4,3}$
\index{Homomorphisms!$\nu_{4,1}$, \ $\nu_{4,2}$, \ $\nu_{4,3}$\hfill}
\index{Artin homomorphisms $\nu_{4,1}$, $\nu_{4,2}$, $\nu_{4,3}$\hfill}
\index{$\nu_{4,1}$, $\nu_{4,2}$, $\nu_{4,3}$\hfill}
defined by
%&
\begin{equation}\label{Artin homomorphisms vu41-43}
\aligned
\nu_{4,1}(\sigma_1) &= (1,2,3,4),\qquad \\
\nu_{4,2}(\sigma_1) &= (1,3,2,4),\qquad \\
\nu_{4,3}(\sigma_1) &= (1,2,3), \qquad  \\
\endaligned
\aligned
\nu_{4,1}(\alpha) &= (1,2);\qquad \\
\nu_{4,2}(\alpha) &= (1,2,3,4);\qquad \\
\nu_{4,3}(\alpha) &= (1,2)(3,4);\qquad \\
\endaligned
\aligned
\lbrack\nu_{4,1}(\sigma_3) &= \nu_{4,1}(\sigma_1)\rbrack\\
\lbrack\nu_{4,2}(\sigma_3) &= \nu_{4,2}(\sigma_1^{-1})\rbrack\\
\lbrack\nu_{4,3}(\sigma_3) &= \nu_{4,3}(\sigma_1)\rbrack.\\
\endaligned
\end{equation}
%&

$d)$ $\psi$ is surjective except for the case
when $k=4$, $\psi\sim\nu_{4,3}$ and $Im (\psi) ={\mathbf A_4}$.}

\hfill $\blacksquare$

The next important work on the subject was done in \cite{Lin04b}. Among other results, Lin classifies
homomorphisms $B_k\to S_{2k}$ for $k > 6$ and we quote this result here

\begin{definition}
\label{Def B_k S_2k}
\index{Model homomorphisms $B_k\to{\mathbf S}(2k)$\hfill}
\index{Standard homomorphisms $B_k\to{\mathbf S}(2k)$\hfill}
\index{Homomorphisms!$B_k\to{\mathbf S}(2k)$!model\hfill}
\index{Homomorphisms!$B_k\to{\mathbf S}(2k)$!standard\hfill}
The following three homomorphisms
$\varphi_1,\varphi_2,\varphi_3\colon B_k\to{\mathbf S_{2k}}$
are called the \textbf{model} ones:
$$
\aligned
&\varphi_1(\sigma_i) = 
\underbrace{(2i-1,2i+2,2i,2i+1)}_{\text {$4$-cycle}};\\
%&\\
&\varphi_2(\sigma_i)=(1,2)(3,4)\cdot\cdot\cdot (2i-3,2i-2)
\underbrace{(2i-1,2i+1)(2i,2i+2)}_{\text {two transpositions}}\times\\
&\hskip2.9in
\times
(2i+3,2i+4)\cdot\cdot\cdot(2k-1,2k);\\
%\\
&\varphi_3(\sigma_i)= (1,2)(3,4)\cdot\cdot\cdot (2i-3,2i-2)
\underbrace{(2i-1,2i+2,2i,2i+1)}_{\text {$4$-cycle}}\times\\
&\hskip2.9in
\times(2i+3,2i+4)\cdot\cdot\cdot (2k-1,2k);\\
\endaligned
$$
in each of the above formulas $i=1,...,k-1$.
A homomorphism \ $\psi\colon B_k\to{\mathbf S_{2k}}$ \
is said to be \textbf{standard} if it is conjugate to one of
the three model homomorphisms \ $\varphi_1$, \ 
$\varphi_2$, \ $\varphi_3$. \hfill $\bullet$
\end{definition}

We will use the following 

\begin{lemma}\label{lemma cyclic part model}
Let $\phi\colon B_k\to S_{2k}$ be a standard homomorphism which is conjugate to either of the model homomorphisms
$\varphi_2$ or $\varphi_3$ as defined in Definition~\ref{Def B_k S_2k}. If there exist permutations
$R, R_1,\ldots, R_{k-1}\in S_{2k}$ with $R$ disjoint from $ R_1,\ldots, R_{k-1}$, such that 

$$ \phi(\sigma_i) = R\cdot R_i \qquad \text{for} \quad i=3,\ldots,k-1 $$ 

or 

$$ \phi(\sigma_i) = R\cdot R_i \qquad \text{for} \quad i=1,\ldots,k-3 $$

then $|supp(R)| \leq 4$.
\end{lemma}
\begin{proof}

Let us prove the claim in case $\phi$ is conjugate to $\varphi_2$ and $i=3,\ldots,k-1$. The other three
cases ($\phi$ is conjugate to $\varphi_3$ or $i=1,\ldots,k-3$) follow in a completely similar way.

First assume $\phi=\varphi_2$. Recall the definition
of this model homomorphism (see Definition~\ref{Def B_k S_2k})

$$
\aligned
&\varphi_2(\sigma_i)=(1,2)(3,4)\cdot\cdot\cdot (2i-3,2i-2)
\underbrace{(2i-1,2i+1)(2i,2i+2)}_{\text {two transpositions}}\times\\
&\hskip2.9in
\times
(2i+3,2i+4)\cdot\cdot\cdot(2k-1,2k);\\
%\\
% &\varphi_3(\sigma_i)= (1,2)(3,4)\cdot\cdot\cdot (2i-3,2i-2)
% \underbrace{(2i-1,2i+2,2i,2i+1)}_{\text {$4$-cycle}}\times\\
% &\hskip2.9in
% \times(2i+3,2i+4)\cdot\cdot\cdot (2k-1,2k).\\
\endaligned
$$

We see that

\begin{equation}\label{100.1}
\varphi_2(\sigma_i) = (1,2)(3,4)\cdot S_i \qquad \text{for} \quad i=3,\ldots,k-1  
\end{equation}

where

$$ 
\aligned
S_i = (5,6)(7,8)\cdot (9,10)\cdot\cdot\cdot (2i-3,2i-2)\cdot
(2i-1,2i+1)(2i,2i+2)\cdot \\
\hskip2.9in
\cdot
(2i+3,2i+4)\cdot\cdot\cdot(2k-1,2k) 
\endaligned
$$

for $i=3,\ldots,k-1$. 

Now suppose that

\begin{equation}\label{100.2}
\varphi_2(\sigma_i) = R\cdot R_i \qquad \text{for} \quad i=3,\ldots,k-1 
\end{equation}

for some $R,R_i\in S_{2k}$ where $R$ is disjoint from the $R_i$. We claim that $R$ is disjoint from each $S_i$
in \eqref{100.1} for $i=3,\ldots,k-1$. For suppose, by way of contradiction, that $supp(R)\cap supp(S_j)\neq \emptyset$ 
for some $j$,$3\leq j \leq k-1$. 
Then equating \eqref{100.1} and \eqref{100.2} for $i=j$ gives $R\cdot R_j = (1,2)(3,4)S_j$ and so the cyclic decomposition of $R$
contains some cycle in the cyclic decomposition of $S_j$. This means that the cyclic decomposition of $R$ contains some cycle 
of the form $(l,l+1)$ or $(l,l+2)$ for some $l$, $5 \leq l\leq 2k-1$. Now if the cyclic decomposition of $R$ contains $(l,l+1)$
then equating \eqref{100.1} and \eqref{100.2} for $l'=min(\lfloor \frac{l+1}{2} \rfloor, k-1)$ gives a contradiction since 
the cyclic decomposition of $S_{l'}$ does not contain $(l,l+1)$. Furthermore, if 
the cyclic decomposition of $R$ contains $(l,l+2)$ then equating \eqref{100.1} and \eqref{100.2} for any
$l''$, $3\leq l'' \leq k-1$, such that $|l''-l'|>1$, gives a contradiction since 
the cyclic decomposition of $S_{l''}$ does not contain $(l,l+2)$.

We conclude that $R$ in \eqref{100.2} is disjoint from each $S_i$ in \eqref{100.1} and by equating
\eqref{100.1} and \eqref{100.2} we see that each cycle in the cyclic decomposition of $R$ must be included
in the cyclic decomposition of $(1,2)(3,4)$. Hence $R$ could be only one of four 
possibilities : $id, \ (1,2),\  (3,4)$ or $(1,2)(3,4)$. In any case $|supp(R)|\leq 4$.

Now suppose $\phi$ is conjugate to $\varphi_2$ and 

$$ \phi(\sigma_i) = R\cdot R_i \qquad \text{for} \quad i=3,\ldots,k-1 $$ 

for some $R, R_1,\ldots, R_{k-1}\in S_{2k}$ where $R$ is disjoint from $R_1,\ldots, R_{k-1}$. By definition there exists a $C\in S_{2k}$ such that
$\varphi_2 = \phi^C $, so we can write

$$ \varphi_2(\sigma_i) = R^C\cdot R^C_i \qquad \text{for} \quad i=3,\ldots,k-1 $$ 

but clearly $R^C$ is disjoint from the $R^C_i$ and according to the argument above we have 
that $|supp(R^C)| \leq 4$ and so necessarily $|supp(R)|\leq 4$ which concludes the proof.

% In a completely similar way we see that

% $$ \varphi_2(\sigma_i) = (2k-3,2k-3)(2k-1,2k)\cdot R'_i \qquad \text{for} \quad i=1,\ldots,k-3  $$

% and 

% $$ \varphi_3(\sigma_i) = (1,2)(3,4)\cdot T_i \qquad \text{for} \quad i=3,\ldots,k-1  $$

% $$ \varphi_3(\sigma_i) = (2k-3,2k-3)(2k-1,2k)\cdot T'_i \qquad \text{for} \quad i=1,\ldots,k-3  $$

% and the claim follows for $\phi=\varphi_2$ or $\varphi_3$.

\end{proof}

The following results are proved in \cite{Lin04b}, we give the references within.

\begin{theorem}\label{Thm F}
$a)$ For $6<k<n<2k$ every transitive
homomorphism $B_k\to{\mathbf S_n}$ is cyclic (Theorem F~$(a)$ in \cite{Lin04b}).

$b)$ 
\index{Standard homomorphisms $B_k\to{\mathbf S}(2k)$\hfill}
\index{Homomorphisms!$B_k\to{\mathbf S}(2k)$!standard\hfill}
\index{Theorem!on homomorphisms!$B_k\to{\mathbf S_{2k}}$ for $k>6$\hfill}
For $k>6$ every non-cyclic transitive homomorphism
$\psi\colon B_k\to{\mathbf S_{2k}}$
is standard (Theorem F~$(b)$ in \cite{Lin04b}).

$c)$
For $6<k<n<2k$ every non-cyclic homomorphism
$\psi\colon B_k\to{\mathbf S_n}$ is
conjugate to a homomorphism of the form $\mu_k\times\widetilde\psi$,
where $\widetilde\psi\colon B_k\to{\mathbf S_{n-k}}$
is a cyclic homomorphism (Theorem $7.26$~$(b)$ in \cite{Lin04b}).

$d)$
for $k>6$ every non-cyclic homomorphism
$\psi\colon B_k\to{\mathbf S_k}$ is conjugate
to $\mu_k$ (Theorem $3.9$ in \cite{Lin04b}).

$e)$
for $k > 4$ and $n<k$ every homomorphism
$\psi\colon B_k\to{\mathbf S_n}$ is cyclic (Theorem A~$(a)$ in \cite{Lin04b}).
\end{theorem}

\subsection{New Model and Standard Permutation Representations}\label{subsec new perm rep}

In this subsection we define new model and standard permutation representations of the braid group, extending
those defined in Definition~\ref{Def B_k S_2k}. We begin with a few notations and preliminary results.

\begin{notation}\label{notation A}
Throughout the paper we shall use the following notation : 
For each $m,i \in \mathbb N$, we denote by $A^m_i$ the following permutation

\begin{equation}\label{def A_i}
A^m_i = (1 + m(i-1),\ldots, mi) 
\end{equation}

so each $A^m_i$ is an $m$-cycle in $\mathbf{S}(\{ 1 + m(i-1),\ldots,mi \})$. We shall sometimes denote $A^m_i$ as $A_i$ when $m$ is 
understood from the context. Also, we denote for each $j$, $0\leq j\leq m-1$

$$ a^i_j = 1 + m(i-1) + j \in supp(A^m_i) $$

Furthermore, we denote $\bigtriangledown_m = \{0,\ldots,m-1\} \subset \mathbb Z$.

Finally, for $a\in \mathbb Z$ we shall denote by $|a|_m\in \bigtriangledown_m$ the residue of $a$ modulo $m$
(e.g. $|8|_5 = 3\in \bigtriangledown_5 \subset \mathbb Z$).% If $m=1$ we take $|a|_1 = 0 $ for all $a\in \mathbb Z$.
%We also denote $\bigtriangledown_m = \{0,\ldots,m-1\} \subset \mathbb Z$. When $m=1$ we take  $\bigtriangledown_1 = \{ 0 \}$. 
\hfill $\bullet$
\end{notation}

We now prove

\begin{lemma}\label{lemma C}
Let $2\leq m,r \in \mathbb N$ and let $\sigma\in S_r$ be some permutation, then for every choice of $r$ elements

$$ t_i \in \mathbb \bigtriangledown_m \quad 1 \leq i \leq r $$

we can define an element $C\in \mathbf{S}( \cup_{i=1\ldots r} supp(A^m_i) ) \cong S_{mr}$ such that

\begin{equation}\label{eq C}
(A^m_i)^{C^{-1}} = A^m_{\sigma(i)} \quad 1 \leq i \leq r
\end{equation}

by the formula

\begin{equation}\label{form C}
C(a^i_j) = a^{\sigma(i)}_{|j + t_{\sigma(i)}|_m}, \quad a^i_j = 1 + m(i-1) + j \in Supp(A_i) 
\end{equation}

for each $1\leq i \leq r$ and $j\in \mathbb \bigtriangledown_m$. 

If $\sigma$ is an $r$-cycle, the cyclic decomposition
of $\ C$ consists of exactly $\frac{m}{l}$ $lr$-cycles where $l$ is the minimal number
between $1$ and $m$ such that

%$$ l(t_i + t_{\sigma(i)} + \cdots + t_{\sigma^{r-1}(i)}) \equiv 0 (\text{mod} \ m). $$

$$ l\cdot \sum_{j=0}^{r-1} t_{\sigma^j(i)} (=  l\cdot \sum_{j=0}^{r-1} t_{j}) \equiv 0 (\text{mod} \ m). $$

%Hence $t_i + t_{\sigma(i)} + \cdots + t_{\sigma^{r-1}(i)}(\text{mod} \ m)$ determines the cycle
Hence in this case $\sum_{j=0}^{r-1} t_{\sigma^j(i)}$ determines the cycle
structure of $\ C$. 

Furthermore, each $C\in \mathbf{S}( \cup_{i=1\ldots r} supp(A^m_i) )$ which satisfies 
equation~\ref{eq C} is given by the formula~\ref{form C} for some choice of $\ t_i\in  \bigtriangledown_m$.

\end{lemma}

\begin{proof}

To show that equation~\ref{form C} defines $C$ as a permutation in $S_{mr}$ it is enough to show that it is injective 
as a map from $\Delta_{mr}$ to $\Delta_{mr}$. Assume then that $C(a^{i_1}_{j_1})=C(a^{i_2}_{j_2})$, by equation~\ref{form C} this means that

$$ a^{\sigma(i_1)}_{|j_1 + t_{\sigma(i_1)}|_m} = a^{\sigma(i_2)}_{|j_2 + t_{\sigma(i_2)}|_m} $$

so necessarily $i_1=i_2$ and

$$ j_1 + t_{\sigma(i_1)} \equiv  j_2 + t_{\sigma(i_2)} (\text{mod} \ m)$$

which means that

$$ j_1 \equiv j_2 (\text{mod} \ m) $$

but since $j_1,j_2 \in \bigtriangledown_m$ this means that $j_1=j_2$ and so $C$ is indeed a well defined permutation.

To see the cyclic decomposition of $C$ in case $\sigma$ is an $r$-cycle, note that for any $n\in \mathbb N$ 

$$ C^n(a^i_j) = a^{\sigma^n(i)}_{|j + t_{\sigma(i)} + \cdots + t_{\sigma^n(i)}|_m} $$

hence, since $\sigma \in S_r$ is an $r$-cycle, we have that

$$ C^r(a^i_j) = a^{i}_{|j + t_i + t_{\sigma(i)} + \cdots + t_{\sigma^{r-1}(i)}|_m} $$

and

$$ C^{lr}(a^i_j) = a^{i}_{|j + l(t_i + t_{\sigma(i)} + \cdots + t_{\sigma^{r-1}(i)})|_m} $$

hence the minimal $n$ such that $C^n(a^i_j) = a^i_j$ is $n=rl$ where $l$ is the minimal number between $1$ and $m$
such that 

$$ l(t_i + t_{\sigma(i)} + \cdots + t_{\sigma^{r-1}(i)}) \equiv 0 (\text{mod} \ m) $$

so the cyclic decomposition of $C$ consists of exactly $\frac{m}{l}$ $lr$-cycles.

Now let $C\in S_{mr}$ satisfy the following equation

$$ (A^m_i)^{C^{-1}} = A^m_{\sigma(i)} \quad 1 \leq i \leq k $$

then clearly we have that for each $1\leq i \leq r$ there exists some $t_i \in \bigtriangledown_m$ such that

$$ C(a^i_0) = a^{\sigma(i)}_{|t_{\sigma(i)}|_m} $$

but then necessarily 

$$ C(a^i_j) = a^{\sigma(i)}_{|j + t_{\sigma(i)}|_m} $$

\end{proof}

\begin{notation}\label{notation C}
% If $\sigma \in \mathbf{S}(\{j,\ldots,j+r-1 \})\cong S_r$ is an $r$-cycle then we denote by $C^{\sigma,m}_{t_1,\ldots,t_r}$
% that permutation in $\mathbf{S}(\{ \cup_{i=j\ldots j+r-1} supp(A^m_i) \}) \cong S_{mr}$ given by Lemma~\ref{lemma C} which satisfies

% \begin{equation}\label{eq notation C}
% (A^m_i)^{C^{\sigma,m}_{t_j,\ldots,t_{j+r-1}}} = A^m_{\sigma(i)}, \quad j \leq i \leq j + r -1 
% \end{equation}

% and is given by the formula~\ref{form C} (where each $t_i \in \mathbb \bigtriangledown_m$). 

Let $\sigma \in \mathbf{S}(\{j,\ldots,j+r-1 \})\cong S_r$ be a permutation
and let $t_1,\ldots,t_r\in \bigtriangledown_m$. We denote by $C^{\sigma,m}_{t_1,\ldots,t_r}$
the permutation in $\mathbf{S}(\{ \cup_{i=j\ldots j+r-1} supp(A^m_i) \}) \cong S_{mr}$ which is given by the
following formula (see notation~\ref{notation A})

\begin{equation}\label{form notation C}
C^{\sigma,m}_{t_1,\ldots,t_r}(a^i_q) = a^{\sigma(i)}_{|q + t_{\sigma(i)}|_m}, \quad a^i_q = 1 + m(i-1) + q \in Supp(A^m_i) 
\end{equation}

for each $j\leq i \leq j+r-1$ and $q\in \mathbb \bigtriangledown_m$.

(that \eqref{form notation C} defines a permutation - see remark~\ref{rem C m=1}).

\hfill $\bullet$
\end{notation}

\begin{rem}\label{rem C m=1}

To see that \eqref{form notation C} indeed defines a permutation we have to show that $C^{\sigma,m}_{t_1,\ldots,t_r}(a^i_q)$
as a function on $\cup_{i=j\ldots j+r-1} supp(A^m_i)$ is injective. But if 

$$ C^{\sigma,m}_{t_1,\ldots,t_r}(a^{i_1}_{q_1}) = C^{\sigma,m}_{t_1,\ldots,t_r}(a^{i_2}_{q_2}) $$

then by \eqref{form notation C} this means that

$$ a^{\sigma(i_1)}_{|q_1 + t_{\sigma(i_1)}|_m} = a^{\sigma(i_2)}_{|q_2 + t_{\sigma(i_2)}|_m} $$

and so $\sigma(i_1) = \sigma(i_2)$ which implies that $i_1=i_2$ (since $\sigma$ is a permutation). Hence let us
denote $i=i_1=i_2$. So we now have that

$$ |q_1 + t_{\sigma(i)}|_m =  |q_2 + t_{\sigma(i)}|_m $$

which means that $|q_1|_m = |q_2|_m$, but since $q_1,q_2\in \bigtriangledown_m$, it follows that $q_1=q_2$, and so
$a^{i_1}_{q_1} = a^{i_2}_{q_2}$.

Let also us remark on the special case $m=1$ in Notation~\ref{notation C}. 
If $\sigma \in \mathbf{S}(\{j,\ldots,j+r-1 \})\cong S_r$ is an $r$-cycle then

\begin{equation}\label{eq rem C}
(A^1_i)^{C^{\sigma,1}_{t_j,\ldots,t_{j+r-1}}} = A^1_{\sigma(i)}, \quad j \leq i \leq j + r -1. 
\end{equation}

But $A^1_i = (i)$ and $t_i \in \mathbb \bigtriangledown_1 = \{ 0 \}$ so we can rewrite \eqref{eq rem C} as

$$ (i)^{C^{\sigma,1}_{0,\ldots,0}} = \sigma(i), \quad j \leq i \leq j + r -1 $$

which means that 

$$ C^{\sigma,1}_{0,\ldots,0} = \sigma. $$

\hfill $\bullet$

\end{rem}

\begin{exa}\label{example C}
Let us give an example for an application of Lemma~\ref{lemma C}. Choosing $m=6, r=4$ 
and $\sigma = (1,3,2,4) \in S_4$ we have

$$
\aligned
&A_1^6 = (1,2,3,4,5,6)\\
&A_2^6 = (7,8,9,10,11,12)\\
&A_3^6 = (13,14,15,16,17,18)\\
&A_4^6 = (19,20,21,22,23,24)\\
\endaligned
$$

now choose $t_1 = 2, t_2 = 1, t_3 = 0, t_4 = 5$. Since $t_1 + t_2 + t_3 + t_4\equiv 2 (\text{mod} \ 6)$, and
since the minimal $l$ (between $1$ and $m=6$) such that $2l\equiv 0 (\text{mod} \ 6)$ is $l=3$, we expect
the cycle decomposition of $C^{\sigma,6}_{2,1,0,5}$ to consist of exactly $(\frac{m}{l} = \frac63 =) 2$
$(rl=)12$-cycles. Indeed

$$ C^{(1,3,2,4),6}_{2,1,0,5} = 
\aligned
&(1,13,8,19,3,15,10,21,5,17,12,23)\cdot\\
&(2,14,9,20,4,16,11,22,6,18,7,24)\\
\endaligned $$

As a simpler example, take $m=3, r=2$ and $\sigma=(1,2)\in S_2$. Here we have

$$
\aligned
&A_1^3 = (1,2,3)\\
&A_2^3 = (4,5,6)\\
\endaligned
$$

note that since $m=3$ is prime then for any choice of $t_i$, the sum of the $t_i$ would be either $0 (\text{mod} \ 3)$ or coprime to $m=3$. This means
that there are only two possible cycle decompositions of $C^{(1,2),3}_{t_1,t_2}$ : It either consists of 
three transpositions (in which case $l=1, \frac{m}{l} = 3$ and $rl=2$) or it is a single $6$-cycle
(in which case $l=3, \frac{m}{l} = 1$ and $rl = 6$).
Let us exhibit the two possibilities. If we choose for example $t_1 = 1, t_2 = 2$ then in this case $t_1 + t_2 \equiv 0 (\text{mod} \ 3)$
and we get

$$ C^{(1,2),3}_{1,2} = (1,6)(2,4)(3,5) $$

choosing $t_1 = 1, t_2 = 0$ we get

$$ C^{(1,2),3}_{1,0} = (1,4,2,5,3,6) $$

As an example for the more general formula in Notation~\ref{notation C} let us take $m=2, r=2$ and the $k-1$
transpositions $(i,i+1)$ for $i=1\ldots k-1$ for some $k\in \mathbb N$. We have

$$ A_i^2 = (2i-1,2i) $$

As in the case of $m=3$, since $2$ is prime there are only two possible cycle decompositions of $C^{(i,i+1),2}_{t_1,t_2}$ : Either a single 
$4$-cycle or two transpositions. Let us give examples exhibiting both possibilities

$$ C^{(i,i+1),2}_{0,1} = (2i-1,2i+2,2i,2i+1) \quad i=1\ldots k-1 $$

$$ C^{(i,i+1),2}_{1,1} = (2i-1,2i+2)(2i,2i+1) \quad i=1\ldots k-1 $$

note that we can rewrite the definition of the model homomorphisms from Definition~\ref{Def B_k S_2k} in the following way

$$ \varphi_1(\sigma_i) = C^{(i,i+1),2}_{0,1} $$

$$ \varphi_2(\sigma_i) = A_1\cdots A_{i-1}\cdot C^{(i,i+1),2}_{0,0} \cdot A_{i+2} \cdots A_k  $$

$$ \varphi_3(\sigma_i) = A_1\cdots A_{i-1}\cdot C^{(i,i+1),2}_{0,1} \cdot A_{i+2} \cdots A_k  $$

\hfill $ \bullet $
\end{exa}

We are now ready to present our main results. We first exhibit a family of homomorphisms $B_k \to S_{mk}$ 
for every $2\leq m \in \mathbb N$ which is a natural generalization of Definition~\ref{Def B_k S_2k}.

Let $m \in \mathbb N$, for each $6 < k \in \mathbb N$ we define

\begin{definition}\label{Def B_k S_mk}

The following homomorphisms
$\psi,\phi \colon B_k\to{\mathbf S_{mk}}$
are called the \textbf{model} ones:

$$
\psi_m(\sigma_i) =  C^{(i,i+1),m}_{1,0}
$$

in other words, $\psi_m(\sigma_i)$ is a $2m$-cycle whose support is 
$supp(A^m_i) \cup supp(A^m_{i+1}) = \{ 1 + m(i-1), 2 + m(i-1),\ldots, m(i+1) \} $ 
and for each $0 \leq j \leq m-1$

$$ 
\aligned
&\psi_m(\sigma_i)(a^i_j) = a^{i+1}_j \\
&\psi_m(\sigma_i)(a^{i+1}_j) = a^i_{j+1} \\
\endaligned
$$

where 

$$ a^i_j = 1 + m(i-1) + j \in supp(A^m_i) \quad j\in \bigtriangledown_m. $$

Still put differently, we can write

$$ \psi_m(\sigma_i) = (a^i_0, a^{i+1}_0, a^i_1, a^{i+1}_1,\ldots,a^i_{m-1},a^{i+1}_{m-1}) $$

(note that $\psi_1 = \mu$ the canonical epimorphism, see Definition~\ref{def canonical}).

We now use $\psi$ to define $\phi$. For each divisor $l$ of $m$ where $1\leq l < m$ 
and for each set of elements $ \{ t^i_{(i-1)l + 1},\ldots,t^i_{(i+1)l} \}_{i=1,\ldots,k-1} $ in $\bigtriangledown_{\frac{m}{l}}$ 

which satisfy the following condition

\begin{equation}\label{eq condition t}
\forall j\in \bigtriangledown_l \quad t^i_{(i-1)l + 1 + |1+j|_l} + t^i_{il + 1 + |1+ j|_l} \equiv
t^{i+1}_{(i+1)l + 1 + j} + t^{i+1}_{il + 1 + |1+ j|_l} (\text{mod} \ \frac{m}{l})
\end{equation}

define

\begin{equation}\label{eq model phi}
\phi_{m,l,\{ t^i_{(i-1)l + 1},\ldots,t^i_{(i+1)l } \}_{i=1,\ldots,k-1} }(\sigma_i) = A^{\frac{m}{l}}_1 \cdots A^{\frac{m}{l}}_{(i-1)l} 
\cdot C^{\psi_l(\sigma_i),\frac{m}{l}}_{t^i_{l(i-1)+1},\ldots,t^i_{l(i+1)} } \cdot A^{\frac{m}{l}}_{(i+1)l+1} \cdots A^{\frac{m}{l}}_{kl} 
\end{equation}

Equation \eqref{eq condition t} is a necessary and sufficient condition 
for \eqref{eq model phi} to be a homomorphism as we prove in Lemma~\ref{lemma condition C_i}.

Note that in \eqref{eq model phi}, since $\psi_l(\sigma_i)$ is a $2l$-cycle in $\mathbf{S}(l(i-1)+1,\ldots,l(i+1))$, 
each permutation $C^{\psi_l(\sigma_i),\frac{m}{l}}_{t^i_{l(i-1)+1},\ldots,t^i_{l(i+1)} }$ is considered as a permutation in 
$\mathbf{S}(\bigcup_{j=l(i-1)+1,\ldots,l(i+1)} supp(A^{\frac{m}{l}}_j))$ (see also remark~\ref{rem convention dis perms}).

In each of the above formulas $i=1,...,k-1$.
A homomorphism \ $\omega\colon B_k\to{\mathbf S_{mk}}$ \
is said to be {\bf standard} if it is conjugate to one of
the model homomorphisms $ \psi,\phi \colon B_k\to{\mathbf S_{mk}}$. 

\hfill $\bullet$
\end{definition}

\begin{rem}

Note that by Remark~\ref{rem C m=1}, in the case $m=l$ in Definition~\ref{Def B_k S_mk}, 
we have that all $t_i$ are in $\bigtriangledown_1 = \{0\}$ and

$$ \phi_{m,m,\{ t^i_{(i-1)l + 1},\ldots,t^i_{(i+1)l } \}_{i=1,\ldots,k-1} }(\sigma_i) =  
C^{\psi_m(\sigma_i),1}_{t^i_{l(i-1)+1},\ldots,t^i_{l(i+1)} } = C^{\psi_m(\sigma_i),1}_{0,\ldots,0}  
= \psi_m(\sigma_i). $$

Hence $\psi$ can be considered as a special case of $\phi$.

We can also see from Example~\ref{example C} that $\varphi_2 = \phi_{2,1,\{ t^i_i = 0, t^i_{i+1} = 0 \}}$ and 
$\varphi_3 = \phi_{2,1,\{ t^i_i = 0, t^i_{i+1} = 1 \}}$. Furthermore, as a consequence of our results in Section~\ref{section B_k S_mk supp=2m}
we have that $\varphi_1$ is conjugate to $\psi_2$ and so we conclude that the model homomorphisms which we define in Definition~\ref{Def B_k S_mk}
are indeed a generalization of the homomorphisms defined in Definition~\ref{Def B_k S_2k}. \hfill $\bullet$

\end{rem}

We dedicate the rest of this subsection to prove that the maps defined in Definition~\ref{Def B_k S_mk} are
indeed transitive and non-cyclic homomorphisms. We begin with the following

\begin{lemma}\label{lemma psi}
Let $m\in \mathbb N$. The map $\psi_m\colon B_k \to S_{mk}$ defined as follows (see also Definition~\ref{Def B_k S_mk})

$$ \psi_m(\sigma_i) = (a^i_0, a^{i+1}_0, a^i_1, a^{i+1}_1,\ldots,a^i_{m-1},a^{i+1}_{m-1}), \quad 1\leq i \leq k-1$$

where $ a^i_j = 1 + m(i-1) + j \in supp(A^m_i) $ and $j\in \bigtriangledown_m$, is a well defined non-cyclic transitive homomorphism.
\end{lemma}

\begin{proof}
For each $j\in \bigtriangledown_m$ and $1\leq i \leq k-1$ we have 

$$ 
\aligned
&\psi_m(\sigma_i)\psi_m(\sigma_{i+1})\psi_m(\sigma_i)(a^i_j) =\\
&\psi_m(\sigma_i)\psi_m(\sigma_{i+1})(a^{i+1}_j) =\\
&\psi_m(\sigma_i)(a^{i+2}_j) = a^{i+2}_j \\
\endaligned
$$

$$ 
\aligned
&\psi_m(\sigma_{i+1})\psi_m(\sigma_i)\psi_m(\sigma_{i+1})(a^i_j) =\\
&\psi_m(\sigma_{i+1})\psi_m(\sigma_i)(a^i_j) =\\
&\psi_m(\sigma_{i+1})(a^{i+1}_j) = a^{i+2}_j \\
\endaligned
$$

$$ 
\aligned
&\psi_m(\sigma_i)\psi_m(\sigma_{i+1})\psi_m(\sigma_i)(a^{i+1}_j) =\\
&\psi_m(\sigma_i)\psi_m(\sigma_{i+1})(a^i_{j+1}) =\\
&\psi_m(\sigma_i)(a^i_{j+1}) = a^{i+1}_{j+1} \\
\endaligned
$$

$$ 
\aligned
&\psi_m(\sigma_{i+1})\psi_m(\sigma_i)\psi_m(\sigma_{i+1})(a^{i+1}_j) =\\
&\psi_m(\sigma_{i+1})\psi_m(\sigma_i)(a^{i+2}_j) =\\
&\psi_m(\sigma_{i+1})(a^{i+2}_j) = a^{i+1}_{j+1} \\
\endaligned
$$

$$ 
\aligned
&\psi_m(\sigma_i)\psi_m(\sigma_{i+1})\psi_m(\sigma_i)(a^{i+2}_j) =\\
&\psi_m(\sigma_i)\psi_m(\sigma_{i+1})(a^{i+2}_j) =\\
&\psi_m(\sigma_i)(a^{i+1}_{j+1}) = a^i_{j+2} \\
\endaligned
$$

$$ 
\aligned
&\psi_m(\sigma_{i+1})\psi_m(\sigma_i)\psi_m(\sigma_{i+1})(a^{i+2}_j) =\\
&\psi_m(\sigma_{i+1})\psi_m(\sigma_i)(a^{i+1}_{j+1}) =\\
&\psi_m(\sigma_{i+1})(a^i_{j+2}) = a^i_{j+2} \\
\endaligned
$$

Hence, since $a^q_j\not \in supp(\psi_m(\sigma_i)) \cup supp(\psi_m(\sigma_{i+1}))$ for $q\neq i,i+1,i+2$ and $j\in \bigtriangledown_m$ , 
we have that for each $1\leq i \leq k-1$ 

$$ \psi_m(\sigma_i)\psi_m(\sigma_{i+1})\psi_m(\sigma_i) = \psi_m(\sigma_{i+1})\psi_m(\sigma_i)\psi_m(\sigma_{i+1}). $$

Furthermore, we have that $\psi_m(\sigma_i)\psi_m(\sigma_j) = \psi_m(\sigma_j)\psi_m(\sigma_i)$ for $|i-j| > 1$ since
for these indices $supp(\psi_m(\sigma_i)) \cap supp(\psi_m(\sigma_j)) = (supp(A_i)\cup supp(A_{i+1})) \cap (supp(A_j) \cup supp(A_{j+1})) = \emptyset$.

Hence we conclude that $\psi_m\colon B_k \to S_{mk}$ is indeed a homomorphism. It is clear that $\psi_m$ is non-cyclic
since $\psi_m(\sigma_i) \neq \psi_m(\sigma_j)$ for $i\neq j$. To see that $\psi_m$ is transitive, consider
the orbit of $1\in \Delta_{mk}$ in $Im(\psi_m)$; it includes $supp(\psi_m(\sigma_1)) = supp(A_1) \cup supp(A_2)$ and since
$supp(\psi_m(\sigma_i)) \cap supp(\psi_m(\sigma_{i+1})) = supp(A_{i+1})$ then it is easy to see by induction on $i$
that the orbit of $1$ includes $\cup_{i=1,\ldots,k} supp(A_i) = \Delta_{mk}$.

\end{proof}

Before proceeding to show that the rest of the maps we defined in Definition~\ref{Def B_k S_mk} are
well defined non-cyclic and transitive homomorphisms, we prove a few preliminary results

\begin{lemma}\label{lemma prep phi}
Let $k,l\in \mathbb N$ and $A_1,\ldots,A_{kl},C_1,\ldots,C_{k-1}$ be permutations in $S_n$ for some 
$n \in \mathbb N$ which satisfy the following restrictions for each $i=1,\ldots,k-1$

\begin{eqnarray}
&A_r^{C_i} =
\begin{cases}
A_r & r = 1,\ldots,(i-1)l \quad  \text{or} \quad r=(i+1)l + 1,\ldots,kl\label{eq prep comm}\\
A_{\psi_l(\sigma_i)(r)} & r = (i-1)l + 1,\ldots, (i+1)l\label{prep eq AC}\\
\end{cases}\\
&[A_r,A_s] = 1 \quad r,s=1,\ldots,kl\label{prep eq A}
\end{eqnarray}

where $\psi_l$ is the model homomorphism from Definition~\ref{Def B_k S_mk}, then we have that $C_i \infty C_{i+1}$
(i.e. $C_i$ and $C_{i+1}$ are a braid-like pair) if and only if $D_i \infty D_{i+1}$ where

$$ D_i = A_1\cdots A_{(i-1)l} \cdot C_i \cdot A_{(i+1)l + 1} \cdots A_{kl} $$

for each $i=1,\ldots,k-1$

\end{lemma}

\begin{proof}

First note that since

$$ \psi_l(\sigma_i)( \{ (i-1)l + 1,\ldots,(i+1)l \} ) = \{ (i-1)l + 1,\ldots,(i+1)l \} $$

we have that the conditions~\ref{prep eq AC} and \ref{prep eq A} imply that

$$ (A_{(i-1)l + 1} \cdots A_{(i+1)l})^{C_i} = A_{(i-1)l + 1} \cdots A_{(i+1)l} $$

and so we can deduce the following two equations

\begin{eqnarray}
&C_i^{ A_{(i-1)l + 1}\cdots A_{il} \cdot A_{il + 1} \cdots A_{(i+1)l} } = C_i\label{prep eq C_i}
\end{eqnarray}

\begin{eqnarray}
&C_{i+1}^{ A_{il + 1}\cdots A_{(i+1)l} \cdot A_{(i+1)l + 1} \cdots A_{(i+2)l} } = C_{i+1}\label{prep eq C_i+1}
\end{eqnarray}

Now, by definition, $D_i \infty D_{i+1}$ means that

$$ D_i D_{i+1} D_i = D_{i+1} D_i D_{i+1} $$

let us write the last equation as follows

$$ D_i^{D_{i+1}} = D_{i+1}^{D_i^{-1}} $$

and explicitly

\begin{equation}\label{eq D_i}
(A_1\cdots A_{(i-1)l} \cdot C_i \cdot A_{(i+1)l + 1} \cdots A_{kl})^{ A_1\cdots A_{il} \cdot C_{i+1} \cdot A_{(i+2)l + 1} \cdots A_{kl} } =
\end{equation}
$$(A_1\cdots A_{il} \cdot C_{i+1} \cdot A_{(i+2)l + 1} \cdots A_{kl})^{ A^{-1}_{kl}\cdots A^{-1}_{(i+1)l+1} \cdot C^{-1}_i 
\cdot A^{-1}_{(i-1)l} \cdots A^{-1}_1 }.$$

Now since $ D_i = A_1\cdots A_{(i-1)l} \cdot C_i \cdot A_{(i+1)l + 1} \cdots A_{kl} $ and the cycles $A_j$ in this
product are disjoint and $C_i$ commutes with all of these $A_j$ (according to \eqref{eq prep comm}) we conclude that all
the cycles in this product commute and we can write

$$ D_i^{-1} = A^{-1}_{kl}\cdots A^{-1}_{(i+1)l+1} \cdot C^{-1}_i 
\cdot A^{-1}_{(i-1)l} \cdots A^{-1}_1 =  $$

$$ = A^{-1}_1\cdots A^{-1}_{(i-1)l} \cdot C^{-1}_i 
\cdot A^{-1}_{(i+1)l + 1} \cdots A^{-1}_{kl}. $$

Hence we can rewrite \eqref{eq D_i} as

$$(A_1\cdots A_{(i-1)l} \cdot C_i \cdot A_{(i+1)l + 1} \cdots A_{kl})^{ A_1\cdots A_{il} \cdot C_{i+1} \cdot A_{(i+2)l + 1} \cdots A_{kl} } =$$
$$(A_1\cdots A_{il} \cdot C_{i+1} \cdot A_{(i+2)l + 1} \cdots A_{kl})^{ A^{-1}_1\cdots A^{-1}_{(i-1)l} \cdot C^{-1}_i 
\cdot A^{-1}_{(i+1)l + 1} \cdots A^{-1}_{kl} }.$$

Now using the given equations~\ref{prep eq AC} and \ref{prep eq A} we can rewrite the last equation as follows
(we denote $\psi_l(\sigma_i)$ as $\widehat \sigma_i$ for short)

$$ A_1\cdots A_{(i-1)l} C_i^{A_{(i-1)l + 1}\cdots A_{il} \cdot C_{i+1} \cdot A_{(i+2)l + 1} \cdots A_{kl}} \cdot $$
$$ \cdot A_{\widehat \sigma_{i+1}((i+1)l + 1)}\cdots A_{\widehat \sigma_{i+1}((i+2)l)} A_{(i+2)l+1}\cdots A_{kl} = $$

$$ A_1\cdots A_{(i-1)l} A_{\widehat \sigma^{-1}_i((i-1)l + 1)}\cdots A_{\widehat \sigma^{-1}_i(il)} C_{i+1}^{C^{-1}_i \cdot A^{-1}_{(i+1)l + 1} \cdots A^{-1}_{kl}} A_{(i+2)l+1}\cdots A_{kl} $$

which can be written after cancellations as

\begin{equation}\label{eq cancel}
C_i^{A_{(i-1)l + 1}\cdots A_{il} \cdot C_{i+1} \cdot A_{(i+2)l + 1} \cdots A_{kl}} A_{\widehat \sigma_{i+1}((i+1)l + 1)}\cdots A_{\widehat \sigma_{i+1}((i+2)l)}  = 
\end{equation}
$$  A_{\widehat \sigma^{-1}_i((i-1)l + 1)}\cdots A_{\widehat \sigma^{-1}_i(il)} C_{i+1}^{C^{-1}_i \cdot A^{-1}_{(i+1)l + 1} \cdots A^{-1}_{kl}}  $$

now since

$$ \widehat \sigma^{-1}_i( \{ (i-1)l + 1,\ldots,il \} ) = \{ il + 1,\ldots, (i+1)l \} = $$ 
$$ = \widehat \sigma_{i+1}(\{ (i+1)l + 1,\ldots,(i+2)l  \})$$

and by the condition~\ref{prep eq A}, we have that

$$ A_{\widehat \sigma_{i+1}((i+1)l + 1)}\cdots A_{\widehat \sigma_{i+1}((i+2)l)} = A_{il+1}\cdots A_{(i+1)l}$$

and 

$$ A_{\widehat \sigma^{-1}_i((i-1)l + 1)}\cdots A_{\widehat \sigma^{-1}_i(il)} = A_{il+1}\cdots A_{(i+1)l} $$

and so we can rewrite \eqref{eq cancel} as

\begin{equation}\label{eq cancel'}
C_i^{A_{(i-1)l + 1}\cdots A_{il} \cdot C_{i+1} \cdot A_{(i+2)l + 1} \cdots A_{kl}}  A_{il+1}\cdots A_{(i+1)l} =
\end{equation}
$$  A_{il+1}\cdots A_{(i+1)l} C_{i+1}^{C^{-1}_i \cdot A^{-1}_{(i+1)l + 1} \cdots A^{-1}_{kl}}  $$

or equivalently as

$$  C_i^{A_{(i-1)l + 1}\cdots A_{il} \cdot C_{i+1} \cdot A_{(i+2)l + 1} \cdots A_{kl} A_{il + 1}\cdots A_{(i+1)l}}  = $$
$$ =  C_{i+1}^{C^{-1}_i \cdot A^{-1}_{(i+1)l + 1} \cdots A^{-1}_{(i+2)l}} $$

which is equivalent to

$$  C_i^{A_{(i-1)l + 1}\cdots A_{il} \cdot C_{i+1} \cdot A_{(i+2)l + 1} \cdots A_{kl} A_{il + 1}\cdots A_{(i+1)l} 
\cdot A_{(i+2)l} \cdots A_{(i+1)l + 1} \cdot C_i  }  = C_{i+1}$$

which can be rewritten, using \ref{eq prep comm} and \ref{prep eq A}, as

$$  C_i^{A_{(i-1)l + 1}\cdots A_{il} \cdot C_{i+1} \cdot A_{il + 1}\cdots A_{(i+1)l} A_{(i+2)l + 1} \cdots A_{kl}   
\cdot C_i A_{(i+1)l + 1} \cdots A_{(i+2)l}}  = C_{i+1}$$

now according to \eqref{prep eq C_i}, we can rewrite the last equation as

$$  C_i^{ (A_{il + 1}\cdots A_{(i+1)l})^{-1} \cdot C_{i+1} \cdot (A_{il + 1}\cdots A_{(i+1)l}) A_{(i+2)l + 1} \cdots A_{kl}   
\cdot C_i A_{(i+1)l + 1} \cdots A_{(i+2)l}}  = C_{i+1} $$

and according to \eqref{prep eq C_i+1} we can rewrite the last equation as

$$  C_i^{ (A_{(i+1)l + 1}\cdots A_{(i+2)l}) \cdot C_{i+1} \cdot (A_{(i+1)l + 1}\cdots A_{(i+2)l})^{-1} A_{(i+2)l + 1} \cdots A_{kl}   
\cdot C_i A_{(i+1)l + 1} \cdots A_{(i+2)l}}  = C_{i+1} $$

and using \ref{eq prep comm} and \ref{prep eq A} again we get that this is equivalent to

$$ C_i^{C_{i+1}C_i} = C_{i+1} $$

\end{proof}

\begin{lemma}\label{lemma condition C_i}
Let $l\in \mathbb N$, $\psi_l\colon B_k\to S_{kl}$ be the model homomorphism defined for each $i=1,\ldots,k-1$ 
as follows (see Definition~\ref{Def B_k S_mk} and Notations~\ref{notation A} and ~\ref{notation C})

$$ \psi_l(\sigma_i) = (a^i_0, a^{i+1}_0, a^i_1, a^{i+1}_1,\ldots,a^i_{l-1},a^{i+1}_{l-1}) $$

where

$$ a^i_j = 1 + l(i-1) + j \in supp(A^l_i) \quad j\in \bigtriangledown_l $$

Let us denote

\begin{equation}\label{notation delta^i}
\Delta^i = \bigcup_{j=0,\ldots,l-1} supp(A_{(i-1)l + 1 + j}).
\end{equation}

Then for each $i = 1,\ldots,k-2$, $\ n\in \mathbb N$ and $t^i_{l(i-1) + 1},\ldots,t^i_{l(i+1)},t^{i+1}_{li+1},\ldots,t^{i+1}_{l(i+2)} \in \bigtriangledown_n$ we have that 

$$ C^{\psi_l(\sigma_i),n}_{t^i_{l(i-1) + 1},\ldots,t^i_{l(i+1)} } \infty C^{\psi_l(\sigma_{i+1}),n}_{t^{i+1}_{li + 1},\ldots,t^{i+1}_{l(i+2)} } $$

if and only if 

$$ (C^{\psi_l(\sigma_i),n}_{t^i_{l(i-1) + 1},\ldots,t^i_{l(i+1)} })^2\big|_{\Delta^{i+1}} = 
(C^{\psi_l(\sigma_{i+1}),n}_{t^{i+1}_{li + 1},\ldots,t^{i+1}_{l(i+2)} })^2\big|_{\Delta^{i+1}} $$

if and only if

$$\forall j\in \bigtriangledown_l \quad t^i_{(i-1)l + 1 + |1+j|_l} + t^i_{il + 1 + |1+ j|_l} \equiv
t^{i+1}_{(i+1)l + 1 + j} + t^{i+1}_{il + 1 + |1+ j|_l} (\text{mod} \ n)  $$

(The $n$ in this lemma is $\frac{m}{l}$ appearing in \eqref{eq condition t} in Definition~\ref{Def B_k S_mk})

\end{lemma}

\begin{proof}

Let us denote $C^{\psi_l(\sigma_i),n}_{t^i_{l(i-1) + 1},\ldots,t^i_{l(i+1)} } \in S_{lnk}$ as $C_i$ for short.

Recall that we consider each element of $S_{lnk}$ as a bijection from the following set to itself 

$$\Delta_{lnk} = \{ 1,\ldots,lnk \} = \bigcup_{1,\ldots,lk} supp(A_i)$$

We denote

$$ \Delta^i = \bigcup_{j=0,\ldots,l-1} supp(A_{(i-1)l + 1 + j}) $$

Using the notation in \eqref{notation delta^i}, we can write

$$ supp(C_i) = \Delta^i \cup \Delta^{i+1} $$

and so

\begin{eqnarray}
&C_i|_{\Delta_{lnk} - (\Delta^i \cup \Delta^{i+1})} = id\label{eq fix i}
\end{eqnarray}

and 

\begin{eqnarray}
&C_{i+1}|_{\Delta_{lnk} - (\Delta^{i+1} \cup \Delta^{i+2})} = id\label{eq fix i+1}
\end{eqnarray}

we can also see that $C_i$ maps $\Delta^i$ bijectively to $\Delta^{i+1}$ and vice versa, i.e.

\begin{eqnarray}
C_i \colon \Delta^i \stackrel{\cong}{\longrightarrow } \Delta^{i+1} \stackrel{\cong}{\longrightarrow } \Delta^i\label{eq bij i}
\end{eqnarray}

and

\begin{eqnarray}
C_{i+1} \colon \Delta^{i+1} \stackrel{\cong}{\longrightarrow } \Delta^{i+2} \stackrel{\cong}{\longrightarrow } \Delta^{i+1}\label{eq bij i+1}
\end{eqnarray}

Consider the following permutation in $S_{lnk}$

$$ D_i = C_i C_{i+1} C_i C^{-1}_{i+1} C^{-1}_i C^{-1}_{i+1} $$

We shall show first that $D_i$ is trivial (as a permutation in $S_{lnk}$) iff $C^2_i = C^2_{i+1}$. According to equations~\ref{eq fix i}
and ~\ref{eq fix i+1} it is enough to show this for the restriction of $D_i$ to $\Delta^i \cup \Delta^{i+1} \cup \Delta^{i+2} $.

Consider first the following restriction

$$ D_i|_{\Delta^i} = C_i C_{i+1} C_i C^{-1}_{i+1} C^{-1}_i C^{-1}_{i+1}|_{\Delta^i} $$

because of equation~\ref{eq fix i+1} we have

$$ C_i C_{i+1} C_i C^{-1}_{i+1} C^{-1}_i C^{-1}_{i+1}|_{\Delta^i} = C_i C_{i+1} C_i C^{-1}_{i+1} C^{-1}_i|_{\Delta^i} $$

and by the diagrams~\ref{eq bij i} and \ref{eq bij i+1} we have that

$$  C^{-1}_{i+1} C^{-1}_i(\Delta^i) = \Delta^{i+2} $$

and by equation~\ref{eq fix i} we have that $C_i^{-1}|_{\Delta^{i+2}} = id$ and so

$$ C_i C_{i+1} C_i (C^{-1}_{i+1} C^{-1}_i|_{\Delta^i}) = C_i C_{i+1} C^{-1}_{i+1} C^{-1}_i|_{\Delta^i} = id $$

hence we conclude that $D_i|_{\Delta^i} = id$.

Now consider

$$ D_i|_{\Delta^{i+2}} = C_i C_{i+1} C_i C^{-1}_{i+1} C^{-1}_i C^{-1}_{i+1}|_{\Delta^{i+2}} $$

by the diagrams~\ref{eq bij i} and \ref{eq bij i+1} we have that

$$  C^{-1}_i C^{-1}_{i+1}(\Delta^{i+2}) = \Delta^i $$

by equation~\ref{eq fix i+1} we have that $C_{i+1}|_{\Delta^i} = id$ and so

$$ C_i C_{i+1} C_i C^{-1}_{i+1} (C^{-1}_i C^{-1}_{i+1}|_{\Delta^{i+2}}) = C_i C_{i+1} C_i C^{-1}_i C^{-1}_{i+1}|_{\Delta^{i+2}} = C_i|_{\Delta^{i+2}} = id $$

the last equation implied from equation~\ref{eq fix i}. Hence we conclude that $D_i|_{\Delta^{i+2}} = id$.

Now consider

$$ D_i|_{\Delta^{i+1}} = C_i C_{i+1} C_i C^{-1}_{i+1} C^{-1}_i C^{-1}_{i+1}|_{\Delta^{i+1}} $$

by the diagram~\ref{eq bij i+1} we have that

$$  C^{-1}_{i+1}(\Delta^{i+1}) = \Delta^{i+2} $$

and by equation~\ref{eq fix i} we have that $C^{-1}_i|_{\Delta^{i+2}}=id$ and so

$$ C_i C_{i+1} C_i C^{-1}_{i+1} C^{-1}_i (C^{-1}_{i+1}|_{\Delta^{i+1}}) = C_i C_{i+1} C_i C^{-1}_{i+1} C^{-1}_{i+1}|_{\Delta^{i+1}} $$

now by the diagrams~\ref{eq bij i} and \ref{eq bij i+1} we have that

$$  C_i C^{-1}_{i+1} C^{-1}_{i+1}(\Delta^{i+1}) = \Delta^i $$

and by equation~\ref{eq fix i+1} we have that $C_{i+1}|_{\Delta^i}=id$ and so

$$ C_i C_{i+1} (C_i C^{-1}_{i+1} C^{-1}_{i+1}|_{\Delta^{i+1}}) = C_i  C_i C^{-1}_{i+1} C^{-1}_{i+1}|_{\Delta^{i+1}}$$

Hence we conclude that $supp(D_i) = \Delta^{i+1}$ and that $D_i|_{\Delta^{i+1}} = C^2_i C^{-2}_{i+1}|_{\Delta^{i+1}}$. 

Since $D_i = C_i C_{i+1} C_i C^{-1}_{i+1} C^{-1}_i C^{-1}_{i+1} = id$ iff $C_i \infty C_{i+1}$, we have that

$$ C_i \infty C_{i+1} \quad \text{iff} \quad C^2_i|_{\Delta^{i+1}} = C^2_{i+1}|_{\Delta^{i+1}} $$

Now let us see what this means in terms of the indices $t_{\ast }$. Let us write explicitly

\begin{eqnarray}
(C^{\psi_l(\sigma_i),n}_{t^i_{l(i-1) + 1},\ldots,t^i_{l(i+1)} })^2\big|_{\Delta^{i+1}} = 
(C^{\psi_l(\sigma_{i+1}),n}_{t^{i+1}_{li + 1},\ldots,t^{i+1}_{l(i+2)} })^2\big|_{\Delta^{i+1}}\label{eq squares}
\end{eqnarray}

Let us denote $\psi_l(\sigma_i)$ as $\widehat \sigma_i$ for short. We have by definition that

$$ C_i(a^r_s) = a^{\widehat \sigma_i(r)}_{ |s + t^i_{\widehat \sigma_i(r)} |_n } $$

where (see notation~\ref{notation A})

$$ a^r_s = 1 + n(r-1) + s \in supp(A^n_r),\quad s\in \bigtriangledown_n $$

since we have to check equation~\ref{eq squares} only on the set $\Delta^{i+1}$ we have that it is equivalent to

$$ a^{\widehat \sigma^2_i(r)}_{ |s + t^i_{\widehat \sigma_i(r)} + t^i_{\widehat \sigma^2_i(r)}  |_n } = a^{\widehat \sigma^2_{i+1}(r)}_{ |s + t^{i+1}_{\widehat \sigma_{i+1}(r)} + t^{i+1}_{\widehat \sigma^2_{i+1}(r)}  |_n } $$

for $r=il+1,\ldots,(i+1)l$ and each $s\in \bigtriangledown_n$. But for each $j\in \bigtriangledown_l$ we have

$$ r=il+1+j \stackrel{\widehat \sigma_i }{\longrightarrow } (i-1)l + 1+|1+j|_l \stackrel{\widehat \sigma_i }{\longrightarrow } il+1+|1+j|_l $$

$$ r=il+1+j \stackrel{\widehat \sigma_{i+1} }{\longrightarrow } (i+1)l + 1 + j \stackrel{\widehat \sigma_{i+1} }{\longrightarrow } il+1+|1+j|_l $$

hence we conclude that for each $r=il+1,\ldots,(i+1)l$ we have that $\widehat \sigma^2_i(r) = \widehat \sigma^2_{i+1}(r)$ and that
equation~\ref{eq squares} holds iff
$t^i_{\widehat \sigma_i(r)}  + t^i_{\widehat \sigma^2_i(r)} \equiv
t^{i+1}_{\widehat \sigma_{i+1}(r)} + t^{i+1}_{\widehat \sigma^2_{i+1}(r)} (\text{mod} \ n) $ 
for each $r=il+1,\ldots,(i+1)l$ which translates to

$$ \forall j\in \bigtriangledown_l \quad t^i_{(i-1)l + 1 + |1+j|_l} + t^i_{il + 1 + |1+ j|_l} \equiv   
t^{i+1}_{(i+1)l + 1 + j} + t^{i+1}_{il + 1 + |1+ j|_l} (\text{mod} \ n)$$

for all $i=1,\ldots,k-2$.

\end{proof}

We are now ready to prove 

\begin{proposition}
Let $m\in \mathbb N$. The model homomorphisms $\psi,\phi\colon B_k \to S_{mk}$ (see Definition~\ref{Def B_k S_mk}) 
are well-defined non-cyclic transitive homomorphisms 
\end{proposition}
\begin{proof}

That $\psi_m$ is a well defined non-cyclic transitive homomorphism is proved in Lemma~\ref{lemma psi}.

Now let $l$ be a divisor of $m$ where $1\leq l < m$  and let 

$$ \{ t^i_{(i-1)l + 1},\ldots,t^i_{(i+1)l} \}_{i=1,\ldots,k-1} \in \bigtriangledown_{\frac{m}{l}} $$

% which satisfy the following condition

% $$\forall j\in \bigtriangledown_l \quad t^{i+1}_{(i+1)l + 1 + j} + t^{i+1}_{il + 1 + |1+ j|_l} \equiv t^i_{(i-1)l + 1 + |1+j|_l} + t^i_{il + 1 + |1+ j|_l} (\text{mod} \ \frac{m}{l})  $$

(when $l=1$ we take  $\bigtriangledown_1 = \{ 0 \}$ )

then for each $i=1,\ldots,k-1$ we define

$$ \phi_{m,l,\{ t^i_{(i-1)l + 1},\ldots,t^i_{(i+1)l } \}_{i=1,\ldots,k-1} }(\sigma_i) = A^{\frac{m}{l}}_1 \cdots A^{\frac{m}{l}}_{(i-1)l} \cdot 
C^{\psi_l(\sigma_i),\frac{m}{l}}_{t^i_{l(i-1)+1},\ldots,t^i_{l(i+1)} } \cdot A^{\frac{m}{l}}_{(i+1)l+1} \cdots A^{\frac{m}{l}}_{kl} $$

Let us denote $A^{\frac{m}{l}}_{\ast}$ as $A_{\ast}$, $C^{\psi_l(\sigma_i),\frac{m}{l}}_{t^i_{l(i-1)+1},\ldots,t^i_{l(i+1)}}$ as $C_i$ 
and $\phi_{m,l,\{ t^i_{(i-1)l + 1},\ldots,t^i_{(i+1)l } \}_{i=1,\ldots,k-1} }(\sigma_i)$ as $\widehat \sigma_i $ for short.
We have by definition of $C_i$ that

\begin{eqnarray}
&A_r^{C_i} =
\begin{cases}
A_r & r = 1,\ldots,(i-1)l \quad  \text{or} \quad r=(i+1)l + 1,\ldots,kl\\
A_{\psi_l(\sigma_i)(r)} & r = (i-1)l + 1,\ldots, (i+1)l\label{prep eq AC'}\\
\end{cases}
\end{eqnarray}

and clearly

\begin{eqnarray}
&[A_r,A_s] = 1 \quad r,s=1,\ldots,kl\label{prep eq A'}
\end{eqnarray}

According to Lemma~\ref{lemma prep phi} we have that $ \widehat \sigma_i \infty \widehat \sigma_{i+1} $ iff $C_i \infty C_{i+1}$ iff

$$\forall j\in \bigtriangledown_l \quad t^i_{(i-1)l + 1 + |1+j|_l} + t^i_{il + 1 + |1+ j|_l} \equiv
t^{i+1}_{(i+1)l + 1 + j} + t^{i+1}_{il + 1 + |1+ j|_l} (\text{mod} \ \frac{m}{l})  $$

which is exactly the condition stated in the definition of $\phi$ (see Definition~\ref{Def B_k S_mk}). Now suppose $|i-j| \geq 2$
where $1\leq i < j \leq k-1$. Then since $j-i\geq 2$ we have that $i+1 \leq j-1$ and we can write

$$ \widehat \sigma_j = A_1\cdots A_{(i-1)l} ( A_{(i-1)l+1}\cdots A_{(i+1)l}) A_{(i+1)l+1}\cdots A_{(j-1)l} C_j A_{(j+1)l+1}\cdots A_{kl} $$

Now $[C_i,C_j]=1$ since $supp(C_i)\cap supp(C_j) = \emptyset$, $[C_i, A_r]=1$ for $r=1,\ldots,(i-1)l$ and $r=(i+1)l+1,\ldots,kl$ according
to equation~\ref{prep eq AC'}, and finally we have that the equations~\ref{prep eq AC'} and \ref{prep eq A'} imply (see the proof of Lemma~\ref{lemma prep phi})

$$ [C_i, A_{(i-1)l + 1}\cdots  A_{(i+1)l}] = 1 $$

and this means that $[\widehat \sigma_i,\widehat \sigma_j]=1$ for $|i-j| > 1$. Hence we conclude that $\phi$ is indeed a homomorphism.

$\phi$ is non-cyclic since $\widehat \sigma_i \neq \widehat \sigma_{i+1}$ for $i\neq j$.

To see that $\phi$ is transitive, consider the orbit of $1\in \Delta_{mk}$ in $Im(\phi)$. We see from $\widehat \sigma_2$ that
$1\in supp(A_1)  $ hence $supp(A_1) \subseteq orbit(1)$. Now since the support of each cycle in the cycle decomposition of $C_1$
contains an element of $supp(A_1)$ we have, by considering $\widehat \sigma_1$, that $supp(C_1) \subseteq orbit(1)$.
Now since $supp(C_i) \cap supp(C_{i+1}) = supp(A_{il+1}) \cup \cdots \cup supp(A_{(i+1)l}) $ and the support
of each cycle in the cycle decomposition of $C_{i+1}$ contains an element of $supp(A_{il+1}) \cup \cdots \cup supp(A_{(i+1)l})$ 
for each $i=1,\ldots,k-1$, we can see (by induction on $i$) that 

$$ \bigcup_{i=1,\ldots,k-1} supp(C_i)  \subseteq orbit(1) $$

but since $supp(C_i) = supp(A_{(i-1)l+1}) \cup \cdots \cup supp(A_{(i+1)l}) $ this means that 

$$ \bigcup_{i=1,\ldots,k-1} supp(A_{(i-1)l+1}) \cup \cdots \cup supp(A_{(i+1)l}) = $$

$$ = \bigcup_{i=1,\ldots,kl} A_i = \Delta_{mk} \subseteq orbit(1) $$

and hence $\phi$ is transitive.

\end{proof}

\subsection{Good Permutation Representations}\label{subsec good perm rep}

In this subsection we define and study a particular property of permutation representations of the braid group.
This property is common to almost all known permutation representations of the braid group (except for a 
finite number of exceptions, i.e., homomorphisms $B_k\to S_{mk}$ for small values of $k$ and $m$ - see section~$4$
in \cite{Lin04b} for these examples). We define this property in the next definition and explain the
motivation for studying it the remark that follows. 

\begin{definition}\label{def good}
We call a homomorphism $\omega \colon B_k\to S_n$ \textbf{good} if one of the following holds

\begin{enumerate}
\item $supp(\widehat \sigma_i) \cap supp(\widehat \sigma_j) = \emptyset$ for $1\leq i,j \leq k-1$ such that $|i-j| > 1$.\\
\item $supp(\widehat \sigma_1) = \cdots = supp(\widehat \sigma_{k-1})$
\end{enumerate}

\hfill $\bullet$
\end{definition}

\begin{rem}
The motivation for defining good homomorphisms in Definition~\ref{def good} is our conjecture that 
in fact all transitive non-cyclic permutation representations of the braid group are good (see Conjecture~\ref{conj good}).
Furthermore, it is our conjecture that all good transitive non-cyclic
permutations representations of the braid group are either standard or derived in some sense from standard homomorphisms. We make these
conjectures precise in Section~\ref{section conjecture}. Hence, we dedicate the rest of the section to the investigation
of good permutation representations of the braid group. \hfill $\bullet$
\end{rem}

\begin{lemma}\label{lemma intersect}
Let $\omega,\omega' \colon B_k\to S_n$ be conjugate homomorphisms, then

$$ | supp(\omega(\sigma_1)) \cap supp(\omega(\sigma_2)) | = | supp(\omega'(\sigma_1)) \cap supp(\omega'(\sigma_2)) | $$ 

\end{lemma}

\begin{proof}
Suppose $\omega' = \omega^{\zeta}$ for some $\zeta \in S_n$. Let 
$j\in supp(\omega(\sigma_1)) \cap supp(\omega(\sigma_2)) $  then 

$$ \zeta(j) \in supp(\omega^{\zeta}(\sigma_1)) \cap supp(\omega^{\zeta}(\sigma_2)) = $$
 
$$ = supp(\omega'(\sigma_1)) \cap supp(\omega'(\sigma_2))$$

and similarly, if $\zeta(j) \in  supp(\omega'(\sigma_1)) \cap supp(\omega'(\sigma_2))$ then by applying $\zeta^{-1}$ on $\zeta(j)$ we get that

$$ j\in supp(\omega(\sigma_1)) \cap supp(\omega(\sigma_2))  $$

Hence $\zeta$ is a bijection between the sets $supp(\omega(\sigma_1)) \cap supp(\omega(\sigma_2))$ and 
$supp(\omega'(\sigma_1)) \cap supp(\omega'(\sigma_2))$. 

\end{proof}

\begin{definition}
For a homomorphism $\omega \colon B_k\to S_n$ define 

$$\mathbf{intersect(\omega) = | supp(\widehat \sigma_1) \cap supp(\widehat \sigma_2) |}$$

and

$$\mathbf{ supp(\omega) = |supp(\widehat \sigma_1)| }$$

\hfill $\bullet$
\end{definition}

\begin{rem}\label{remark lemma intersect}
Lemma~\ref{lemma intersect} shows that for any homomorphism $\omega \colon B_k\to S_n$, $intersect(\omega)$ is 
invariant under conjugation. In particular, since $\alpha \sigma_i \alpha^{-1} = \sigma_{i+1}$ 
(where as above $\alpha = \sigma_1\sigma_2\cdots \sigma_{k-1}$) then by conjugating $\omega$ with
$\omega(\alpha)$ we have that for any $i$ ( where $1\leq i \leq k-2$ )

$$ intersect(\omega) =  | supp(\widehat \sigma_i) \cap supp(\widehat \sigma_{i+1}) | $$

Similarly, since $x \mapsto \alpha(x)$ is a bijection between $supp(\widehat \sigma_1)$ and $supp(\widehat \sigma_1^{\zeta})$
for any $\zeta \in S_n$, we see that $supp(\omega)$ is also invariant under conjugation and in particular we get that

$$ supp(\omega) = |supp(\widehat \sigma_i)| $$

for each $i=1,\ldots,k-1$ 

\hfill $\bullet$
\end{rem}

We also prove

\begin{lemma}\label{lemma supp bound}

Let $\omega\colon B_k \to S_n$ be a homomorphism and let $x_1\in supp(\widehat \sigma_i)$ for some $i$, $1\leq i \leq k-1$.
If $x_1\notin supp(\widehat \sigma_{i\pm 1})$ then $\widehat \sigma_i(x_1)\in supp(\widehat \sigma_{i\pm 1})$
(for $i=1$ and $i=k-1$ we take $\widehat \sigma_{1\pm 1}$ and $\widehat \sigma_{(k-1)\pm 1}$ to mean $\widehat \sigma_2$ and $\widehat \sigma_{k-2}$
respectively)

\end{lemma}
\begin{proof}

\vskip 0.2cm

Let us denote $x_2 = \widehat \sigma_i(x_1)$. Assume by way of contradiction that $x_2\notin supp(\widehat \sigma_{i\pm 1})$. Let us write

$$ D = (x_1, x_2, \ldots) $$

where $D$ is a cycle in the cycle decomposition of $\widehat \sigma_i$. Now since 
$\widehat \sigma_i = \widehat \sigma_{i\pm 1}^{\widehat \sigma_i \widehat \sigma_{i\pm 1} } $, we have that
$\widehat \sigma_i^{\widehat \sigma^{-1}_{i\pm 1} \widehat \sigma^{-1}_i}=\widehat \sigma_{i\pm 1}$ and we expect that

$$ D^{\widehat \sigma^{-1}_{i\pm 1} \widehat \sigma^{-1}_i} = (\widehat \sigma_i \widehat \sigma_{i\pm 1} (x_1), 
\widehat \sigma_i \widehat \sigma_{i\pm 1} (x_2),\ldots ) = (x_2,\widehat \sigma_i(x_2),\ldots) $$

be a cycle in the cycle decomposition of $\widehat \sigma_{i\pm 1}$, i.e., that $x_2 \in supp(\widehat \sigma_{i\pm 1})$ which
is contrary to our assumption.

\end{proof}

As a consequence we get

\begin{corollary}\label{corollary lemma good}
Let $\omega\colon B_k \to S_n$ be a homomorphism. Then for every $i=1,\ldots,k-1$ and for every cycle $D$ in 
the cycle decomposition of $\widehat \sigma_i$ we have that 

$$ |supp(D) \cap supp(\widehat \sigma_{i\pm 1})| \geq \frac12 |supp(D)| $$

and so

$$intersect(\omega) \geq \frac12 supp(\omega)$$

\end{corollary}

\begin{proof}

Let $D = (x_0,\ldots,x_{l-1})$ be a cycle in the cycle decomposition of $\widehat \sigma_i$. 
Consider $(x_0,\ldots,x_{l-1})$ as an ordered tuple of elements in $\Delta_n$, then according to Lemma~\ref{lemma supp bound}
for each $j=0,\ldots,l-1$ we have that at least one of $x_j$ and $x_{|j+1|_l}$ are in $supp(\widehat \sigma_{i\pm 1})$ which means
that $|supp(D) \cap supp(\widehat \sigma_{i\pm 1})| \geq \frac12 |supp(D)|$. Since this holds for every cycle in the 
cycle decomposition of $\widehat \sigma_1$, say $\widehat \sigma_1 = D_1\cdots D_r$, we conclude that

$$intersect(\omega) = \sum_{j=1,\ldots,r} |supp(D_j)\cap supp(\widehat \sigma_{i\pm 1})| \geq $$
$$ \geq  \sum_{j=1,\ldots,r} \frac12 |supp(D_j)| = \frac12 supp(\omega) $$

\end{proof}

Let us now focus on good transitive homomorphisms $\omega \colon B_k\to S_{mk}$. We prove

\begin{lemma}\label{lemma good}

Let $\omega \colon B_k\to S_{mk}$ be a good transitive homomorphism where $k\geq 4$, then exactly one of the following holds~: 

\begin{enumerate}
\item $supp(\omega) = 2m\quad $ and $\quad intersect(\omega) = m$\\
\item $supp(\omega) = mk\quad $ and $\quad intersect(\omega) = mk$
\end{enumerate}
\end{lemma}

\begin{proof}

Suppose $supp(\widehat \sigma_1) = \cdots = supp(\widehat \sigma_{k-1})$, then since $\omega$ is transitive
we necessarily have that $supp(\widehat \sigma_1) = \cdots = supp(\widehat \sigma_{k-1}) = \Delta_{mk} $
but then also $supp(\widehat \sigma_i) \cap supp(\widehat \sigma_{i+1}) = \Delta_{mk}$ 
for $i=1,\ldots,k-2$ and so $supp(\omega) = mk\ $ and $\ intersect(\omega) = mk$.

Assume then that $supp(\widehat \sigma_i) \cap supp(\widehat \sigma_j) = \emptyset$ for $1\leq i,j \leq k-1$ such that $|i-j| > 1$
and we will prove that $supp(\omega) = 2m\ $ and $\ intersect(\omega) = m$.

Since $\omega$ is transitive we must have that

\begin{equation}
\bigcup_{i=1,\ldots,k-1} supp(\widehat \sigma_i) = \Delta_{mk}\label{equation sets mk}
\end{equation}

We shall prove, by induction on $k\geq 4$, that

\begin{equation}\label{equation supp intersect sets}
\left| \bigcup_{i=1,\ldots,k-1} supp(\widehat \sigma_i) \right| = (k-1)supp(\omega) - (k-2)intersect(\omega)
\end{equation}

If $k=4$ then $\omega\colon B_4 = < \sigma_1,\sigma_2,\sigma_3 > \to S_{4m}$ and since, by assumption we have 

$$ supp(\widehat \sigma_1) \cap supp(\widehat \sigma_3) = \emptyset $$

and according to remark~\ref{remark lemma intersect} we have

$$ |supp(\widehat \sigma_1) \cap supp(\widehat \sigma_2)| = |supp(\widehat \sigma_2) \cap supp(\widehat \sigma_3)| = intersect(\omega)  $$

all in all this means that

$$ |supp(\widehat \sigma_1) \cup supp(\widehat \sigma_2) \cup supp(\widehat \sigma_3)| = $$

$$ = |supp(\widehat \sigma_1)| + |supp(\widehat \sigma_2)| + |supp(\widehat \sigma_3)| - |supp(\widehat \sigma_1) \cap supp(\widehat \sigma_2)| - $$

$$ - |supp(\widehat \sigma_2) \cap supp(\widehat \sigma_3)| = (k-1)supp(\omega) - (k-2)intersect(\omega) $$

and equation~\ref{equation supp intersect sets} holds in this case.

Now assume that equation~\ref{equation supp intersect sets} holds for $k-1$. Since by assumption we have that 

$$ supp(\widehat \sigma_{k-1}) \cap (\bigcup_{i=1,\ldots,k-3} supp( \sigma_i )  ) = \emptyset $$  

and so

$$ \left| supp(\widehat \sigma_{k-1}) \cap (\bigcup_{i=1,\ldots,k-2} supp( \sigma_i )  ) \right| =  \left| supp(\widehat \sigma_{k-1}) \cap supp(\widehat \sigma_{k-2}) \right| = intersect(\omega) $$

(see remark~\ref{remark lemma intersect} for the last equality) hence we conclude that

$$ \left| \bigcup_{i=1,\ldots,k-1} supp(\widehat \sigma_i) \right| = \left| supp(\widehat \sigma_{k-1}) \cup (\bigcup_{i=1,\ldots,k-2} supp( \sigma_i )  ) \right| = $$

$$ = |supp(\widehat \sigma_{k-1})| + \left|  \bigcup_{i=1,\ldots,k-2} supp( \sigma_i )  \right| -   
\left| supp(\widehat \sigma_{k-1}) \cap (\bigcup_{i=1,\ldots,k-2} supp( \sigma_i )  ) \right| = $$

$$ = |supp(\widehat \sigma_{k-1})| + \left|  \bigcup_{i=1,\ldots,k-2} supp( \sigma_i )  \right| -  intersect(\omega) = $$

$$ = |supp(\widehat \sigma_{k-1})| + \left( (k-2)supp(\omega) - (k-3)intersect(\omega)\right) - intersect(\omega) =  $$

$$ = (k-1)supp(\omega) - (k-2)intersect(\omega) $$

see remark~\ref{remark lemma intersect} for the last equality.

Now since, by equation~\ref{equation sets mk}, $|\bigcup_{i=1,\ldots,k-1} supp(\widehat \sigma_i)| = mk $, 
then \eqref{equation supp intersect sets} implies that

\begin{equation}\label{equation supp intersect}
(k-1)supp(\omega) - (k-2)intersect(\omega) = mk
\end{equation}

let us rewrite this equation as follows

% $$ (k-1) supp(\omega) = mk + (k-2)intersect(\omega) $$

% or

$$ supp(\omega) = \frac{m(k-1) + m + (k-1)intersect(\omega) - intersect(\omega)}{k-1} = $$

\begin{equation}\label{equation intersect supp}
= m + intersect(\omega) + \frac{ m  - intersect(\omega)}{k-1} 
\end{equation}

but since this equation involves only integers, there exists an $n\in \mathbb Z$ such that

$$ \frac{ m  - intersect(\omega)}{k-1} = -n $$

or

\begin{equation}\label{equation supp intersect'}
intersect(\omega) = m + n(k-1)
\end{equation}

Now by substituting the expression of $intersect(\omega)$ from equation~\ref{equation supp intersect'} 
into equation~\ref{equation intersect supp} we get

$$  supp(\omega) = m + (m + n(k-1)) - n = 2m + n(k-2) $$

and so we have the two equations

\begin{eqnarray}
intersect(\omega) = m + n(k-1)\label{eq int}\\
supp(\omega) = 2m + n(k-2)\label{eq supp}
\end{eqnarray}

Consider now $\widehat \sigma_1, \widehat \sigma_2$ and $\widehat \sigma_3$ (here we
assume that $k \geq 4$). Since, by assumption we have that 

$$ supp(\widehat \sigma_1) \cap supp(\widehat \sigma_3) = \emptyset $$

which implies that

$$ (supp(\widehat \sigma_1) \cap supp(\widehat \sigma_2)) \cap (supp(\widehat \sigma_2) \cap supp(\widehat \sigma_3)) = \emptyset $$

but since $|supp(\widehat \sigma_1) \cap supp(\widehat \sigma_2)| = |supp(\widehat \sigma_2) \cap supp(\widehat \sigma_3)| = intersect(\omega)$
and since $supp(\widehat \sigma_1) \cap supp(\widehat \sigma_2), supp(\widehat \sigma_2) \cap supp(\widehat \sigma_3) \subseteq supp(\widehat \sigma_2)$
we have that

$$ 2(intersect(\omega)) \leq supp(\omega), $$

but according to Corollary~\ref{corollary lemma good}, we also have that

$$ 2(intersect(\omega)) \geq supp(\omega). $$

Hence

$$ 2(intersect(\omega)) = supp(\omega) $$

which, according to equations~\ref{eq int} and~\ref{eq supp}, translates to

$$ 2m + 2n(k-1) = 2m + n(k-2) $$

and this is equivalent to

$$ nk = 0 $$

which means that $n = 0$. 

Hence we conclude from \eqref{eq int} and~\eqref{eq supp} that $intersect(\omega) = m$ and
$supp(\omega) = 2m$.

\end{proof}

\subsection{Conjugacies between Model Homomorphisms}\label{subsec conj model homs}

In this subsection we prove that there are non-trivial conjugacies between the model homomorphisms defined in Definition~\ref{Def B_k S_mk}.
Since we are interested in model homomorphisms up to conjugacy this will lead to a more concise definition of the model homomorphisms.

By definition~\ref{Def B_k S_mk}, for each $m$ there is only one model homomorphism $S_k\to B_{mk}$ which is good of type $1$ 
(see definition~\ref{def good}). Hence, we will study the conjugacies
between the model homomorphisms which are good of type $2$, as defined in Definition~\ref{Def B_k S_mk}.

We shall use the following notations (see notation~\ref{notation A} and notation~\ref{notation C})~:

Fix a natural number $m\in \mathbb N$, $m\geq 2$.

For each $i\geq 1$ and $q\in \bigtriangledown_m $ we denote

$$ a^i_q = 1 + m(i-1) + q \in supp(A^m_i). $$

Now let $\sigma \in \mathbf{S}(\{j,\ldots,j+r-1 \})\cong S_r$ be some permutation and let $t_j,\ldots,t_{j+r-1}\in \bigtriangledown_m$.
We denote by $C^{\sigma,m}_{t_j,\ldots,t_{j+r-1}}$ the permutation given by the formula

\begin{equation}\label{71.2}
C^{\sigma,m}_{t_j,\ldots,t_{j+r-1}}(a^i_p) = a^{\sigma(i)}_{|p + t_{\sigma(i)}|_m} \quad \text{for all} \ p\in \bigtriangledown_m
\end{equation}

for each $i\in \{j,\ldots,j+r-1 \}$.

According to Lemma~\ref{lemma C}, we have that 

%$$ supp(C^{\sigma,m}_{t_j,\ldots,t_{j+r-1}}) \subseteq \bigcup_{i\in supp(\sigma)} supp(A^m_i) $$

\begin{equation}\label{71.1}
(A^m_i)^{( C^{\sigma,m}_{t_j,\ldots,t_{j+r-1}} )^{-1}} = A^m_{\sigma(i)}
\end{equation}

$\newline$

% \begin{rem}\label{rem C^C}
% Note that we allow in \eqref{71.2} that $i\not \in supp(\sigma)$ which is slightly more general
% than lemma~\ref{lemma C} (where we require that $i \in supp(\sigma))$. 
% In this case, if $i\not \in supp(\sigma)$, then $C^{\sigma,m}_{t_j,\ldots,t_{j+r-1}}|_{supp(A^m_i)} = (A^m_i)^{t_i}$

% For example, consider $\sigma = (1,2)\in \mathbf{S}(\{1,2,3\})$ and let $m=4$, then

% $$ C^{(1,2), 4}_{1,0,2} = (1,5,2,6,3,7,4,8)(9,11)(10,12)  $$

% and indeed 

% $$ (A^4_3)^2 = ((9,10,11,12))^2 = (9,11)(10,12). $$

% More generally, if $i\not \in supp(\sigma)$ then 

% $$ (A^m_i)^{(C^{\sigma,m}_{t_j,\ldots,t_{j+r-1}})^{-1}} = (A^m_i)^{(A^m_i)^{t_i}} = A^m_{i} $$

% which means that \eqref{71.1} holds for all $i\in \{j,\ldots,j+r-1 \}$.

% \hfill $\bullet$
% \end{rem}

\begin{lemma}\label{lemma C^C}

Let $m\in \mathbb N$, $m\geq 2$. Let $\sigma,\tau \in \mathbf{S}(\{j,\ldots,j+r-1 \})\cong S_r$ be any two permutations and let 
$t^1_j,\ldots,t^1_{j+r-1}, t^2_j,\ldots,t^2_{j+r-1}\in \bigtriangledown_m$. Then the following formulas hold

\textbf{(a)} 

$$ C^{\sigma,m}_{t^1_j,\ldots,t^1_{j+r-1}} \cdot C^{\tau,m}_{t^2_j,\ldots,t^2_{j+r-1}} = C^{\sigma \cdot \tau,m}_{t^3_j,\ldots,t^3_{j+r-1}} $$

where

\begin{equation}\label{eq sec a'}
t^3_i = |t^2_{\sigma^{-1}(i)} + t^1_i |_m
\end{equation}

for each $i\in \{j,\ldots,j+r-1\}$.

\textbf{(b)}

$$ (C^{\sigma,m}_{t^1_j,\ldots,t^1_{j+r-1}})^{-1} = C^{\sigma^{-1},m}_{t^2_j,\ldots,t^2_{j+r-1}} $$ 

where

$$ t^2_i = |-t^1_{\sigma(i)}|_m $$

for each $i\in \{j,\ldots,j+r-1\}$.

\textbf{(c)}

$$ (C^{\tau,m}_{t^2_j,\ldots,t^2_{j+r-1}})^{-1} \cdot C^{\sigma,m}_{t^1_j,\ldots,t^1_{j+r-1}} \cdot 
C^{\tau,m}_{t^2_j,\ldots,t^2_{j+r-1}} = C^{\tau^{-1} \cdot \sigma \cdot \tau,m}_{t^3_j,\ldots,t^3_{j+r-1}} $$

where

\begin{equation}\label{103.1}
t^3_i = | t^2_{\sigma^{-1}( \tau(i))} + t^1_{\tau(i)} - t^2_{\tau(i)}  |_m 
\end{equation}

for each $i\in \{j,\ldots,j+r-1\}$.

\end{lemma}
\begin{proof}
$\newline$

In this proof we denote, for each $i\in \mathbb N$, $i\geq 1$

$$ A_i = (1 + m(i-1),\ldots, mi) $$

(i.e., $A_i=A_i^m$ from notation~\ref{notation A} but we suppress the superscript).

\textbf{Proof of (a) : }

% First note, according to \eqref{71.1} and remark~\ref{rem C^C}, that for each $i\in \{j,\ldots,j+r-1\}$ we have 

% $$ (A_i)^{C^{\sigma,m}_{t^1_j,\ldots,t^1_{j+r-1}}\cdot C^{\tau,m}_{t^2_j,\ldots,t^2_{j+r-1}}} = 
% (A_{\sigma(i)})^{C^{\tau,m}_{t^2_j,\ldots,t^2_{j+r-1}}} = A_{\tau(\sigma(i))} = A_{\tau\cdot \sigma (i)} $$

% Hence

% $$ C^{\sigma,m}_{t^1_j,\ldots,t^1_{j+r-1}} \cdot C^{\tau,m}_{t^2_j,\ldots,t^2_{j+r-1}} = C^{\tau \cdot \sigma,m}_{t^3_j,\ldots,t^3_{j+r-1}} $$

% for some $t^3_p\in \bigtriangledown_m $ whose values we now compute. 

By \eqref{71.2} we have for each $i\in \{j,\ldots,j+r-1\}$ and $q\in \bigtriangledown_m$

$$ C^{\sigma,m}_{t^1_j,\ldots,t^1_{j+r-1}} \cdot C^{\tau,m}_{t^2_j,\ldots,t^2_{j+r-1}} (a^i_q) = 
C^{\sigma,m}_{t^1_j,\ldots,t^1_{j+r-1}} ( C^{\tau,m}_{t^2_j,\ldots,t^2_{j+r-1}} (a^i_q)) = $$

$$ = C^{\sigma,m}_{t^1_j,\ldots,t^1_{j+r-1}} (a^{\tau(i)}_{|q + t^2_{\tau(i)}|_m}) =
 a^{\sigma(\tau(i))}_{|q + t^2_{\tau(i)} + t^1_{\sigma(\tau(i))} |_m} = 
a^{\sigma\cdot \tau (i)}_{|q + t^2_{\tau(i)} + t^1_{\sigma\cdot \tau (i)} |_m} $$

which means, by \eqref{71.2}, that

$$ C^{\sigma,m}_{t^1_j,\ldots,t^1_{j+r-1}} \cdot C^{\tau,m}_{t^2_j,\ldots,t^2_{j+r-1}}  = 
C^{\sigma \cdot \tau,m}_{t^3_j,\ldots,t^3_{j+r-1}} $$ 

where

$$ t^3_{\sigma\cdot \tau(i)} = |t^2_{\tau(i)} + t^1_{\sigma\cdot \tau (i)} |_m$$

or equivalently

$$ t^3_i = |t^2_{\sigma^{-1}(i)} + t^1_i |_m$$

\textbf{Proof of (b) : }

For each $i\in \{j,\ldots,j+r-1\}$ set $t^2_i = |-t^1_{\sigma(i)}|_m $.

According to part $(a)$ we have

$$ C^{\sigma,m}_{t^1_j,\ldots,t^1_{j+r-1}} \cdot C^{\sigma^{-1},m}_{t^2_j,\ldots,t^2_{j+r-1}} = C^{id,m}_{t^3_j,\ldots,t^3_{j+r-1}} $$

where

$$ t^3_i = | t^2_{\sigma^{-1}(i)} + t^1_i |_m = | |-t^1_i|_m + t^1_i |_m = 0 $$

Hence 

$$ C^{id,m}_{t^3_j,\ldots,t^3_{j+r-1}} = C^{id,m}_{0,\ldots,0} = id $$

and so

$$ (C^{\sigma,m}_{t^1_j,\ldots,t^1_{j+r-1}})^{-1} = C^{\sigma^{-1},m}_{t^2_j,\ldots,t^2_{j+r-1}} $$

where $t^2_i = |-t^1_{\sigma(i)}|_m $.

\textbf{Proof of (c) : }

By part $(b)$ we have that

$$ (C^{\tau,m}_{t^2_j,\ldots,t^2_{j+r-1}})^{-1} = C^{\tau^{-1},m}_{|-t^2_{\tau(j)}|_m,\ldots,|-t^2_{\tau(j+r-1)}|_m} $$ 

and by part $(a)$ and \eqref{eq sec a'} we have that

$$ C^{\sigma,m}_{t^1_j,\ldots,t^1_{j+r-1}} \cdot C^{\tau,m}_{t^2_j,\ldots,t^2_{j+r-1}} = 
C^{\sigma \cdot \tau,m}_{|t^2_{\sigma^{-1}(j)} + t^1_j |_m,\ldots,|t^2_{\sigma^{-1}(j+r-1)} + t^1_{j+r-1} |_m}. $$

Using part $(a)$ and \eqref{eq sec a'} again we now have

$$ (C^{\tau,m}_{t^2_j,\ldots,t^2_{j+r-1}})^{-1} \cdot C^{\sigma,m}_{t^1_j,\ldots,t^1_{j+r-1}} \cdot 
C^{\tau,m}_{t^2_j,\ldots,t^2_{j+r-1}} = $$

$$ = C^{\tau^{-1},m}_{|-t^2_{\tau(j)}|_m,\ldots,|-t^2_{\tau(j+r-1)}|_m} \cdot 
C^{\sigma \cdot \tau,m}_{|t^2_{\sigma^{-1}(j)} + t^1_j |_m,\ldots,|t^2_{\sigma^{-1}(j+r-1)} + t^1_{j+r-1} |_m} = $$

$$ = C^{\tau^{-1} \cdot \sigma \cdot \tau,m}_{t^3_j,\ldots,t^3_{j+r-1}} $$

where

$$ t^3_i = | t^2_{\sigma^{-1}( \tau(i))} + t^1_{\tau(i)} - t^2_{\tau(i)}  |_m  $$

for each $i\in \{j,\ldots,j+r-1\}$.

\end{proof}

\begin{proposition}\label{prop conj model}
Let $k,m\in \mathbb N$, $m\geq 2$.
Let $l$ be a divisor of $m$, $1\leq l < m$, 
and for each $i=1,\ldots,k-1$ let $\{ t^i_{(i-1)l + 1},\ldots,t^i_{(i+1)l } \}\in \bigtriangledown_{\frac{m}{l}}$.

Suppose that $\phi_{m,l,\{ t^i_{(i-1)l + 1},\ldots,t^i_{(i+1)l } \}_{i=1,\ldots,k-1} }\colon B_k\to S_{mk}$, defined by

\begin{equation}\label{eq model phi 1}
\phi_{m,l,\{ t^i_{(i-1)l + 1},\ldots,t^i_{(i+1)l } \}_{i=1,\ldots,k-1} }(\sigma_i) = A^{\frac{m}{l}}_1 \cdots A^{\frac{m}{l}}_{(i-1)l} 
\cdot C^{\psi_l(\sigma_i),\frac{m}{l}}_{t^i_{l(i-1)+1},\ldots,t^i_{l(i+1)} } \cdot A^{\frac{m}{l}}_{(i+1)l+1} \cdots A^{\frac{m}{l}}_{kl} 
\end{equation}

for $i=1,\ldots,k-1$, is a homomorphism. 

Then $\phi_{m,l,\{ t^i_{(i-1)l + 1},\ldots,t^i_{(i+1)l } \}_{i=1,\ldots,k-1} } \sim 
\phi_{m,l,\{ t'^i_{(i-1)l + 1},\ldots,t'^i_{(i+1)l } \}_{i=1,\ldots,k-1} }$ where
for each $i=1,\ldots,k-1$ and $q$, $(i-1)l + 1\leq q \leq (i+1)l$

\begin{equation}\label{48.1}
t'^i_q = \begin{cases}
0 & q\neq (i-1)l + 1 \\
p & q = (i-1)l + 1
\end{cases} 
\end{equation}

for some $p\in \bigtriangledown_{\frac{m}{l}}$.

\end{proposition}
\begin{proof}

For the purposes of the proof we denote $\phi_{m,l,\{ t^i_{(i-1)l + 1},\ldots,t^i_{(i+1)l } \}_{i=1,\ldots,k-1} }$ as $\phi$,
$\phi_{m,l,\{ t'^i_{(i-1)l + 1},\ldots,t'^i_{(i+1)l } \}_{i=1,\ldots,k-1} }$ as $\phi'$,
and $\phi_{m,l,\{ t^i_{(i-1)l + 1},\ldots,t^i_{(i+1)l } \}_{i=1,\ldots,k-1} }(\sigma_i)$ as $\widehat \sigma_i$.

Furthermore, we use the notation~\ref{notation A}, where for each $p\in \mathbb N$, $p\geq 2$, and
each $i\geq 1$ we denote

$$ A^p_i = (1 + p(i-1),\ldots, pi) $$

We shall denote the $d$-th power of $A^p_i$ as $(A^p_i)^d$ for any $d\in \mathbb Z$. 

Consider the following element in $S_{mk}$~:

$$ D = \prod_{ 1\leq i \leq k} C^{A^l_i,\frac{m}{l}}_{s_{(i-1)l+1},\ldots,s_{il}}  $$

where the $s_i$-s are numbers in $\bigtriangledown_{\frac{m}{l}}$, yet to be determined, 
and in each $C^{A_i,\frac{m}{l}}_{s_{(i-1)l+1},\ldots,s_{il}}$ we consider $A_i$ as
a permutation in $\mathbf{S}(\{(i-1)l+1,\ldots,il \})\cong S_l$.

Let us now show that the conjugation of $\phi$ by $D$ is a model homomorphism (see Definition~\ref{Def B_k S_mk}).

We have that

$$ (\widehat \sigma_i)^D = (C^{\psi_l(\sigma_i),\frac{m}{l}}_{t^i_{l(i-1)+1},\ldots,t^i_{l(i+1)} })^D \cdot 
\prod_{j=1,\ldots,(i-1)l,(i+1)l+1,\ldots,kl} (A^{\frac{m}{l}}_j)^D $$

First note that $C^{A^l_i,\frac{m}{l}}_{s_{(i-1)l+1},\ldots,s_{il}}$ permutes the cycles
$A^{\frac{m}{l}}_{(i-1)l+1},\ldots,A^{\frac{m}{l}}_{il}$ among themselves (see notation~\ref{notation C}), i.e.,

$$ (A^{\frac{m}{l}}_{(i-1)l+1} \cdots A^{\frac{m}{l}}_{il})^{C^{A^l_i,\frac{m}{l}}_{s_{(i-1)l+1},\ldots,s_{il}}} = 
A^{\frac{m}{l}}_{(i-1)l+1}\cdots A^{\frac{m}{l}}_{il}. $$

This means
that for each $i=1,\ldots,k-1$ we have

\begin{equation}\label{eq A^C=A}
\prod_{j=1,\ldots,(i-1)l,(i+1)l+1,\ldots,kl} (A^{\frac{m}{l}}_j)^D = \prod_{j=1,\ldots,(i-1)l,(i+1)l+1,\ldots,kl} A^{\frac{m}{l}}_j
\end{equation}

Now note that since $(\psi_l(\sigma_i))^2 = A^l_i\cdot A^l_{i+1}$ for each $i=1,\ldots,k-1$, 
$C^{A^l_j,\frac{m}{l}}_{s_{(j-1)l+1},\ldots,s_{jl}}$ is disjoint from
$ C^{\psi_l(\sigma_i),\frac{m}{l}}_{t^i_{l(i-1)+1},\ldots,t^i_{l(i+1)} } $ except for $j=i,i+1$. Hence

\begin{equation}\label{eq C^D=C}
(C^{\psi_l(\sigma_i),\frac{m}{l}}_{t^i_{l(i-1)+1},\ldots,t^i_{l(i+1)} })^D =
(C^{\psi_l(\sigma_i),\frac{m}{l}}_{t^i_{l(i-1)+1},\ldots,t^i_{l(i+1)} })^{C^{A^l_i \cdot A^l_{i+1},\frac{m}{l}}_{s_{(i-1)l+1},\ldots,s_{(i+1)l}}} 
\end{equation}

(see notation~\ref{notation C} for the expression $C^{A^l_i \cdot A^l_{i+1},\frac{m}{l}}_{s_{(i-1)l+1},\ldots,s_{(i+1)l}}$).

Let us now compute the expression on the right hand side of \eqref{eq C^D=C}.

According to part~$(c)$ of Lemma~\ref{lemma C^C} and using the fact that 
$(\psi_l(\sigma_i))^2 = A^l_i\cdot A^l_{i+1}$ for each $i=1,\ldots,k-1$, we have that 

\begin{equation}\label{eq C^D=C'}
(C^{\psi_l(\sigma_i),\frac{m}{l}}_{t^i_{l(i-1)+1},\ldots,t^i_{l(i+1)} })^{C^{A^l_i \cdot A^l_{i+1},\frac{m}{l}}_{s_{(i-1)l+1},\ldots,s_{(i+1)l}}}=
C^{\psi_l(\sigma_i),\frac{m}{l}}_{t'^i_{l(i-1)+1},\ldots,t'^i_{l(i+1)} }
\end{equation}

where

\begin{equation}\label{eq t exp}
t'^i_j = | s_{\psi_l(\sigma_i)(j)} + t^i_{\psi_l^2(\sigma_i)(j)} - s_{\psi_l^2(\sigma_i)(j)} |_{\frac{m}{l}}  
\end{equation}

for each $j=l(i-1)+1,\ldots,l(i+1)$.

The equations \eqref{eq A^C=A}, \eqref{eq C^D=C} and \eqref{eq C^D=C'} prove that $\phi^D$ is indeed a model homomorphism.

We now set the values of the $s_q\in \bigtriangledown_{\frac{m}{l}}$ so that the $t'^i_q$ in \eqref{eq t exp}
will satisfy the equations \eqref{48.1} in our claim. 

We continue the proof in a series of steps as follows.

\textbf{Step 1 : }
At this step we set the values of $s_1,\ldots,s_{2l}$.

First set $s_{\psi_l^2(\sigma_1)(1)}=0$. Next set inductively the values of 
$s_{\psi_l^3(\sigma_1)(1)},\ldots,s_{\psi_l^{2l}(\sigma_1)(1)},s_{\psi_l(\sigma_1)(1)}$ by the following set of equations
(recall that $\psi_l^{2l}(\sigma_i) = id$ for any $i$)

\begin{eqnarray}\label{eq ind s}
s_{\psi_l^3(\sigma_1)(1)} & = & |s_{\psi_l^2(\sigma_1)(1)} + t^1_{\psi_l^3(\sigma_1)(1)}|_{\frac{m}{l}}\nonumber \\
& \vdots \nonumber \\
s_{\psi_l^{j+2}(\sigma_1)(1)} & = & |s_{\psi_l^{j+1}(\sigma_1)(1)} + t^1_{\psi_l^{j+2}(\sigma_1)(1)}|_{\frac{m}{l}}\nonumber \\
s_{\psi_l^{j+3}(\sigma_1)(1)} & = & |s_{\psi_l^{j+2}(\sigma_1)(1)} + t^1_{\psi_l^{j+3}(\sigma_1)(1)}|_{\frac{m}{l}}\nonumber \\
& \vdots \nonumber \\ 
s_{\psi_l^{2l}(\sigma_1)(1)} & = & |s_{\psi_l^{2l-1}(\sigma_1)(1)} + t^1_{\psi_l^{2l}(\sigma_1)(1)}|_{\frac{m}{l}}\nonumber \\
s_{\psi_l(\sigma_1)(1)} & = & |s_{\psi_l^{2l}(\sigma_1)(1)} + t^1_{\psi_l(\sigma_1)(1)}|_{\frac{m}{l}}\nonumber \\
\end{eqnarray}

where the $j$-th equation in \eqref{eq ind s} sets the value of $s_{\psi_l^{j}(\sigma_1)(1)}$
using the value of $s_{\psi_l^{j-1}(\sigma_1)(1)}$ defined in the $j-1$-th equation.

We now have, according to \eqref{eq t exp}, that

$$ t'^1_{\psi^j_l(\sigma_1)(1)} = | s_{\psi_l^{j+1}(\sigma_1)(1)} + t^1_{\psi_l^{j+2}(\sigma_1)(1)} - s_{\psi^{j+2}_l(\sigma_1)(1)} |_{\frac{m}{l}} = $$

$$ = | s_{\psi_l^{j+1}(\sigma_1)(1)} + t^1_{\psi_l^{j+2}(\sigma_1)(1)} - ( s_{\psi_l^{j+1}(\sigma_1)(1)} + 
t^1_{\psi_l^{j+2}(\sigma_1)(1)} ) |_{\frac{m}{l}} = 0 $$

for each $j=1,\ldots,2l-1$ and furthermore

$$ t'^1_1 = | s_{\psi_l(\sigma_1)(1)} + t^1_{\psi_l^2(\sigma_1)(1)} - s_{\psi_l^2(\sigma_1)(1)} |_{\frac{m}{l}} = 
|s_{\psi_l(\sigma_1)(1)} + t^1_{\psi_l^2(\sigma_1)(1)} |_{\frac{m}{l}}$$

Let us denote

$$ p = |s_{\psi_l(\sigma_1)(1)} + t^1_{\psi_l^2(\sigma_1)(1)} |_{\frac{m}{l}} \in \bigtriangledown_{\frac{m}{l}} $$

Hence we now have that the $t'^1_j$ satisfy \eqref{48.1} for each $j=1,\ldots,2l$.\hfill $\bullet$

\textbf{Step 2 : }
At this step we set the rest of the values~: $s_{2l+1},\ldots,s_{kl}$.
We do this in $k-2$ steps where in step $q$ we set the values of $s_{ql+1},\ldots,s_{(q+1)l}$
where $q=2,\ldots,k-1$.

For each $q=2,\ldots,k-1$ we set

$$ s_j = |s_{\psi_l^{-1}(\sigma_q)(j)} + t^q_j|_{\frac{m}{l}} \qquad \text{for} \quad j=ql+1,\ldots,(q+1)l $$

which is equivalent to

\begin{equation}\label{101.1}
s_{\psi_l^2(\sigma_q)(j)}  = |s_{\psi_l(\sigma_q)(j)} + t^q_{\psi_l^2(\sigma_q)(j)}|_{\frac{m}{l}} \qquad \text{for} \quad j=ql+1,\ldots,(q+1)l 
\end{equation}

Now according to \eqref{eq t exp} and \eqref{101.1}, we have that

$$ t'^q_{j} = | s_{\psi_l(\sigma_q)(j)} + t^q_{\psi_l^2(\sigma_q)(j)} - s_{\psi_l^2(\sigma_q)(j)} |_{\frac{m}{l}} =  $$

$$ = | s_{\psi_l(\sigma_q)(j)} + t^q_{\psi_l^2(\sigma_q)(j)} - ( s_{\psi_l(\sigma_q)(j)} + t^q_{\psi_l^2(\sigma_q)(j)}) |_{\frac{m}{l}} = 0 $$

for each $q=2,\ldots,k-1$ and $j=ql+1,\ldots,(q+1)l$, in accordance with \eqref{48.1}. \hfill $\bullet$

\textbf{Step 3 : }
It is now left to show that the values of $t'^q_j$ satisfy \eqref{48.1} for $q=2,\ldots,k-1$
and $j=(q-1)l+1,\ldots,ql$.

Since $\phi^D = \phi'$ is a model homomorphism, as we have shown above, we have by 
Lemma~\ref{lemma prep phi} and Lemma~\ref{lemma condition C_i} that

\begin{equation}\label{eq condition t 1}
\forall j\in \bigtriangledown_l \quad t'^i_{(i-1)l + 1 + |1+j|_l} + t'^i_{il + 1 + |1+ j|_l} \equiv
t'^{i+1}_{(i+1)l + 1 + j} + t'^{i+1}_{il + 1 + |1+ j|_l} (\text{mod} \ \frac{m}{l})
\end{equation}

for each $i=1,\ldots,k-2$. But since, by step $1$ and step $2$, we have that
$t'^i_{il + 1 + |1+ j|_l}=0$ for each $i=1,\ldots,k-1$ and $j\in \bigtriangledown_l$,
and since each summand in \eqref{eq condition t 1} is in $\bigtriangledown_{\frac{m}{l}}$, we have
that \eqref{eq condition t 1} implies 

\begin{equation}\label{eq condition t 2}
\forall j\in \bigtriangledown_l \quad t'^i_{(i-1)l + 1 + |1+j|_l} = t'^{i+1}_{il + 1 + |1+ j|_l} 
\end{equation}

Now since, by step $1$, $t'^1_1=p$ and $t'^1_2 = \cdots = t'^1_l = 0$, \eqref{eq condition t 2} implies
that for each $i=2,\ldots,k-1$ we have

$$ t'^i_{(i-1)l + 1} = p \qquad \text{and} \qquad t'^i_{(i-1)l+2} = \cdots = t'^i_{il} = 0 $$

\hfill $\bullet$

Summing up, we have shown that $\phi' = \phi^D$ is a model homomorphism and
in steps $1$, $2$ and $3$ we have shown that the values of $t'^i_q$ in $\phi'$ satisfy
\eqref{48.1} which completes the proof.

\end{proof}

\begin{proposition}\label{prop m l cond}
Let $k,m\in \mathbb N$, $m\geq 2$.
Let $l$ be a divisor of $m$, $1\leq l < m$.

For each $p\in \bigtriangledown_{\frac{m}{l}}$ let $l_p$ be the minimal number 
between $1$ and $\frac{m}{l}$ such that 

\begin{equation}\label{102.2}
l_p\cdot p \equiv 0 \ (mod \ \frac{m}{l}). 
\end{equation}

Suppose that 

\begin{equation}\label{102.1}
2l\cdot l_p \neq \frac{m}{l} \qquad \text{for all} \quad p\in \bigtriangledown_{\frac{m}{l}}
\end{equation}

(\eqref{102.1} holds for example when $\frac{m}{l}$ is odd).
Then there are exactly $\frac{m}{l} + 1$ pairwise non-conjugate model homomorphisms $B_k\to S_{mk}$
in the set of model homomorphisms defined in definition~\ref{Def B_k S_mk}.

Explicitly, these homomorphisms are

$$ \psi_m(\sigma_i) =  C^{(i,i+1),m}_{1,0} $$

and

\begin{equation}\label{eq model phi'}
\phi_{m,l,\{ p,0,\ldots,0  \}_{i=1,\ldots,k-1} }(\sigma_i) = A^{\frac{m}{l}}_1 \cdots A^{\frac{m}{l}}_{(i-1)l} 
\cdot C^{\psi_l(\sigma_i),\frac{m}{l}}_{p,0,\ldots,0 } \cdot A^{\frac{m}{l}}_{(i+1)l+1} \cdots A^{\frac{m}{l}}_{kl} 
\end{equation}

for each $p=0,\ldots, \frac{m}{l}-1$.
\end{proposition}
\begin{proof}

According to proposition~\ref{prop conj model}, every model homomorphism which is good of type $2$ (see definition~\ref{def good})
is conjugate to a homomorphism of the form \eqref{eq model phi'} for some $p\in \bigtriangledown_{\frac{m}{l}}$.

Let $\phi_1$ and $\phi_2$ be two homomorphisms of the form \eqref{eq model phi'} with $p_1$ and $p_2$ instead
of $p$ respectively. Suppose that there is an element $D\in S_{mk}$ such that $\phi_1^D = \phi_2$.
We will prove that necessarily $p_1=p_2$ which will prove our claim.

Let us denote $\phi_1(\sigma_i)$ as $\widehat \sigma_i$.

Let us write the equations $(\widehat \sigma_i)^D = \phi_2(\sigma_i)$ for $i=1,3$ as follows~:

\begin{equation}\label{eq sig1=phi}
(C^{\psi_l(\sigma_1),\frac{m}{l}}_{p_1,0,\ldots,0 })^D \cdot 
(\prod_{j=2l+1,\ldots,kl} (A^{\frac{m}{l}}_j))^D = C^{\psi_l(\sigma_1),\frac{m}{l}}_{p_2,0,\ldots,0 } \cdot 
\prod_{j=2l+1,\ldots,kl} (A^{\frac{m}{l}}_j) 
\end{equation}

\begin{equation}\label{eq sig3=phi}
(C^{\psi_l(\sigma_3),\frac{m}{l}}_{p_1,0,\ldots,0 })^D \cdot 
(\prod_{j=1,\ldots,2l,4l+1,\ldots,kl} (A^{\frac{m}{l}}_j))^D = C^{\psi_l(\sigma_3),\frac{m}{l}}_{p_2,0,\ldots,0 } \cdot 
\prod_{j=1,\ldots,2l,4l+1,\ldots,kl} (A^{\frac{m}{l}}_j) 
\end{equation}

According to lemma~\ref{lemma C}, the cyclic decomposition of each $C^{\psi_l(\sigma_i),\frac{m}{l}}_{p_1,0,\ldots,0 }$
consists of exactly $\frac{m}{l}\cdot \frac{1}{l_{p_1}}\ $ $(2l\cdot l_{p_1})$-cycles where $l_p$ is defined in \eqref{102.2}.
Since, according to our assumption \eqref{102.1}, $2l\cdot l_{p_1}\neq \frac{m}{l}$ we have that
each side of the equations \eqref{eq sig1=phi} and \eqref{eq sig3=phi} 
consists of a product between disjoint permutations whose cyclic decompositions contain cycles of different lengths. 

In particular, \eqref{eq sig1=phi} implies 

\begin{equation}\label{200.1}
(\prod_{j=2l+1,\ldots,kl} (A^{\frac{m}{l}}_j))^D = \prod_{j=2l+1,\ldots,kl} (A^{\frac{m}{l}}_j) 
\end{equation}

and \eqref{eq sig3=phi} implies

\begin{equation}\label{200.2}
(\prod_{j=1,\ldots,2l,4l+1,\ldots,kl} (A^{\frac{m}{l}}_j))^D = \prod_{j=1,\ldots,2l,4l+1,\ldots,kl} (A^{\frac{m}{l}}_j). 
\end{equation}

Now \eqref{200.1} implies in particular that

$$ \{ (A^{\frac{m}{l}}_{4l+1})^D,\ldots, (A^{\frac{m}{l}}_{kl})^D \} \subseteq \{ A^{\frac{m}{l}}_{2l+1},\ldots,A^{\frac{m}{l}}_{kl} \} $$

and \eqref{200.2} implies

$$ \{ (A^{\frac{m}{l}}_{4l+1})^D,\ldots, (A^{\frac{m}{l}}_{kl})^D \} \subseteq \{A^{\frac{m}{l}}_1,\ldots,A^{\frac{m}{l}}_{2l} \} 
\cup \{ A^{\frac{m}{l}}_{4l+1},\ldots,A^{\frac{m}{l}}_{kl} \} $$

Hence

$$ \{ (A^{\frac{m}{l}}_{4l+1})^D,\ldots, (A^{\frac{m}{l}}_{kl})^D \} \subseteq \{ A^{\frac{m}{l}}_{2l+1},\ldots,A^{\frac{m}{l}}_{kl} \} 
\cap (\{A^{\frac{m}{l}}_1,\ldots,A^{\frac{m}{l}}_{2l} \} \cup \{ A^{\frac{m}{l}}_{4l+1},\ldots,A^{\frac{m}{l}}_{kl} \}) $$

or equivalently

$$ \{ (A^{\frac{m}{l}}_{4l+1})^D,\ldots, (A^{\frac{m}{l}}_{kl})^D \} \subseteq  \{ A^{\frac{m}{l}}_{4l+1},\ldots,A^{\frac{m}{l}}_{kl} \} $$

but since the $A_i$-s are disjoint this means that

\begin{equation}\label{200.8}
\{ (A^{\frac{m}{l}}_{4l+1})^D,\ldots, (A^{\frac{m}{l}}_{kl})^D \} =  \{ A^{\frac{m}{l}}_{4l+1},\ldots,A^{\frac{m}{l}}_{kl} \} 
\end{equation}

hence \eqref{200.2} and \eqref{200.8} imply that

$$ (\prod_{j=1,\ldots,2l} (A^{\frac{m}{l}}_j))^D = \prod_{j=1,\ldots,2l} (A^{\frac{m}{l}}_j). $$

Hence, by lemma~\ref{lemma C} we can write

\begin{equation}\label{eq supp A1}
D|_{supp(\bigcup_{j=1,\ldots,2l} A_j)} = C^{\tau,\frac{m}{l}}_{s_1,\ldots,s_{2l}} 
\end{equation}

where $\tau \in S_{2l}$.

Furthermore, by \eqref{eq sig1=phi}, we have that

$$ (C^{\psi_l(\sigma_1),\frac{m}{l}}_{p_1,0,\ldots,0 })^D = C^{\psi_l(\sigma_1),\frac{m}{l}}_{p_2,0,\ldots,0 } $$

which, by \eqref{eq supp A1}, can be rewritten as

\begin{equation}\label{eq C^D=C''}
(C^{\psi_l(\sigma_1),\frac{m}{l}}_{p_1,0,\ldots,0 })^{C^{\tau,\frac{m}{l}}_{s_1,\ldots,s_{2l}}} = 
C^{\psi_l(\sigma_1),\frac{m}{l}}_{p_2,0,\ldots,0 }
\end{equation}

According to part~$(c)$ of lemma~\ref{lemma C^C}, \eqref{eq C^D=C''} implies that

$$ (\psi_l(\sigma_1))^{\tau} = \psi_l(\sigma_1) $$

but since $\psi_l(\sigma_1)$ is a $2l$-cycle, then according to part~$(b)$ of lemma~\ref{Lemma commuting permutations} we have that

$$ \tau = \psi^q_l(\sigma_1) \qquad \text{for some} \ q, \ 0\leq q \leq 2l-1.  $$

Hence we can rewrite \eqref{eq C^D=C''} as

\begin{equation}\label{eq C^D=C'''}
(C^{\psi_l(\sigma_1),\frac{m}{l}}_{p_1,0,\ldots,0 })^{C^{\psi^q_l(\sigma_1),\frac{m}{l}}_{s_1,\ldots,s_{2l}}} = 
C^{\psi_l(\sigma_1),\frac{m}{l}}_{p_2,0,\ldots,0 }
\end{equation}

In order to prove our claim it is enough to show that equation \eqref{eq C^D=C'''} has a solution
$s_1,\ldots,s_{2l}\in \bigtriangledown_{\frac{m}{l}}$ only if $p_1=p_2$.

Suppose then that \eqref{eq C^D=C'''} has a solution $s_1,\ldots,s_{2l}\in \bigtriangledown_{\frac{m}{l}}$. 

%We proceed according to the two possible cases~: $\psi^{-q}_l(\sigma_1)(1) = 1$ or $\psi^{-q}_l(\sigma_1)(1) \neq 1$

%\textbf{Case 1 :} Suppose $\psi^{-q}_l(\sigma_1)(1) = 1$.

Let us rewrite \eqref{103.1} from part $(c)$ of lemma~\ref{lemma C^C} as

\begin{equation}\label{104.1}
t^3_{\tau^{-1}(i)} = | t^2_{\sigma^{-1}(i)} + t^1_i - t^2_i  |_m 
\end{equation}

(the $m$ in \eqref{104.1} is $\frac{m}{l}$ in our case, $\tau$ is $\psi_l^q(\sigma_1)$ and $\sigma$ is $\psi_l(\sigma_1)$).

Then \eqref{eq C^D=C'''} and \eqref{104.1} imply the following set
of $2l$ equations~:

\begin{eqnarray}\label{eq ind s C^D 1}
p_2 & = & | s_{\psi^{-1}_l(\sigma_1)(1)} + 0 - s_1 |_{\frac{m}{l}}\nonumber \\
0 & = & | s_{\psi_l^{-1}(\sigma_1)(2)} + 0 - s_2 |_{\frac{m}{l}}\nonumber \\
& \vdots \nonumber \\ 
0 & = & | s_{\psi_l^{-1}(\sigma_1)(\psi^{-q}_l(\sigma_1)(1))} + p_1 - s_2 |_{\frac{m}{l}} \\
& \vdots\nonumber \\ 
0 & = & |s_{\psi_l^{-1}(\sigma_1)(2l)} + 0 - s_{2l} |_{\frac{m}{l}}\nonumber
\end{eqnarray}

Now the set of $2l$ equations \eqref{eq ind s C^D 1} can be rewritten in matrix form as follows

\begin{equation}\label{eq ind s matrix1}
(P - I)\left(\begin{matrix}
s_1 \\
\vdots\\
s_{2l}
\end{matrix}\right) \equiv p_2\cdot \vec e_1 - p_1\cdot \vec e_{\psi^{-q}_l(\sigma_1)(1)} \ (mod\  \frac{m}{l})
\end{equation}

where $I$ is the unit matrix of size $2l$, $\vec e_j$ is the unit vector of size $2l$ (it has $1$ in the $j$-th place and $0$-s elsewhere) and 
$P$ is the permutation matrix of size $2l$ corresponding to $\psi_l^{-1}(\sigma_1)$, i.e., 
its $j$-th row consists of zeros except for the $\psi^{-1}_l(\sigma_1)(j)$ column where the entry is $1$. 

Now note that since $\psi_l(\sigma_1)$ is a $2l$-cycle, its only power that has a fixed point is the identity, and so
$\psi_l^{-1}(\sigma_1)(i)\neq i$ for all $i=1,\ldots,2l$.
Hence the $j$-th column of the matrix $P-I$ in \eqref{eq ind s matrix1} has the entry $-1$ on the $j$-th row, it has
the entry $1$ on the $\psi^{-1}_l(\sigma_1)(j)$-th row (which is $\neq j$) and $0$ in the rest of the entries.
This means that the submodule of $(\mathbb Z_{\frac{m}{l}})^{2l}$ spanned
by the columns of the matrix $P-I$ consists of vectors whose entries sum
is $0$. Hence \eqref{eq ind s C^D 1} can have a solution only if $p_1=p_2$.

% In order to show that the submodule of $(\mathbb Z_{\frac{m}{l}})^{2l}$ spanned
% by the columns of the matrix $P-I$ consists of vectors whose entries sum
% is $0$, as in the previous case, we have to show that $\psi_l^{1-q}(\sigma_1)$ has no 
% fixed points. Now since $\psi_l(\sigma_1)$ is a cycle, its only power that has a fixed point
% is the identity. But $\psi_l^{1-q}(\sigma_1) = id$ would imply that $I = P$, in which
% case \eqref{eq ind s matrix2} would clearly have no solution.

\end{proof}

\section{Homomorphisms $\omega\colon B_k \to S_{mk}$ with $supp(\omega)=2m$ }\label{section B_k S_mk supp=2m}

Recall the definition of model homomorphisms $\omega\colon B_k \to S_{mk}$ with $supp(\omega)=2m$~:

$$
\psi_m(\sigma_i) =  C^{(i,i+1),m}_{1,0}
$$

(see Definition~\ref{Def B_k S_mk} and Notations~\ref{notation C} and \ref{notation A}). Alternatively, we can write

$$ \psi_m(\sigma_i) = (a^i_0, a^{i+1}_0, a^i_1, a^{i+1}_1,\ldots,a^i_{m-1},a^{i+1}_{m-1}) $$

where 

$$ a^i_j = 1 + m(i-1) + j \in supp(A^m_i) \quad j\in \bigtriangledown_m. $$

Recall (see Definition~\ref{Def B_k S_mk}) that a homomorphism $B_k\to S_n$ is called standard if it is conjugate to a model homomorphism.

\begin{theorem}\label{proposition 2m}
Let $k,m\in \mathbb N$ be such that $k>6$ and $m\geq 2$.
Every good transitive non-cyclic homomorphism $\omega\colon B_k \to S_{mk}$ with 
$supp(\omega) = 2m$  is standard.

(The condition $supp(\omega) = 2m$ corresponds to case~$1$ of Definition~\ref{def good} according to 
Lemma~\ref{lemma good})

\end{theorem}

\begin{proof}

Let $\omega\colon B_k \to S_{mk}$ be a good transitive non-cyclic homomorphism with $supp(\omega) = 2m$.

We will define a permutation $\theta\in S_{mk}$ such that $\omega^{\theta} = \psi_m$ (where $\psi_m$ is
a model homomorphism defined in Definition~\ref{Def B_k S_mk}). Recall that $\omega^{\theta}$ is defined
by conjugating each element in the image of $\omega$ by $\theta$, i.e. 

$$ \omega^{\theta}(\sigma_i) = (\omega(\sigma_i))^{\theta} = (\widehat \sigma_i)^{\theta} \quad i=1,\ldots,k-1 $$

We shall prove our claim in five steps as follows :

{\bf Step 1:} $intersect(\varphi) = m$.  

\vskip 0.2cm

{\bf Proof of Step 1:} This follows directly from Lemma~\ref{lemma good} since $supp(\omega) = 2m$ and $\omega$
is good.

\vskip 0.2cm

{\bf Step 2 :} For each $i=1,\ldots,k-1$, each cycle $D$ in the cycle decomposition of $\widehat \sigma_i$ is of even length and
$|supp(D)\cap supp(\widehat \sigma_{i\pm 1})| = \frac12 |supp(D)|$
(as in Lemma~\ref{lemma supp bound}, for $i=1$ and $i=k-1$ we take $\widehat \sigma_{1\pm 1}$ and $\widehat \sigma_{(k-1)\pm 1}$ 
to mean $\widehat \sigma_2$ and $\widehat \sigma_{k-2}$ respectively).

\vskip 0.2 cm

{\bf Proof of Step 2 : } Let $i\in \{1,\ldots,k-1\}$ and let us write

$$ \widehat \sigma_i = D_1\cdots D_r $$

according to Corollary~\ref{corollary lemma good}, for each $j=1,\ldots,r$ we have that
$|supp(D_j)\cap supp(\widehat \sigma_{i\pm 1} )| \geq \frac12 supp(D_j)$ and so

$$ intersect(\omega) = |supp(\widehat \sigma_i) \cap supp(\widehat \sigma_{i\pm 1})| = $$
$$ = \sum_{j=1,\ldots,r} |supp(D_j) \cap supp(\widehat \sigma_{i\pm 1})| \geq  \sum_{j=1,\ldots,r} \frac12 |supp(D_j)| = $$
$$ = \frac12 |supp(\widehat \sigma_i)| = m$$

but since $intersect(\omega)=m$ by step 1, we get that for each $j=1,\ldots,r$ : $|supp(D_j)\cap supp(\widehat \sigma_{i\pm 1} )| = \frac12 |supp(D_j)|$
and $|supp(D_j)|$ must be even, i.e., each $D_j$ is of even length.

\vskip 0.2cm

{\bf Step 3:} Let $i\in \{1,\ldots,k-2\}$ and let $D$ be a cycle in the cycle decomposition of $\widehat \sigma_i$
and let us write $D$ as follows (recall that $D$ is of even length by step 2)

$$ D = (x_0^i,x_0^{i+1},x_1^i,x_1^{i+1},\ldots,x^i_r,x^{i+1}_r) $$

assume w.l.o.g that $x_0^i \not \in supp(\widehat \sigma_{i+1})$ (this can be assumed by Step 2 since 
$|supp(D) \cap supp(\widehat \sigma_{i+1})| = \frac12 |supp(D)|$), 
then 

$$\{ x_0^i, x_1^i,\ldots,x_r^i \} \cap supp(\widehat \sigma_{i+1}) = \emptyset $$

and $D^{\widehat \sigma^{-1}_{i+1} \widehat \sigma^{-1}_i}$ can be written as follows

$$ D^{\widehat \sigma^{-1}_{i+1} \widehat \sigma^{-1}_i} = (x_0^{i+1},x_0^{i+2},x_1^{i+1},x_1^{i+2},\ldots,x^{i+1}_r,x^{i+2}_r) $$

where $\{ x_0^{i+2}, x_1^{i+2},\ldots,x_r^{i+2} \} \cap supp(\widehat \sigma_i) = \emptyset$

{\bf Proof of Step 3:} Let us first consider $D$ 

$$ D = (x_0^i,x_0^{i+1},x_1^i,x_1^{i+1},\ldots,x^i_r,x^{i+1}_r) $$

Consider $(x_0^i,x_0^{i+1},x_1^i,x_1^{i+1},\ldots,x^i_r,x^{i+1}_r)$ as an ordered tuple of (distinct) elements in $\Delta_{mk}$.
According to Lemma~\ref{lemma supp bound} at least one element in each consecutive pair in this tuple 
belongs to $supp(\widehat \sigma_{i+1})$ and according to 
Step 2 exactly half of the elements of this tuple belong to $supp(\widehat \sigma_{i+1})$. This means that
in each consecutive pair in this tuple exactly one element belongs to $supp(\widehat \sigma_{i+1})$. 
Since, according to our assumption, $x_0^i \not \in supp(\widehat \sigma_{i+1})$, we conclude that

$$ \{ x_0^i,x_1^i,\ldots,x_r^i \} \cap supp(\widehat \sigma_{i+1}) = \emptyset $$

and $x_j^{i+1} \in supp(\widehat \sigma_{i+1})$ for each $j=0,\ldots,r$.

This means that $\widehat \sigma_i \widehat \sigma_{i+1}(x_j^i) = \widehat \sigma_i(x_j^i) = x_j^{i+1}$ for each $j=0,\ldots,r$.

Let us now compute $D^{\widehat \sigma^{-1}_{i+1} \widehat \sigma^{-1}_i}$ directly 

$$  D^{\widehat \sigma^{-1}_{i+1} \widehat \sigma^{-1}_i} = (\widehat \sigma_i \widehat \sigma_{i+1}(x_0^i),
\widehat \sigma_i \widehat \sigma_{i+1}(x_0^{i+1}),\widehat \sigma_i \widehat \sigma_{i+1}(x_1^i),\widehat \sigma_i \widehat \sigma_{i+1}(x_1^{i+1}),
\ldots,\widehat \sigma_i \widehat \sigma_{i+1}(x^i_r),\widehat \sigma_i \widehat \sigma_{i+1}(x^{i+1}_r)) = $$

$$  = (x_0^{i+1},\widehat \sigma_i \widehat \sigma_{i+1}(x_0^{i+1}),x_1^{i+1},\widehat \sigma_i \widehat \sigma_{i+1}(x_1^{i+1}),
\ldots,x^{i+1}_r,\widehat \sigma_i \widehat \sigma_{i+1}(x^{i+1}_r))  $$
 
Let us denote $\widehat \sigma_i \widehat \sigma_{i+1}(x_j^{i+1})$ as $x_j^{i+2}$ for $j=0,\ldots,r$. According to step 2
we have that $D^{\widehat \sigma_{i+1} \widehat \sigma_i}$, which is a cycle of $\widehat \sigma_{i+1}$, satisfies 
$|supp(D^{\widehat \sigma_{i+1} \widehat \sigma_i}) \cap supp(\widehat \sigma_i)| = \frac12 |supp(D^{\widehat \sigma_{i+1} \widehat \sigma_i})| $
and so, since $x_j^{i+1} \in supp(\widehat \sigma_i)$ for each $j=0,\ldots,r$ ,this means that

$$\{ x_0^{i+2}, x_1^{i+2},\ldots,x_r^{i+2} \} \cap supp(\widehat \sigma_i) = \emptyset$$
 
\vskip 0.2cm

{\bf Step 4:} The cyclic type of $\widehat \sigma_1$ (and hence of $\widehat \sigma_i$ for all $i=1,\ldots,k-1$)
is $[2m]$ (i.e. the cycle decomposition of each $\widehat \sigma_i$ consists of a single $2m$-cycle)

{\bf Proof of Step 4:} Let $D_1$ be a cycle in the cycle decomposition of $\widehat \sigma_1$. Let us write $D_1$
as follows (recall that $D_1$ is of even length by step 2)

$$ D_1 = (x_0^1,x_0^2,x_1^1,x_1^2,\ldots,x^1_r,x^2_r) $$

where we assume w.l.o.g that $x_0^1 \not \in supp(\widehat \sigma_2)$ (this can be assumed by Step 2 since 
$|supp(D) \cap supp(\widehat \sigma_{i+1})| = \frac12 |supp(D)|$ for each $i=1,\ldots,k-2$)

For $i=1,\ldots,k-2$ define inductively

$$ D_{i+1} = D_i^{\widehat \sigma^{-1}_{i+1} \widehat \sigma^{-1}_i} $$

since $\widehat \sigma_i^{\widehat \sigma^{-1}_{i+1} \widehat \sigma^{-1}_i} = \widehat \sigma_{i+1}$ we have that each $D_i$
is a cycle in the cycle decomposition of $\widehat \sigma_i$ and according to step 3 we can write

$$ D_i = (x_0^i,x_0^{i+1},x_1^i,x_1^{i+1},\ldots,x^i_r,x^{i+1}_r) $$

where $x_j^i \not \in supp(\widehat \sigma_{i+1})$ and $x_j^{i+1} \in supp(\widehat \sigma_{i+1})$ for $j=0,\ldots,r$.

Now let us prove that the orbit of $x_0^1\in \Delta_{mk}$ inside the permutation
group generated by $\{\widehat \sigma_1,\ldots,\widehat \sigma_{k-1}\}$, denoted $orbit(x_0^1)$, 
is equal to the union of the supports of the $D_i$, i.e.,

$$ orbit(x_0^1) = \bigcup_{i=1,\ldots,k-1} supp(D_i). $$

Let us denote the orbit of $x_0^1\in \Delta_{mk}$ inside the permutation
group generated by $\{\widehat \sigma_1,\ldots,\widehat \sigma_{j-1}\}$ as $orbit_j(x_0^1)$, so $orbit(x_0^1) = orbit_{k-1}(x_0^1)$.
The proof will be carried out by induction on $j$.

The induction base is clear since the orbit of $x_0^1$ inside the group generated by $\widehat \sigma_1$ is
exactly $supp(D_1)$. Now assume the induction hypothesis for $j-1$ and prove for $j$. Since
$\omega$ is good (case~$1$ of Definition~\ref{def good}), we have that 

\begin{equation}\label{50.3}
supp(\widehat \sigma_j) \cap \bigcup_{i=1,\ldots,j-2} supp(\widehat \sigma_i) = \emptyset 
\end{equation}

and in particular 

\begin{equation}\label{50.1}
supp(D_j) \cap \bigcup_{i=1,\ldots,j-2} supp(\widehat \sigma_i) = \emptyset.
\end{equation}

Furthermore, by step 3 we have that

\begin{equation}\label{50.2}
supp(D_j) \cap supp(\widehat \sigma_{j-1}) = \{ x_0^j,\ldots,x_r^j \} 
\end{equation}

% and by step 2 we have that

% $$ |supp(D_{j-1}) \cap supp(\widehat \sigma_j)| = \frac12 |supp(D_{j-1})| = |\{ x_0^j,\ldots,x_r^j \}| $$

and so, it follows from \eqref{50.1} and \eqref{50.2} that

$$ supp(D_j) \cap \bigcup_{i=1,\ldots,j-1} supp(\widehat \sigma_i) = $$ 

$$ =  (supp(D_j) \cap \bigcup_{i=1,\ldots,j-2} supp(\widehat \sigma_i)) \cup (supp(D_j \cap supp(\widehat \sigma_{j-1}) ) = $$ 

\begin{equation}\label{50.4}
= \{ x_0^j,\ldots,x_r^j \} 
\end{equation}

and from \eqref{50.3} it follows that for any other cycle $E_j$, $E_j\neq D_j$, in the cycle decomposition of $\widehat \sigma_j$

\begin{equation}\label{50.5}
supp(E_j) \cap \bigcup_{i=1,\ldots,j-1} supp(D_i) =  \emptyset. 
\end{equation}

Now according to the induction hypothesis we have that $orbit_{j-1}(x^1_0) = \bigcup_{i=1,\ldots,j-1} supp(D_i)$ and
by \eqref{50.4} and \eqref{50.5} we have that $orbit_{j-1}(x^1_0)$ has a non-trivial intersection only with the cycle $D_j$
in $\widehat \sigma_j$. This means that

$$ orbit_j(x_0^1) = \bigcup_{i=1,\ldots,j-1} supp(D_i) \cup supp(D_j) = \bigcup_{i=1,\ldots,j} supp(D_i) $$

which concludes the induction proof. We now have that 

$$ |orbit(x_0^1)| = |\bigcup_{i=1,\ldots,k-1} supp(D_i)| = |supp(D_1)| + \sum_{j=2,\ldots,k-1} |supp(D_j) - supp(D_{j-1})| =  $$
$$ = (2r+2) + (k-2)(r+1) $$

but since $\omega$ is transitive we must have that

$$ 2r+2 + (k-2)(r+1) = mk $$

which is equivalent to 

$$ k(r+1) = mk $$

and we conclude that $r=m-1$ and each $\widehat \sigma_i$ is a single $2m$-cycle.

{\bf Step 5:} There exists a permutation $\theta \in S_{mk}$ such that $\omega^{\theta} = \psi_m$
($\psi_m$ is the model homomorphism defined in Definition~\ref{Def B_k S_mk}).
  
\vskip 0.2cm

{\bf Proof of Step 5:} By step 4, each $\widehat \sigma_i$ for $i=1,\ldots,k-1$ is
a single $2m$-cycle. Let us write $\widehat \sigma_1$ explicitly

$$ \widehat \sigma_1 = (x_0^1,x_0^2,x_1^1,x_1^2,\ldots,x^1_{m-1},x^2_{m-1}) $$

where we assume that $x_0^1 \not \in supp(\widehat \sigma_2)$ (this can be assumed by Step 2 since 
$|supp(D) \cap supp(\widehat \sigma_{i+1})| = \frac12 |supp(D)|$ for each $i=1,\ldots,k-2$).

Since $\widehat \sigma_i^{\widehat \sigma^{-1}_{i+1} \widehat \sigma^{-1}_i} = \widehat \sigma_{i+1}$ for each $i=1,\ldots,k-2$, step 3
implies that we can write

$$ \widehat \sigma_i = (x_0^i,x_0^{i+1},x_1^i,x_1^{i+1},\ldots,x^i_{m-1},x^{i+1}_{m-1}) $$

where $x_j^i \not \in supp(\widehat \sigma_{i+1})$ and $x_j^{i+1} \in supp(\widehat \sigma_{i+1})$ for $j=0,\ldots,m-1$.
This means that $\{x_0^i,\ldots,x_{m-1}^i\}_{i=1,\ldots,k}$ is a set of $mk$ distinct elements, i.e. it is
equal to $\Delta_{mk}$. Hence we can define the permutation $\theta \in S_{mk}$ as follows

$$ \theta(a_j^i) = x_j^i \quad j=0,\ldots,m-1 \quad i=1,\ldots,k $$

where

$$ a^i_j = 1 + m(i-1) + j \in supp(A^m_i) \quad j\in \bigtriangledown_m $$

which would give $\omega^{\theta} = \psi_m$ (see Definition~\ref{Def B_k S_mk}).

\end{proof}

\section{Good Transitive Homomorphisms $\omega\colon B_k \to S_{3k}$ }\label{section B_k S_3k supp=mk}

In this section we prove that good transitive homomorphisms $\omega\colon B_k\to S_{mk}$ are standard for $m=3$, i.e. they
are conjugate to the model homomorphisms (see Definition~\ref{Def B_k S_mk}).

\begin{notation}\label{notation new A}

Recall notation~\ref{notation A} in the case $m=3$ :

\begin{equation}
A^3_i = (3i-2,3i-1,3i) \quad \quad i\geq 1
\end{equation}

In this section we shall denote $A^3_i$ simply as $A_i$ and we denote

\begin{equation}\label{equation A^-1}
(A_i\cdot A_i=)A^{-1}_i = (3i-2,3i,3i-1) \quad \quad i\geq 1
\end{equation}

(i.e. we use the superscripts of the $A_i$-s in the conventional manner instead of their use in notation~\ref{notation A}).\hfill $\bullet$

\end{notation}

% We also denote for each $j\in \bigtriangledown_m$

% $$ a^i_j = 3i-2 + j \in supp(A_i) $$

We shall use the following important results from \cite{Lin04b}

\begin{lemma}[cf. Lemma 7.12 in \cite{Lin04b}]\label{Lemma k-3 bound} 
Let $k>6$ and let $\omega\colon B_k\to{\mathbf S_n}$ be
a homomorphism such that all components
of $\widehat\sigma_1$ $($including the degenerate 
component $Fix(\widehat\sigma_1))$
are of lengths at most $k-3$. Then $\omega$ is cyclic.
\end{lemma}

\begin{lemma}[cf. \cite{Artin2} and Lemma 2.7 in \cite{Lin04b}]
\label{Lemma commuting permutations}
Suppose that $A,B\in{\mathbf S_n}$ and $AB = BA$. Then:

$a)$ The set $supp(A)$ is $B$-invariant;

$b)$ If for some $r$, \ $2\le r\le n$, the $r$-component
of $A$ consists of a single $r$-cycle $C$,
then the set $supp(C)$ is $B$-invariant and
$B|supp(C) = C^q$ for some integer $q$, \ $0\le q<r$.
\end{lemma}

% As a corollary to Lemma~\ref{Lemma commuting permutations} we get

% \begin{corollary}
% \label{corollary commuting permutations}
% Let $C,B_1,\ldots,B_t$ be $3$-cycles in $S_n$ (for some $n \geq  3t$) and suppose that $B_1,\ldots,B_t$ are
% mutually disjoint. If $[C,B_1\cdots B_t]=1$ then for each $i=1,\ldots,t$ we have that $B_i|supp(C)=C^q$ for some integer $q$, \ $0\le q<3$.
% \end{corollary}
% \begin{proof}
% According to our assumption, we have that

% $$ (B_1\cdots B_t)^C = B_1^C\cdots B_t^C = B_1\cdots B_t $$

% and since the $B_i$ are mutually disjoint we must have that for each $i=1,\ldots,t$ $B_i^C = B_j$.
% Since $C$ is a $3$-cycle, $B_i$ and $B_j$ could not be disjoint 
% and we conclude that for each $i=1,\ldots,t$ we have that $B_i^C = B_i$ in which case
% we can apply Lemma~\ref{Lemma commuting permutations} to conclude that $B_i|supp(C)=C^q$ for some integer $q$, \ $0\le q<3$.

% \end{proof}

\begin{lemma}[Normalization assumption]\label{lemma normal}
Let $\omega\colon B_k \to S_n$ be a homomorphism and assume $\mathfrak D_r = \{ A^r_j,\ldots,A^r_{j+t-1} \}$ 
(for some $j\geq1$, see notation~\ref{notation A}) be an $r$-subcomponent of $\widehat \sigma_1$ (of $\widehat \sigma_{k-1}$).
Let $\Omega_{\mathfrak D_r}$ be the (co)retraction of $\omega$ to $\mathfrak D_r$ (assuming it is well defined). Suppose $\Omega_{\mathfrak D_r}$
is conjugate to $\Omega'$, then $\omega$ is conjugate to $\omega^C =\omega'$ for some $C\in \mathbf{S}(supp(\mathfrak D_r))$ 
where the (co)retraction of $\omega'$ to $\mathfrak D_r$ 
is equal to $\Omega'$ and the conjugation of the cycles of $\mathfrak D_r$ by $C$ induces a permutation between them.
\end{lemma}
\begin{proof}

We will prove the claim only for retractions, the dual claim (for coretractions) being completely similar.

Assume then that $\mathfrak D_r = \{ A^r_j,\ldots,A^r_{j+t-1} \}$ is an $r$-subcomponent of $\widehat \sigma_1$
and denote $\Omega_{\mathfrak D_r}=\Omega$. Recall that by definition $\Omega(\sigma_l)\in \mathbf{S}(\{j,\ldots,j+t-1 \})\simeq S_t$
for each $l=3,\ldots,k-1$. Suppose $\Omega' = \eta \Omega \eta^{-1}$ for some $\eta\in \mathbf{S}(\{j,\ldots,j+t-1 \})\simeq S_t$.

By using Lemma~\ref{lemma C}, 
we construct a permutation $C\in \mathbf{S}(\bigcup_{s=j,\ldots,j+t-1}supp(A_s))$ that corresponds to $\eta$ and satisfies

$$ (A^r_p)^{C} = A^r_{\eta(p)} \quad p\in supp(\eta) $$

And so, conjugation of the cycles $\mathfrak D_r = \{ A^r_j,\ldots,A^r_{j+t-1} \}$
by $C$ induces a permutation between them. Furthermore, we have for each $l=3,\ldots,k-1$

$$ (A^r_p)^{C^{-1}\widehat \sigma_l C} = (A^r_{\eta \Omega(\sigma_l)\eta^{-1}(p)}) \quad \text{for} \ j\leq p \leq j+t-1 $$

or equivalently

$$ (A^r_p)^{(\widehat \sigma_l)^C } = (A^r_{\Omega'(\sigma_l)(p)}) \quad \text{for} \ j\leq p \leq j+t-1 $$

which means that the retraction of $\omega^C$ to $\mathfrak D_r$ is $\Omega'$.

\end{proof}

\begin{lemma}\label{lemma missing mu}
Let $\omega\colon B_k \to S_k$ be a non-cyclic homomorphism such that $\omega(\sigma_i)=(i,i+1)$ for $i \neq k-2$, then
either $\omega(\sigma_{k-2}) = (k-2,k-1)$ or  $\omega(\sigma_{k-2}) = (k-2,k)$

\end{lemma}
\begin{proof}
If $\omega(\sigma_{k-2})$ were disjoint from $\omega(\sigma_{k-3})$ or $\omega(\sigma_{k-1})$ it would commute with
either of them and then $\omega$ would be cyclic according to Lemma~\ref{lemma cyclic homomorphism} and Fact~\ref{fact braid commute}, 
contradicting our assumption. For the same reason, $\omega(\sigma_{k-2})$ could not be equal to either $\omega(\sigma_{k-3})$ or $\omega(\sigma_{k-1})$.
This leaves only $4$ possibilities for $\omega(\sigma_{k-2})$ : $(k-3,k-1)$, $\ (k-3,k)$, $\ (k-2,k-1)$ or $\ (k-2,k)$. 
$(k-3,k-1)$ and $\ (k-3,k)$ are impossible since they do not commute with $\omega(\sigma_{k-4}) = (k-4,k-3)$.
$\ (k-2,k-1)$ and $\ (k-2,k)$ are readily seen to be appropriate possibilities.
\end{proof}

We are now ready to prove

\begin{theorem}
For $k > 8$ every good transitive non-cyclic homomorphism $\omega\colon B_k \to S_{3k}$ is standard.
\end{theorem}

\begin{proof}

Since $\omega$ is good by assumption, then according to Lemma~\ref{lemma good} there are only two possibilities~: Either $supp(\omega)=6$
or $supp(\omega)=3k$. It was proved in Theorem~\ref{proposition 2m} (for $m=3$) 
that $\omega$ is standard in the case $supp(\omega)=6$. Assume then that $supp(\omega)=3k$.
Then the degenerate component of $\widehat \sigma_1$ is empty. 
Furthermore, by Lemma~\ref{lemma good} we have in this case $intersect(\omega)=3k$.

In the rest of the proof we will denote the $r$-component of $\widehat \sigma_i$ as $\mathfrak C^i_r$ for $i$, $1\leq i \leq k-1$.
In case $i=1$ we also denote $\mathfrak C_r = \mathfrak C_r^1$ and in case $i=k-1$ we denote $\mathfrak C^*_r = \mathfrak C^{k-1}_r$.

Note that we have the following equality

\begin{equation}\label{eq r comps}
\sum_{r > 1} r\cdot |\mathfrak C_r | = supp(\omega) = 3k 
\end{equation}

and as a consequence for any $r > 1$

\begin{equation}\label{ineq r comps}
r\cdot |\mathfrak C_r | \leq 3k.
\end{equation}

According to Lemma~\ref{Lemma k-3 bound} there must be some $r$-component of $\widehat \sigma_1$ 
of length $\ge k-2$. But if $|\mathfrak C_r| \geq k-2$ then by \eqref{ineq r comps} we have 

\begin{equation}\label{106.1}
r \leq \frac{3k}{|\mathfrak C_r |} \leq \frac{3k}{k-2} = 3 + \frac{6}{k-2} < 4 \quad \text{for} \ k>8 
\end{equation}

Hence $r$ is $2$ or $3$ and we now proceed according to these two possibilities.

% (both possibilities cannot hold simultaneously~: By \eqref{eq r comps} we have that

% $$ 2\cdot |\mathfrak C_2| + \sum_{r > 3} r\cdot |\mathfrak C_r | = 3k  $$

% and so, the assumption $|\mathfrak C_2| \geq k-2$ would imply that

% $$ \sum_{r > 2} r\cdot |\mathfrak C_r | = 3k - 2\cdot |\mathfrak C_2| \leq 3k - 2(k-2) = k + 4$$

% and in particular $3\cdot |\mathfrak C_3| \leq k+4$ hence it is impossible that $|\mathfrak C_3| \geq k-2$
% since $\frac{k+4}{3} <k-2$ for $k>8$).

%We now proceed according to these two possibilities.
 
\textbf{Case 1 : $\mathbf{|\mathfrak C_2 |\geq k-2} $ }
%(The $2$-component of $\widehat \sigma_1$ is of length  $\geq k-2$) } 

We will show that this case is impossible.

First we claim that $|\mathfrak C_r | < k-2 $ for each $r>2$. 
For $r>3$ this follows from \eqref{106.1} and for $r=3$, 
if $|\mathfrak C_3| \geq k-2$, we would have that

$$ 3\cdot |\mathfrak C_3| + 2\cdot |\mathfrak C_2| \geq 3(k-2) + 2(k-2) = 5k-10 > 3k \qquad \text{for} \quad k>8$$

contradicting \eqref{eq r comps}.

Hence we conclude from Theorem~\ref{Thm F}~$(e)$ that for each $r$-component $\mathfrak C_r$ where $r>2$, the
retraction of $\omega$ to $\mathfrak C_r$, $\Omega_{\mathfrak C_r}\colon B_{k-2}\to \mathbf{S}(\mathfrak C_r)\cong S_{|\mathfrak C_r|} $ , 
is cyclic (since $|\mathfrak C_r|<k-2$).
Put $\Sigma_{\mathfrak C_r} = supp(\mathfrak C_r)$. According to Corollary~\ref{cor cyclic omega} the reduction of $\omega$ corresponding
to each such component $\mathfrak C_r$, $\omega_{\mathfrak C_r}\colon B_{k-2}\to \mathbf{S}(supp(\mathfrak C_r))$, is also cyclic.
This means that
$\omega_{{\Sigma_{\mathfrak C_r}}}(\sigma_3)=\ldots
=\omega_{{\Sigma_{\mathfrak C_r}}}(\sigma_{k-1})$,
and thus
%&
\begin{equation}\label{7.5}
\widehat\sigma_3| {\Sigma_{\mathfrak C_r}}
=\ldots =\widehat\sigma_{k-1}| {\Sigma_{\mathfrak C_r}}.
\end{equation}
%&
Put $\Sigma=\bigcup_{\mathfrak C_r} \Sigma_{\mathfrak C_r}$,
where $\mathfrak C_r$ runs over all the $r$-components of
$\widehat\sigma_1$ for $r>2$. The set
$\Sigma$ is invariant under all the permutations 
$\widehat\sigma_3,...,\widehat\sigma_{k-1}$. Since \eqref{7.5} holds for
every $r$-component $\mathfrak C_r$ of $\widehat\sigma_1$ for $r>2$,
it follows that there is a permutation $S\in{\mathbf S}(\Sigma)$ such that
\ \ $\widehat\sigma_3|\Sigma=\ldots
=\widehat\sigma_{k-1}|\Sigma = S$. \

Consider now $\mathfrak C_2$. Denote $\Sigma' = supp(\mathfrak C_2)$.
There are two possibilities for $\Omega_{\mathfrak C_2}$ (the retraction of $\omega$ to $\mathfrak C_2$) : It is either cyclic or non-cyclic. If
$\Omega_{\mathfrak C_2}$ were cyclic, we would get that the reduction of $\omega$ corresponding to $\Sigma'$ would be cyclic
by Corollary~\ref{cor cyclic omega} and hence there would be a permutation $S'\in \mathbf{S}(\Sigma')$ such that
$\ \ \widehat\sigma_3|\Sigma'=\ldots
=\widehat\sigma_{k-1}|\Sigma' = S',\ $
but since $\Sigma \cup \Sigma' = \Delta_{3k} = supp(\widehat \sigma_i)$ for each $i=3,\ldots,k-1$, this would mean that
\ \ $\widehat\sigma_3=\ldots
=\widehat\sigma_{k-1} = SS'$ \,
which would imply by Lemma~\ref{lemma cyclic homomorphism} that $\omega$ is cyclic which contradicts our assumption.

Hence assume that $\Omega_{\mathfrak C_2}$ is non-cyclic. Note that $|\mathfrak C_2| < 2(k-2)$ for 
otherwise

$$ 2\cdot |\mathfrak C_2| \geq 2\cdot 2(k-2) = 4k-8 > 3k \quad \text{for} \ k>8  $$

contradicting \eqref{ineq r comps}. Hence $k-2 \leq |\mathfrak C_2| < 2(k-2)$ and we may conclude from 
Theorem~\ref{Thm F}~$(c)$ and $(d)$ that 
$\Omega_{\mathfrak C_2}\colon B_{k-2}\to \mathbf{S}(\mathfrak C_2)\cong S_{|\mathfrak C_2|}$ is conjugate to
$\mu_{k-2}\times \kappa$ where $\kappa$ is cyclic ($\mu_{k-2}$ is the canonical epimorphism, see Definition~\ref{def canonical}). 
This means that there is a $2$-subcomponent, $\mathfrak C_2'$, of $\widehat \sigma_1$ such that $\kappa$ is the retraction of
$\omega$ to $\mathfrak C_2'$. Since $\kappa$ is cyclic then according to Corollary~\ref{cor cyclic omega}, the reduction of $\omega$ 
corresponding to the component $\mathfrak C'_2$, $\omega_{\mathfrak C'_2}\colon B_{k-2}\to \mathbf{S}(supp(\mathfrak C'_2))$, is also cyclic.
This means that
$\omega_{{\Sigma_{\mathfrak C'_2}}}(\sigma_3)=\ldots
=\omega_{{\Sigma_{\mathfrak C'_2}}}(\sigma_{k-1})$,
and thus
%&
\begin{equation}\label{7.15}
\widehat\sigma_3| {\Sigma_{\mathfrak C'_2}}
=\ldots =\widehat\sigma_{k-1}| {\Sigma_{\mathfrak C'_2}}.
\end{equation}
%&
Now, the set
$\Sigma \cup supp(\mathfrak C'_2)$ is invariant under all the permutations 
$\widehat\sigma_3,...,\widehat\sigma_{k-1}$, and hence , by combining \eqref{7.15} and \eqref{7.5}
it follows that there is a permutation $T'\in \mathbf{S}(\Sigma \cup supp(\mathfrak C_2'))$ such that
\ \ $\widehat\sigma_3|\Sigma \cup supp(\mathfrak C_2')=\ldots
=\widehat\sigma_{k-1}|\Sigma \cup supp(\mathfrak C_2') = T'$. \

Denote $\mathfrak C_2'' = \mathfrak C_2 - \mathfrak C_2'$ and set $T_i'' = \widehat\sigma_i|supp(\mathfrak C_2'')$, \ $i=3,...,k-1$.
We now have

\begin{equation}\label{7.7}
\widehat\sigma_i=T'\cdot T_i'' \qquad {\text{for all}} \ \ i=3,...,k-1.
\end{equation}

where 

\begin{equation}\label{7.17}
supp(T')\cap supp(T_i'')=\emptyset, \quad |supp(T')| = k+4 \quad \text{and} \quad |supp(T_i'')| = 2(k-2) 
\end{equation}

for each $i=3,...,k-1$.

Consider the reduction of $\omega$ to $supp(\mathfrak C_2'')$, 
$\omega_{supp(\mathfrak C_2'')}\colon B_{k-2}\to \mathbf{S}(supp(\mathfrak C_2''))\cong S_{2(k-2)}$. Denote $\phi=\omega_{supp(\mathfrak C_2'')}$.
According to the above arguments $\phi$ is not cyclic and the degenerate component of $\phi(\sigma_i)$ is empty for each $i=3,\ldots,k-1$. 
Hence, according to Lemma~$7.28$ in \cite{Lin04b}, $\phi$ is standard and is conjugate to either of the 
model homomorphisms $\varphi_2$ or $\varphi_3$ as defined in Definition~\ref{Def B_k S_2k}.

% hence it follows from Equation~\ref{7.7} that the permutations 
% $S_3',...,S_{k-1}'$ satisfy the standard braid relations
% $S_i' S_j'=S_j' S_i'$ for
% $|i-j|>1$ and $S_i' S_{i+1}' S_i'
% = S_{i+1}' S_i' S_{i+1}'$ for
% $3\le i<k-1$. This means that we can define a group homomorphism
% $\phi\colon B_{k-2}=<s_1,\ldots,s_{k-3}>\to{\mathbf S}(supp(\mathfrak C_2'')) \cong  S_{2(k-2)}$
% by \ $\phi(s_i)=S_{i+2}'$, \ \ $i=3,...,k-1$.
% By Theorem~\ref{Thm F} we have that $\phi$ is standard.

Since $\widehat \sigma_1$ and $\widehat \sigma_{k-1}$ are conjugate, they have the same cycle structure. Hence
we can argue in a completely similar way on the components of $\widehat \sigma_{k-1}$ (instead of $\widehat \sigma_1$)
to conclude that there are permutations $R',R''_i\in S_{3k}$ such that

\begin{equation}\label{7.8}
\widehat\sigma_i=R'\cdot R_i'' \qquad {\text{for all}} \ \ i=1,...,k-3.
\end{equation}

where

\begin{equation}\label{7.18}
supp(R')\cap supp(R_i'')=\emptyset, \quad |supp(R')| = k+4 \quad \text{and} \quad |supp(R_i'')| = 2(k-2) 
\end{equation}

for each $i=3,...,k-1$.

% Furthermore $\phi^*$, the reduction of $\omega$ to $supp(R_3'')$, 
% is standard and is conjugate to either of the model homomorphisms $\varphi_2$ or $\varphi_3$ 
% as defined in Definition~\ref{Def B_k S_2k}.

% and 
% $\phi^* \colon B_{k-2}=<s_1,\ldots,s_{k-3}>\to{\mathbf S}(\bigcup_{i=1,\ldots,k-3} supp(R_i')) \cong  S_{2(k-2)}$
% where $\phi^*(s_i) = R_i'$ for $i=1,\ldots,k-3$ is a standard homomorphism.

By comparing \eqref{7.7} and \eqref{7.8} we have that

\begin{equation}\label{7.7+7.8}
\widehat\sigma_i=T'\cdot T_i'' = R'\cdot R_i'' \qquad {\text{for all}} \ \ i=3,...,k-3. 
\end{equation}

In particular

\begin{equation}\label{8.1}
T_i''|_{supp(T_i'')} = R'|_{supp(T_i'')} \cdot R_i''|_{supp(T_i'')} \qquad i=3,\ldots,k-3. 
\end{equation}

Since $\phi\colon B_{k-2}\to \mathbf{S}(supp(\mathfrak C_2'')) \cong S_{2(k-2)}$, where $\phi(\sigma_i) = T_i''|_{supp(T_i'')}$, 
is conjugate to either $\varphi_2$ or $\varphi_3$ and since $supp(R')\cap supp(R''_i)=\emptyset$ by \eqref{7.18} , we get from \eqref{8.1}, 
according to Lemma~\ref{lemma cyclic part model}, that $ |supp(R'|_{supp(T_i'')})| \leq 4 $, i.e., $|supp(R')\cap supp(T_i'')|\leq 4$. 
Hence, by using \eqref{7.17} and \eqref{7.18}, we have

$$ k+4 = |supp(R')| = |supp(R') \cap \Delta_{3k}| = |supp(R') \cap ( supp(T') \cup supp(T''_i)  )| =  $$

$$ = |(supp(R') \cap supp(T')) \cup ( supp(R') \cap supp(T''_i) )| \leq $$ 

$$ \leq |(supp(R') \cap supp(T'))| + |( supp(R') \cap supp(T''_i) )| \leq |(supp(R') \cap supp(T'))| + 4 $$

And so $\ |(supp(R') \cap supp(T'))| \geq k$.

Now according to \eqref{7.7+7.8}, the set $\Psi' = supp(R') \cap supp(T')$ is a union of supports of cycles in each 
$\widehat \sigma_i$, $i=3,\ldots,k-3$ and so there exists a non-trivial permutation $U\in \mathbf{S}({supp(R')\cap supp(T')})$ such that

$$ R'|_{\Psi'} = T'|_{\Psi'} = U $$ 

%In particular $\Psi' \in Inv(\langle \widehat \sigma_3,\ldots, \widehat \sigma_{k-3} \rangle)$ and we conclude that there is 

which means that

\begin{equation}\label{7.9}
\widehat\sigma_i=U\cdot U_i \qquad {\text{for all}} \ \ i=1,...,k-1.
\end{equation}

where $supp(U) \cap supp(U_i) = \emptyset$ for $i=1,\ldots,k-1$ and so $\omega$ is not transitive which 
contradicts our assumption.

\textbf{Case 2 :  $\mathbf{|\mathfrak C_3 |\geq k-2} $} 

We will divide this case to the three possibilities : $\mathbf{|\mathfrak C_3 |= k-1, k, \ \text{or} \  k-2} $ respectively.

\textbf{Subcase 2.1 :  $\mathbf{|\mathfrak C_3 |= k-1} $} 

We will show that this case is impossible.

Suppose $|\mathfrak C_3| = k-1$. Let us write \eqref{eq r comps} in this case~:

$$ 3\cdot |\mathfrak C_3| + \sum_{r\neq 3, r>1} r\cdot |\mathfrak C_r| = 3(k-1) + \sum_{r\neq 3, r>1} r\cdot |\mathfrak C_r| = 3k  $$

and this implies the equation

$$ \sum_{r\neq 3, r>1} r\cdot |\mathfrak C_r| = 3 $$

which has no solutions.

\textbf{Subcase 2.2 :  $\mathbf{|\mathfrak C_3 |= k} $} 

We will show that this case is impossible as well.

Denote $\Sigma = supp(\mathfrak C_3)$. 
There are two possibilities for $\Omega_{\mathfrak C_3}$ (the retraction of $\omega$ to $\mathfrak C_3$) : It is either cyclic or non-cyclic. If
$\Omega_{\mathfrak C_3}$ were cyclic, we would get that the reduction of $\omega$ corresponding to $\Sigma$ would be cyclic
by Corollary~\ref{cor cyclic omega} and hence there would be a permutation $S\in \mathbf{S}(\Sigma)$ such that
$\ \ \widehat\sigma_3|\Sigma=\ldots
=\widehat\sigma_{k-1}|\Sigma = S,\ $
but since $\Sigma = \Delta_{3k}$ in this case, this would mean that
\ \ $\widehat\sigma_3=\ldots
=\widehat\sigma_{k-1} = S$ \,
which would imply by Lemma~\ref{lemma cyclic homomorphism} that $\omega$ is cyclic which contradicts our assumption.

Hence we can assume that $\Omega_{\mathfrak C_3}$ is non-cyclic. Since $|\mathfrak C_3|=k$ we have by Theorem~\ref{Thm F}~$(c)$
that $\Omega_{\mathfrak C_3}\colon B_{k-2}\to \mathbf{S}(\mathfrak C_3)\cong S_k $ 
(the retraction of $\omega$ to the $3$-component $\mathfrak R$) is conjugate to $\mu_{k-2} \times \kappa$
where $\kappa$ is cyclic and $\mu_{k-2}\colon B_{k-2}\to S_{k-2}$ is the canonical epimorphism (see Definition~\ref{def canonical}).
In particular, this means that there are $C_1,C_2\in \mathfrak C_3$ such that 

$$ C_1^{\widehat \sigma_3} = C_2 \qquad \text{and} \qquad C_2^{\widehat \sigma_3}=C_1. $$

But according to Lemma~\ref{lemma C} and Example~\ref{example C}, $\widehat \sigma_3|_{supp(C_1)\cup supp(C_2)}$ 
must be of cyclic type $[2,2,2]$ or $[6]$ which is impossible in this case.

\textbf{Subcase 2.3 :  $\mathbf{|\mathfrak C_3 |= k-2} $} 

We have shown that this is the only possible value of $|\mathfrak C_3|$. We will now prove that in this case $\omega$
is indeed conjugate to a model homomorphism, as defined in Definition~\ref{Def B_k S_mk}.

According to \eqref{eq r comps} we have in this case

$$ 3\cdot |\mathfrak C_3| + \sum_{r>1,r\neq 3} r\cdot |\mathfrak C_r| = 3k $$

or (since $|\mathfrak C_3| = k-2$)

$$ \sum_{r>1,r\neq 3} r\cdot |\mathfrak C_r| = 6 $$

and in particular for any $r>1, r\neq 3$~:

$$ r\cdot |\mathfrak C_r| \leq 6 $$

which implies that for any $r>1, r\neq 3$~:

\begin{equation}\label{13.1}
|\mathfrak C_r| \leq 3. 
\end{equation}

Hence, according to Theorem~\ref{Thm F}~$(e)$, the retraction of $\omega$ to $\mathfrak C_r$ for any $r>1, r\neq 3$, 
$\Omega_{\mathfrak C_r}\colon B_{k-2} \to \mathbf{S}(\mathfrak C_r)\cong S_{|\mathfrak C_r |}$, 
is cyclic since $k-2 > 3$ (for $k>8$) and  $|\mathfrak C_r |\leq 3$. It follows from Corollary~\ref{cor cyclic omega}
that the reduction of $\omega$ corresponding to $\mathfrak C_r$ for any $r>1, r\neq 3$ is cyclic. 
We claim that $\Omega_{\mathfrak C_3}$, the retraction of $\omega$ to $\mathfrak C_3$, is non-cyclic. For suppose 
$\Omega_{\mathfrak C_3}$ were cyclic, then according to Corollary~\ref{cor cyclic omega} the reduction of $\omega$ corresponding to
$\mathfrak C_3$, $\omega_{supp(\mathfrak C_3)}$, would be cyclic. Now since
 
$$ supp(\mathfrak C_3) \cup \bigcup_{r>1, r\neq 3}supp(\mathfrak C_r) = \bigcup_{r>1} supp(\mathfrak C_r) = supp(\omega) $$  

this would mean that $\omega|_{supp(\omega)} = \omega$ is cyclic on $B_{k-2}$, i.e., $\omega(\sigma_3) = \cdots = \omega(\sigma_{k-1})$ 
and by lemma~\ref{lemma cyclic homomorphism} this means that $\omega$ itself is cyclic which contradicts our assumption.

Hence we can assume that $\Omega_{\mathfrak C_3}\colon B_{k-2}\to \mathbf{S}(\mathfrak C_3)\cong S_{k-2}$ is non-cyclic. 
According to Theorem~\ref{Thm F}~$(d)$, this means that  $\Omega_{\mathfrak C_3}$ is conjugate to $\mu_{k-2}$, the
canonical epimorphism (see Definition~\ref{def canonical}). Since $\widehat \sigma_1$ and $\widehat \sigma_{k-1}$ are conjugate,
they have the same cycle structure and we can argue as above on the $r$-components of $\widehat \sigma_{k-1}$ to deduce
that $\Omega_{\mathfrak C^*_3}$ is conjugate to $\mu_{k-2}$ as well.

By conjugating $\omega$ with an appropriate permutation we may assume w.l.o.g that $\mathfrak C_3^* = \{ A_1,\ldots,A_{k-2} \}$ 
(see Notation~\ref{notation new A}). Furthermore, by Lemma~\ref{lemma normal}
we can assume that $\Omega_{\mathfrak C_3^*} = \mu_{k-2}$ (and not just $\Omega_{\mathfrak C_3^*} \sim  \mu_{k-2}$)
so that

\begin{equation}\label{12.4}
A_i^{\widehat \sigma_j} = \begin{cases}
A_{i+1} & i=j \\
A_{i-1} & i=j+1 \\
A_i & i\neq j,j+1
\end{cases} \qquad 1 \leq j \leq k-3, \ 1\leq i \leq k-2
\end{equation}

In particular we see that for each $i$, $1\leq i \leq k-3$, $\ \widehat \sigma_i$ {\em induces a transposition} between $A_i$ and $A_{i+1}$, i.e.,

\begin{equation}\label{12.1}
A_i^{\widehat \sigma_i} = A_{i+1} \quad \text{and} \quad A_{i+1}^{\widehat \sigma_i} = A_i.
\end{equation}

Hence, 

\begin{equation}\label{12.2}
supp(A_i)\cup supp(A_{i+1})\in Inv(\widehat \sigma_i).
\end{equation}

Denote $C_i = \widehat \sigma_i|_{supp(A_i)\cup supp(A_{i+1})}$.
Then \eqref{12.1} and \eqref{12.2} mean that each cycle in the cyclic decomposition of $C_i$ 
is included in the cyclic decomposition of $\widehat \sigma_i$ and we can write

\begin{equation}\label{12.3}
A_i^{C_i} = A_{i+1} \quad \text{and} \quad A_{i+1}^{C_i} = A_i.
\end{equation}
  
According to Lemma~\ref{lemma C} and as explained
in Example~\ref{example C}, the cyclic type of $C_i$ must be either $[2,2,2]$ or $[6]$ and so, it is 
disjoint from $\mathfrak C^i_3$. This means that the cyclic type of each $\widehat \sigma_i$ has
only two possibilities~:

\begin{equation}\label{eq cyclic type}
[\underbrace{3,\ldots,3}_{k-2 \ times},2,2,2] \qquad \text{or} \qquad [\underbrace{3,\ldots,3}_{k-2 \ times},6] 
\end{equation}

Furthermore, from \eqref{12.4} we get that

\begin{equation}\label{12.5}
A_i^{\widehat \sigma_j} = A_i \qquad \text{for} \ 1\leq i \leq k-2,\  1\leq j \leq k-3 ,\ i\neq j,j+1 .
\end{equation}

But since $supp(A_i)\subset supp(\widehat \sigma_j)$ for these indices 
we have, by Lemma~\ref{Lemma commuting permutations}~$(b)$, that $\widehat \sigma_j|_{supp(A_i)} = A_i^{\pm 1}$.

So far we have concluded that (up to conjugation) 
$\omega|_{supp(\mathfrak C^*_3)}\colon B^*_{k-2}\to \mathbf{S}(supp(\mathfrak C^*_3))\cong S_{3(k-2)}$ satisfies

\begin{equation}\label{eq mid conclusion}
\widehat \sigma_i|_{supp(\mathfrak C^*_3)} = A_1^{\pm 1}\cdots A^{\pm 1}_{i-1}\cdot C_i \cdot A^{\pm 1}_{i+2}\cdots A^{\pm 1}_{k-2} \qquad \text{for} \ i=1,\ldots,k-3 
\end{equation}

Let us use the following notation~:

\begin{equation}\label{eq notation 3 comp}
\mathfrak C_3^i = \{ D^i_1,\ldots,D^i_{i-1}, D^i_{i+2},\ldots,D^i_{k} \} \qquad 1\leq i \leq k-1 
\end{equation}

where $D^i_j = A_j^{\pm 1}$ for $1\leq j \leq k-2, \ 1\leq i \leq k-3$ (this can be assumed according to \eqref{eq mid conclusion})
and $D^{k-1}_j = A_j$ for $1\leq j \leq k-2$ (this is so by assumption). 

% \begin{table}[h]
% \begin{center}
%   \begin{tabular}{ | c || c | c | r c l | c | c | c | }
%     \hline
%     $\widehat \sigma_1$ & $D^1_1$ & $D^1_2$ & $\cdots$ & $\cdots$ & $\cdots$ & $D^1_{k-2}$ & $D^1_{k-1}$ & $D^1_k$ \\ \hline
%     $\widehat \sigma_2$ & $D^2_1$ & $D^2_2$ & $\cdots$ & $\cdots$ & $\cdots$ & $D^2_{k-2}$ & $D^2_{k-1}$ & $D^2_k$ \\ \hline
%     $\widehat \sigma_3$ & $D^3_1$ & $D^3_2$ & $\cdots$ & $\cdots$ & $\cdots$ & $D^3_{k-2}$ & $D^3_{k-1}$ & $D^3_k$\\ \hline
%     & & & & & & & & \\
%     \vdots & \vdots & \vdots & \vdots & \vdots & \vdots & \vdots & \vdots & \vdots \\
%     & & & & & & & & \\ \hline
%     $\widehat \sigma_{k-3}$ & $D^{k-3}_1$ & $D^{k-3}_2$ & $\cdots$ & $\cdots$ & $\cdots$ & $D^{k-3}_{k-2}$ & $D^{k-3}_{k-1}$ & $D^{k-3}_k$\\ \hline
%     $\widehat \sigma_{k-2}$ & $D^{k-2}_1$ & $D^{k-2}_2$ & $\cdots$ & $\cdots$ & $\cdots$ & $D^{k-2}_{k-2}$ & $D^{k-2}_{k-1}$ & $D^{k-2}_k$\\ \hline
%     $\widehat \sigma_{k-1}$ & $D^{k-1}_1$ & $D^{k-1}_2$ & $\cdots$ & $\cdots$ & $\cdots$ & $D^{k-1}_{k-2}$ & $D^{k-1}_{k-1}$ & $D^{k-1}_k$\\
%     \hline
%   \end{tabular}
% \end{center}
% \caption{test}
% \end{table}

% Also write

% $$ \widehat \sigma_{k-1} = D_1\cdots D_{k-2} = A_1\cdots A_{k-2} \cdot C_{k-1} $$

%By \eqref{eq cyclic type} we know that the cyclic type of $C_{k-1}$ is either $[2,2,2]$ or $[6]$.

We will prove the rest of our claim in a series of steps, as follows.

{\bf Step 1:} Up to conjugation in $S_{\Delta_{3k} - supp(\mathfrak C^*_3)}$
(a conjugation that does not alter \eqref{eq mid conclusion}), $\ \omega$ satisfies the following

$$ \widehat \sigma_i|_{\Delta_{3k} - supp(\mathfrak C_3^*)} = A_{k-1}\cdot A_k \qquad \text{for} \quad 1\leq i \leq k-3 $$  

and if we write

$$ \widehat \sigma_{k-1} = A_1\cdots A_{k-2} \cdot C_{k-1} $$

then 

$$ A_{k-1}^{C_{k-1}} = A_k \qquad \text{and} \qquad  A_k^{C_{k-1}} = A_{k-1} $$

\vskip 0.2cm

{\bf Proof of Step 1:}

Since $\widehat \sigma_1$ and $\widehat \sigma_{k-1}$ are conjugate, they have the same cycle structure. 
Hence, from \eqref{13.1} we have that $|\mathfrak C^*_r| \leq 3$ for any $r>1, r\neq 3$. In fact, by \eqref{eq cyclic type}
we know that $\mathfrak C^*_r \neq \emptyset$ for $r>1, r\neq 3$ only if $r$ is equal either to $2$ or $6$. Let $s$
be that value such that $s>1, s\neq 3$ and $\mathfrak C^*_s \neq \emptyset$.

Hence, according to Theorem~\ref{Thm F}~$(e)$, the coretraction of $\omega$ to $\mathfrak C^*_s$, 
$\Omega_{\mathfrak C^*_s}\colon B^*_{k-2} \to \mathbf{S}(\mathfrak C^*_s)\cong S_{|\mathfrak C^*_s |}$, 
is cyclic since $k-2 > 3$ (for $k>8$) and  $|\mathfrak C^*_s |\leq 3$. It follows by Corollary~\ref{cor cyclic omega}
that the coreduction of $\omega$ corresponding to $\mathfrak C^*_s$, $\omega_{supp(\mathfrak C^*_s)}$, is cyclic.

Now since 

$$ supp(\mathfrak C^*_s) = \Delta_{3k} - \bigcup_{i=1,\ldots,k-2} supp(A_i) = \Delta_{3k} - supp(\mathfrak C^*_3) $$

and since, by \eqref{eq mid conclusion}, there are $k-4$ $3$-cycles in the cyclic decomposition of $\widehat \sigma_i|_{supp(\mathfrak C_3^*)}$
for $1\leq i \leq k-3$, 
this means that there are two (disjoint) $3$-cycles, $D_{k-1},D_k\in S_{supp(\mathfrak C^*_s)}$, such that

\begin{equation}\label{107.1}
\widehat \sigma_i|_{\Delta_{3k} - supp(\mathfrak C_3^*)} = D_{k-1}\cdot D_k \qquad \text{for} \quad 1\leq i \leq k-3 
\end{equation}

Putting $i=1$ in \eqref{107.1} and using notation~\eqref{eq notation 3 comp} 
we can write in particular $D_{k-1} = D^1_{k-1}$ and $D_k = D^1_k$.

Now again by \eqref{eq mid conclusion} we have that 
%(using the notation in \eqref{eq notation 3 comp})

$$ (D^1_{i})^{\widehat \sigma_{k-1}|_{supp(\mathfrak C^*_3)}} = D^1_i \qquad \text{for} \quad i=3,\ldots,k-2$$

and if we write

$$ \widehat \sigma_{k-1} = A_1\cdots A_{k-2} \cdot C_{k-1} $$

i.e., $\widehat \sigma_{k-1}|_{supp(C^*_s)} = C_{k-1}$ then since $\Omega_{\mathfrak C_3}\sim \mu_{k-2}$
we necessarily have that

$$ (D^1_{k-1})^{C_{k-1}} = D^1_k \qquad \text{and} \qquad (D^1_k)^{C_{k-1}} = D^1_{k-1} $$

Finally, by conjugating $\omega$ with an appropriate permutation in $S_{\Delta_{3k} - supp(\mathfrak C^*_3)} = S_{supp(\mathfrak C_s)}$
we can assume that $D_{k-1} = A_{k-1}$ and $D_k = A_k$. In terms of notation~\eqref{eq notation 3 comp} this means
that $D^i_{k-1} = A_{k-1}$ and $D^i_k = A_k$ for each $i=1,\ldots,k-3$.

\hfill $\square$

{\bf Step 2:} There are only two possibilities~: Either

$$ D^1_j = A_j \qquad \text{for} \ j=3,\ldots,k-2 $$

or

$$ D^1_j = A^{-1}_j \qquad \text{for} \ j=3,\ldots,k-2 $$

(see the notation in \eqref{eq notation 3 comp})
\vskip 0.2cm

{\bf Proof of Step 2:}

Using the notation in \eqref{eq notation 3 comp} and using step 1 and \eqref{eq mid conclusion} we have that

$$ \widehat \sigma_i = D^i_1 \cdots D^i_{i-1} \cdot C_i \cdot D^i_{i+2}\cdots D^i_{k-2} A_{k-1} A_k \qquad \text{for} \quad i=1,\ldots,k-3 $$

where $D^i_j = A_j^{\pm 1}$. Hence we have by \eqref{12.3}

$$ (D^1_i)^{\widehat \sigma_i} = (D^1_i)^{C_i} = D^1_{i+1} \qquad \text{for} \quad i=3,\ldots,k-3 $$

But according to \eqref{12.3}

$$ A_i^{C_i} = A_{i+1}  \qquad \text{for} \quad i=3,\ldots,k-3 $$

which implies that 

$$ (A^{-1}_i)^{C_i} = A^{-1}_{i+1}  \qquad \text{for} \quad i=3,\ldots,k-3 $$

Hence, for each $i=3,\ldots,k-3$, we have that $D^1_i = A_i$ iff $D^1_{i+1} = A_{i+1}$.
And so, either $D^1_i = A_i$ for all $i=3,\ldots,k-3$ or $D^1_i = A^{-1}_i$ for all $i=3,\ldots,k-3$.

\hfill $\square$

{\bf Step 3:} For each $i$, $1\leq i \leq k-3$, we have that (using the notation in \eqref{eq notation 3 comp})

$$ D^i_j \in \mathfrak C^i_3, \ D^{i+1}_j \in \mathfrak C^{i+1}_3 \quad \text{implies that}   $$

$$ D^i_j = A_j \qquad \text{iff} \qquad D^{i+1}_j = A_j $$

For any $j$, $1\leq j \leq k-3$.

\vskip 0.2cm

{\bf Proof of Step 3:}

Suppose $D^i_j \in \mathfrak C^i_3, \ D^{i+1}_j \in \mathfrak C^{i+1}_3$ for some $i$, $1\leq i \leq k-3$.
Recall (see \eqref{eq notation 3 comp}) that $D^i_j, D^{i+1}_j = A^{\pm 1}_j$. 

Since $D^i_j$ is a cycle in the cyclic decomposition of $\widehat \sigma_i$ for each $j\neq i,i+1$, we have that

\begin{equation}\label{14.3}
(D^i_j)^{\widehat \sigma_{i+1} \widehat \sigma_i} \in \mathfrak C^{i+1}_3 
\end{equation}

But since 

\begin{equation}\label{14.2}
supp(D^i_j) = supp(D^{i+1}_j) \in Inv(\widehat \sigma_i) \cap Inv(\widehat \sigma_{i+1}) 
\end{equation}

we have that

\begin{equation}\label{14.4}
(D^i_j)^{\widehat \sigma_{i+1} \widehat \sigma_i} = (D^i_j)^{D^{i+1}_j D^i_j} = D^{i+1}_j
\end{equation}

Now suppose $D^i_j=A_j$, then 

$$ (D^i_j)^{D^{i+1}_j D^i_j} = A_j^{A^{\pm 1}_j A_j} = A_j  $$ 

which means by \eqref{14.4} that $D^{i+1}_j = A_j$. Similarly, $D^i_j = A^{-1}_j$ implies that 
$D^{i+1}_j = A^{-1}_j$.

\hfill $\square$

{\bf Step 4:} For each $i$, $1 \leq i \leq k-4$, we have 

$$ D^i_{i+2} = A_{i+2} \qquad \text{iff} \qquad D^{i+1}_i = A_i $$

\vskip 0.2cm

{\bf Proof of Step 4:}

Rewriting \eqref{eq mid conclusion} in terms of the notation in \eqref{eq notation 3 comp} we have

\begin{equation}\label{14.5}
\widehat \sigma_i|_{supp(\mathfrak C^*_3)} = D^i_1 \cdots D^i_{i-1}\cdot C_i \cdot D^i_{i+2}\cdots D^i_{k-2} \qquad \text{for} \ i=1,\ldots,k-3 
\end{equation}

where $D^i_j = A^{\pm 1}_j$ for any $j$, $1\leq j \leq k-2$. 

Since $D^i_j$ is a cycle in the cyclic decomposition of $\widehat \sigma_i$ for each $i$, we have that

\begin{equation}\label{14.6}
(D^i_{i+2})^{\widehat \sigma_{i+1} \widehat \sigma_i} \in \mathfrak C^{i+1}_3 \qquad 1\leq i \leq k-3 
\end{equation}

But since

$$ supp(D^i_{i+2}) \cup supp(C_i) = supp(D^{i+1}_i) \cup supp(C_{i+1}) \in Inv(\{ \widehat \sigma_i, \widehat \sigma_{i+1}\}) $$

and since the cyclic type of each $C_i$ is either $[2,2,2]$ or $[6]$ (and in any case not $[3,3]$, see \eqref{eq cyclic type}), 
\eqref{14.6} implies

\begin{equation}\label{14.7}
(D^i_{i+2})^{\widehat \sigma_{i+1} \widehat \sigma_i} = (D^i_{i+2})^{D^{i+1}_i C_{i+1} D^i_{i+2} C_i} = D^{i+1}_i  
\end{equation}

Now assume that $D^i_{i+2} = A_{i+2}$, then

$$ (D^i_{i+2})^{D^{i+1}_i C_{i+1} D^i_{i+2} C_i} = (A_{i+2})^{A^{\pm 1}_i C_{i+1} A_{i+2} C_i } = A_i $$

where we used \eqref{12.3} in the last equality. So \eqref{14.7} implies that $D^{i+1}_i = A_i$ in this case. 
Similarly, $D^i_{i+2} = A^{-1}_{i+2}$ implies that $D^{i+1}_i = A^{-1}_i$.

\hfill $\square$

{\bf Step 5:} There are only two possibilities~: Either 

$$ D^i_j = A_j \qquad \text{for all} \qquad 1\leq i \leq k-3, \ 1\leq j \leq k-2, \ j\neq i,i+1 $$

or

$$ D^i_j = A^{-1}_j \qquad \text{for all} \qquad 1\leq i \leq k-3, \ 1\leq j \leq k-2, \ j\neq i,i+1 $$

\vskip 0.2cm

{\bf Proof of Step 5:}

Rewriting \eqref{eq mid conclusion} in terms of the notation in \eqref{eq notation 3 comp} we have

\begin{equation}\label{15.1}
\widehat \sigma_i|_{supp(\mathfrak C^*_3)} = D^i_1 \cdots D^i_{i-1}\cdot C_i \cdot D^i_{i+2}\cdots D^i_{k-2} \qquad \text{for} \ i=1,\ldots,k-3 
\end{equation}

Let $D^i_j \in \mathfrak C^i_3$ for some $i$, $1\leq i \leq k-3$. By \eqref{15.1} there are two possibilities~: Either  $j \geq i+2$
or $j \leq i-1$. Assume first that $j\geq i+2$.

According to step 3 and by \eqref{15.1} we have that 

\begin{equation}\label{15.2}
D^i_j = A_j \quad \text{iff} \quad D^{i-1}_j = A_j \quad \text{iff} \quad \cdots  \quad \text{iff} \quad D^1_j = A_j 
\end{equation}

Now assume that $j \leq i-1$.

According to step 3 and \eqref{15.1} again we have that

\begin{equation}\label{15.4}
D^i_j = A_j \quad \text{iff} \quad D^{i-1}_j = A_j \quad \text{iff} \quad \cdots  \quad \text{iff} \quad D^{j+1}_j = A_j 
\end{equation}

By step 4 we have that 

\begin{equation}\label{15.5}
D^{j+1}_j = A_j \quad \text{iff} \quad D^j_{j+2} = A_{j+2} 
\end{equation}

and again by step 3

\begin{equation}\label{15.6}
D^j_{j+2} = A_{j+2} \quad \text{iff} \quad D^{j-1}_{j+2} = A_{j+2} \quad \text{iff} \quad \cdots  \quad \text{iff} \quad D^1_{j+2} = A_{j+2} 
\end{equation}

Summing up, we have by \eqref{15.2} that for each $i$, $1\leq i \leq k-3$

\begin{equation}\label{15.8}
D^i_j = A_j \quad \text{iff} \quad D^1_j = A_j \qquad \text{for any} \quad j\geq i+2 
\end{equation}

and by \eqref{15.4}, \eqref{15.5} and \eqref{15.6}

\begin{equation}\label{15.9}
D^i_j = A_j \quad \text{iff} \quad D^1_{j+2} = A_{j+2} \qquad \text{for any} \quad j\leq i-1 
\end{equation}

But according to step 2

\begin{equation}\label{15.7}
D^1_j = A_j \quad \text{iff} \quad D^1_3 = A_3  \qquad \text{for any} \quad 3\leq j\leq k-2
\end{equation}

Hence \eqref{15.8}, \eqref{15.9} and \eqref{15.7} imply that

$$ D^i_j = A_j \quad \text{iff} \quad D^1_3 = A_3 \qquad \text{for any} \quad 1\leq i \leq k-3, \quad 1\leq j\leq k-2, j\neq i,i+1 $$

Since $D^i_j = A_j^{\pm 1}$ this implies the claim.

\hfill $\square$

{\bf Step 6:}  

$$ D^i_j = A_j \qquad \text{for all} \qquad 1\leq i \leq k-3, \ 1\leq j \leq k-2, \ j\neq i,i+1 $$

and 

$$ \widehat \sigma_i|_{\Delta_{3k} - supp(\mathfrak C_3)} = A_1 \cdot A_2 \qquad \text{for} \quad 3\leq i \leq k-1 $$  

\vskip 0.2cm

{\bf Proof of Step 6:}

From \eqref{13.1} we have that $|\mathfrak C_r| \leq 3$ for any $r>1, r\neq 3$. In fact, by \eqref{eq cyclic type}
we know that $\mathfrak C_r \neq \emptyset$ for $r>1, r\neq 3$ only if $r$ is equal either to $2$ or $6$. Let $s$
be that value such that $s>1, s\neq 3$ and $\mathfrak C_s \neq \emptyset$.

Hence, according to Theorem~\ref{Thm F}~$(e)$, the retraction of $\omega$ to $\mathfrak C_s$, 
$\Omega_{\mathfrak C_s}\colon B_{k-2} \to \mathbf{S}(\mathfrak C_s)\cong S_{|\mathfrak C_s |}$, 
is cyclic since $k-2 > 3$ (for $k>8$) and  $|\mathfrak C_s |\leq 3$. It follows from Corollary~\ref{cor cyclic omega}
that the reduction of $\omega$ corresponding to $\mathfrak C_s$, $\omega_{supp(\mathfrak C_s)}$, is cyclic.

Now since 

$$ supp(\mathfrak C_s) = \Delta_{3k} - \bigcup_{i=3,\ldots,k} supp(A_i) = \Delta_{3k} - supp(\mathfrak C_3) $$

and since, by \eqref{eq mid conclusion}, there are $k-4$ $3$-cycles in the cyclic decomposition of $\widehat \sigma_i|_{supp(\mathfrak C_3)}$
for $1\leq i \leq k-3$ and hence also for $i=k-2$ and $i=k-1$, 
this means that there are two (disjoint) $3$-cycles, $D_1,D_2\in S_{supp(\mathfrak C_s)}$, such that

$$ \widehat \sigma_i|_{\Delta_{3k} - supp(\mathfrak C_3)} = D_1\cdot D_2 \qquad \text{for} \quad 3\leq i \leq k-1 $$  

But we already have that 

$$ \widehat \sigma_{k-1}|_{\Delta_{3k} - supp(\mathfrak C_3)} = A_1 \cdot A_2 $$

Hence

$$ \widehat \sigma_i|_{\Delta_{3k} - supp(\mathfrak C_3)} = A_1\cdot A_2 \qquad \text{for} \quad 3\leq i \leq k-1 $$  

In particular, this means that $D^{k-3}_1 = A_1$ and by step 5 this implies that

$$ D^i_j = A_j \qquad \text{for all} \qquad 1\leq i \leq k-3, \ 1\leq j \leq k-2, \ j\neq i,i+1 $$

\hfill $\square$

By combining \eqref{eq mid conclusion}, 
step 1 and step 6 we have established that up to conjugation, $\omega$ satisfies

\begin{equation}\label{16.1}
\widehat \sigma_i = A_1\cdots A_{i-1}\cdot C_i \cdot A_{i+2}\cdots A_k \qquad \text{for} \ i=1,\ldots,k-3 \ \text{and} \ i=k-1
\end{equation}

where

\begin{equation}\label{16.2}
A_i^{C_i} = A_{i+1} \quad \text{and} \quad A_{i+1}^{C_i} = A_i \qquad \text{for} \ i=1,\ldots,k-3 \ \text{and} \ i=k-1
\end{equation}

It is now left to determine $\widehat \sigma_{k-2}$ which we do in the next step.

{\bf Step 7:}  Up to conjugation that does not alter \eqref{16.1} we can assume that $\omega$ 
satisfies

$$ \widehat \sigma_{k-2} = A_1\cdots A_{k-3} C_{k-2} A_k $$

where

$$ (A_{k-2})^{C_{k-2}} = A_{k-1} \qquad \text{and} \qquad (A_{k-1})^{C_{k-2}} = A_{k-2} $$

\vskip 0.2cm

{\bf Proof of Step 7:}

First, by step 6, we already know that

$$ \widehat \sigma_{k-2}|_{\Delta_{3k} - supp(\mathfrak C_3)} = A_1\cdot A_2 $$

Let us now determine $\widehat \sigma_{k-2}|_{supp(\mathfrak C_3)}$.

By \eqref{16.1} and \eqref{16.2}, we have that for any $i$, $1\leq i \leq k-3$ or $i=k-1$ and 
any $D^1_j \in \mathfrak C_3$ where $3\leq j \leq k-1$ and $D^1_j = A_j$

\begin{equation}\label{17.1}
(D^1_j)^{\widehat \sigma_i} = (A_j)^{A_1\cdots A_{i-1}\cdot C_i \cdot A_{i+2}\cdots A_k} = (A_j)^{C_i} = \begin{cases}
A_{i+1} & j=i\\ A_i & j=i+1 \\ A_j & j\neq i,i+1 \end{cases} = \begin{cases}
D^1_{i+1} & j=i\\ D^1_i & j=i+1 \\ D^1_j & j\neq i,i+1 \end{cases} 
\end{equation}

By identifying $D^1_i$ with $i-2\in \Delta_{k-2}$ for $i=3,\ldots,k$, \eqref{17.1} implies that 

$$\Omega_{\mathfrak C_3}(\sigma_i) = (i,i+1)\in \mathbf{S}(\mathfrak C_3)\cong S_{k-2} \qquad \text{for} \quad i=3,\ldots,k-3 
\quad \text{and} \quad i=k-1.$$ 

Now since $\Omega_{\mathfrak C_3}\sim \mu_{k-2}$ this means, by Lemma~\ref{lemma missing mu}, that there
are only two possibilities~: 

$$ \Omega_{\mathfrak C_3}(\sigma_{k-2}) = (k-2,k-1) \quad \text{or} \quad (k-2,k) $$

Assume first that $\Omega_{\mathfrak C_3}(\sigma_{k-2}) = (k-2,k-1)$. This means that 
for any $D^1_j \in \mathfrak C_3$ where $3\leq j \leq k$

\begin{equation}\label{17.2}
(D^1_j)^{\widehat \sigma_{k-2}|_{supp(\mathfrak C_3)}} = \begin{cases}
D^1_{k-1} & j=k-2\\ D^1_{k-2} & j=k-1 \\ D^1_j & j\neq k-2,k-1 \end{cases}  
\end{equation}

We see that for any $j\neq k-2,k-1$, $D^1_j$ commutes with $\widehat \sigma_{k-2}|_{supp(\mathfrak C_3)}$.
Since $supp(D^1_j)\subset supp(\mathfrak C_3)$ and since $D^1_j$ is a $3$-cycle, 
we have by Lemma~\ref{Lemma commuting permutations}~$(b)$ that

$$ \widehat \sigma_{k-2}|_{supp(D^1_j)} = (D^1_j)^{\pm 1} = A^{\pm 1}_j \qquad j\neq k-2,k-1 $$

But since $A_j\in \mathfrak C_3^*$ for each $j=3,\ldots,k-3$ and since $A_k\in \mathfrak C_3^{k-3}$, we can
use step 3 to conclude that

\begin{equation}\label{17.3}
\widehat \sigma_{k-2}|_{supp(D^1_j)} = (D^1_j) = A_j \qquad j\neq k-2,k-1 
\end{equation}

Furthermore, \eqref{17.2} implies that 

$$ supp(D^1_{k-2}) \cup supp(D^1_{k-1}) = supp(A_{k-2}) \cup supp(A_{k-1}) \in Inv(\widehat \sigma_{k-2}) $$

so we can denote $C_{k-2} = \widehat \sigma_{k-2}|_{supp(D^1_{k-2}) \cup supp(D^1_{k-1})}$ and by \eqref{17.2} we have

$$ (D^1_{k-2})^{C_{k-2}} = D^1_{k-1} \qquad \text{and} \qquad (D^1_{k-1})^{C_{k-2}} = D^1_{k-2} $$

or 

\begin{equation}\label{17.4}
(A_{k-2})^{C_{k-2}} = A_{k-1} \qquad \text{and} \qquad (A_{k-1})^{C_{k-2}} = A_{k-2}
\end{equation}

Combining \eqref{17.3} and \eqref{17.4} we have that

$$ \widehat \sigma_{k-2} = A_1\cdots A_{k-3} C_{k-2} A_k $$

and the claim follows in this case.

Now assume that $\Omega_{\mathfrak C_3}(\sigma_{k-2}) = (k-2,k)$. Then in this case we have that 
for any $D^1_j \in \mathfrak C_3$ where $3\leq j \leq k-1$

\begin{equation}\label{17.5}
(D^1_j)^{\widehat \sigma_{k-2}|_{supp(\mathfrak C_3)}} = \begin{cases}
D^1_k & j=k-2\\ D^1_{k-2} & j=k \\ D^1_j & j\neq k-2,k \end{cases}  
\end{equation}

By arguing in a completely similar way to the case $\Omega_{\mathfrak C_3}(\sigma_{k-2}) = (k-2,k-1)$ we can
conclude that

$$ \widehat \sigma_{k-2} = A_1\cdots A_{k-3} A_{k-1} C'_{k-2} $$

where $supp(C'_{k-2}) = supp(A_{k-2}) \cup supp(A_k)$ and

\begin{equation}\label{17.6}
(A_{k-2})^{C'_{k-2}} = A_k \qquad \text{and} \qquad (A_k)^{C'_{k-2}} = A_{k-2} 
\end{equation}

Recall that

\begin{equation}\label{17.7}
(A_{k-1})^{C_{k-1}} = A_k \qquad \text{and} \qquad (A_k)^{C_{k-1}} = A_{k-1} 
\end{equation}

Furthermore, since $supp(C_{k-1}) = supp(A_{k-1}) \cup supp(A_k)$, then

\begin{equation}\label{17.8}
(A_j)^{C_{k-1}} = A_j \qquad \text{for any} \qquad j\neq k-1,k 
\end{equation}

and since $supp(C'_{k-2}) = supp(A_{k-2}) \cup supp(A_k)$, then

\begin{equation}\label{17.9}
(A_j)^{C'_{k-2}} = A_j \qquad \text{for any} \qquad j\neq k-2,k 
\end{equation}

Let us now conjugate $\omega$ by $C_{k-1}$. By \eqref{16.1} and \eqref{16.2} we have for any $i$, $1\leq i \leq k-3$

$$ (\widehat \sigma_i)^{C_{k-1}} = (A_1\cdots A_{i-1}\cdot C_i \cdot A_{i+2}\cdots A_k)^{C_{k-1}} = $$

$$ = A_1^{C_{k-1}}\cdots A_{i-1}^{C_{k-1}}\cdot C_i^{C_{k-1}} \cdot A_{i+2}^{C_{k-1}}\cdots A_k^{C_{k-1}} =  $$

$$ = A_1\cdots A_{i-1}\cdot C_i \cdot A_{i+2}\cdots A_k = \widehat \sigma_i $$

The last equation follows from the fact that $supp(C_i) = supp(A_i) \cup supp(A_{i+1})$ for $1\leq i \leq k-3$
is disjoint from $supp(C_{k-1}) = supp(A_{k-1}) \cup supp(A_k)$.

For $i=k-1$ we have

$$ (\widehat \sigma_{k-1})^{C_{k-1}} = (A_1 \cdots A_{k-2} \cdot C_{k-1})^{C_{k-1}} =  $$

$$ = A_1^{C_{k-1}} \cdots A_{k-2}^{C_{k-1}} \cdot C_{k-1}^{C_{k-1}} = A_1 \cdots A_{k-2} \cdot C_{k-1} = \widehat \sigma_{k-1} $$

And finally, for $i=k-2$ we have

$$ (\widehat \sigma_{k-2})^{C_{k-1}} = (A_1\cdots A_{k-3} A_{k-1} C'_{k-2})^{C_{k-1}} =  $$

$$ A_1^{C_{k-1}}\cdots A_{k-3}^{C_{k-1}} A_{k-1}^{C_{k-1}} (C'_{k-2})^{C_{k-1}} = $$

\begin{equation}\label{17.10}
= A_1\cdots A_{k-3} A_k (C'_{k-2})^{C_{k-1}}
\end{equation}

Denote $C_{k-2} = (C'_{k-2})^{C_{k-1}}$. Now \eqref{17.10} implies that $supp(C_{k-2}) = supp(A_{k-2}) \cup supp(A_{k-1})$. Furthermore

$$ (A_{k-2})^{C_{k-2}} = (A_{k-2})^{C^{-1}_{k-1} C'_{k-2} C_{k-1}} = (A_{k-2})^{C'_{k-2} C_{k-1}} = (A_k)^{C_{k-1}} = A_{k-1}$$

and

$$ (A_{k-1})^{C_{k-2}} = (A_{k-1})^{C^{-1}_{k-1} C'_{k-2} C_{k-1}} = (A_k)^{C'_{k-2} C_{k-1}} = (A_{k-2})^{C_{k-1}} = A_{k-2}$$

Summing up, we have shown that conjugating $\omega$ by $C_{k-1}$ does not alter \eqref{16.1}, and $\widehat \sigma_{k-2}$ satisfies

$$ \widehat \sigma_{k-2} = A_1\cdots A_{k-3} C_{k-2} A_k $$

where

$$ (A_{k-2})^{C_{k-2}} = A_{k-1} \qquad \text{and} \qquad (A_{k-1})^{C_{k-2}} = A_{k-2} $$

\hfill $\square$

Combining \eqref{16.1} and step 7 we have established that, up to conjugation, $\omega$ satisfies

\begin{equation}\label{18.1}
\widehat \sigma_i = A_1\cdots A_{i-1}\cdot C_i \cdot A_{i+2}\cdots A_k \qquad \text{for} \ i=1,\ldots,k-1 
\end{equation}

where

\begin{equation}\label{18.2}
A_i^{C_i} = A_{i+1} \quad \text{and} \quad A_{i+1}^{C_i} = A_i \qquad \text{for} \ i=1,\ldots,k-1 
\end{equation}

Using Notation~\ref{notation C} we can write

$$ C_i = C^{(i,i+1),3}_{t^i_1,t^i_2} \qquad \text{for each} \quad i=1,\ldots,k-1 $$

where $t^i_j \in \mathbb \bigtriangledown_3$.

Since $\omega$ is a homomorphism, we have according to Lemma~\ref{lemma prep phi} (in the case $l=1$, $m=3$ in which case
$\psi_l(\sigma_i) = \psi_1(\sigma_i) = (i,i+1)$) that 

\begin{equation}\label{18.3}
C_i \infty C_{i+1} \qquad \text{for} \quad i=1,\ldots,k-2.
\end{equation}

or equivalently, using Notation~\ref{notation C}

\begin{equation}\label{18.4}
C^{(i,i+1),3}_{t^i_1,t^i_2} \ \infty \ C^{(i+1,i+2),3}_{t^{i+1}_1,t^{i+1}_2} \qquad \text{for} \quad i=1,\ldots,k-2
\end{equation}

According to Lemma~\ref{lemma condition C_i}, \eqref{18.4} is equivalent to

$$ \quad t^i_i + t^i_{i+1} \equiv t^{i+1}_{i+2} + t^{i+1}_{i+1} (\text{mod} \ 3) \qquad \text{for} \quad i=1,\ldots,k-2 $$

which is condition \eqref{eq condition t} in Definition~\ref{Def B_k S_mk} for $m=3$ and $l=1$. Thus we see
that \eqref{18.1} is a model homomorphism given by \eqref{eq model phi} in Definition~\ref{Def B_k S_mk}
for the case $m=3$ and $l=1$. 

We have shown that any good non-cyclic and transitive homomorphism $\omega\colon B_k \to S_{3k}$ 
is conjugate to a model homomorphism, hence each such homomorphism is standard.

\end{proof}

% \section{Finite index subgroups of the braid group}\label{section braid subgroups}

% As is well known there is a correspondence between finite index subgroups of a group $G$ and
% its permutation representations : A permutation representation $G \to S_k$ is sent to the
% index $k$ subgroup which stabilizes (w.l.o.g) $1\in \Delta_k$ and given an index $k$
% subgroup of $G$ we get a permutation representation $G \to S_k$ by letting $G$ 
% act on the cosets of the subgroup by left multiplication. 

\section{An Updated Conjecture}\label{section conjecture}

It is already known (see Theorem~\ref{Thm F}~$(a)$ through $(e)$) that any non-cyclic transitive homomorphism $B_k\to S_n$
is good for $n \leq 2k $, $k>8$. Our first conjecture then is the following

\begin{conj}\label{conj good}
Every non-cyclic transitive homomorphism $B_k\to S_n$ for $n > 2k$, $k>8$,  is good (see Definition~\ref{def good}).\hfill $\bullet$
\end{conj}

It was conjectured in \cite{MaSu05} that there is no non-cyclic transitive homomorphism $B_k\to S_n$ 
for $k\not | \ n$ but as we shall see in Example~\ref{example conj} there are an infinite number
of refutations for this conjecture. First we define

\begin{definition}\label{def refute conj}
Let $2\leq k \in \mathbb N$, $1 \leq m \in \mathbb N$ and let $H$ be an index $n$ subgroup of $S_{mk}$
where $k\not | \ n$. Let $\varphi$ be one of the standard homomorphisms (Definition~\ref{Def B_k S_mk}) if $m\geq 2$ or
the canonical homomorphism (Definition~\ref{def canonical}) if $m=1$. Then the homomorphism $f\colon B_k\to S_n$ which
sends $\sigma_i \in B_k$ to the permutation in $S_n$ that $\varphi(\sigma_i)$ induces on the $n$ cosets of $H$ in $S_{mk}$
by left multiplication is said to be \textbf{derived from $\mathbf{\varphi}$ using $\mathbf{H}$}. In general, such a homomorphism $f$ is 
said to be \textbf{derived from a canonical or a standard homomorphism}. \hfill $\bullet$
\end{definition}

Note that a homomorphism $B_k \to S_n$ derived from a canonical or a standard homomorphism is not necessarily transitive nor is it necessarily non-cyclic
as the images of the standard homomorphisms do not generate $S_{mk}$ in general. However, we do have

\begin{exa}\label{example conj}
In Definition~\ref{def refute conj}, we take $m=1$, 

$$H = \mathbf{S}(\{1,2\}) \times \mathbf{S}(\{3,\ldots,k\}) \cong  S_2\times S_{k-2}$$ 

as a $\frac{k(k-1)}{2} $ index subgroup of $S_k$. Consider the homomorphism 
$f\colon B_k \to S_{\frac{k(k-1)}{2}}$ which
is derived from $\mu_k\colon B_k\to S_k$ using 
$H$ (see definition~\ref{def canonical}). Then

\textbf{f is transitive : } First note that $\mu_k$ is surjective : $\mu_k(\sigma_i) = (i,i+1)$ which generate $S_k$ for $i=1,\ldots,k-1$.
Hence, for any $g_1, g_2 \in S_k$ there exists an element $g\in B_k$ such that $\mu_k(g) = g_1 g_2^{-1}$. So, 
for any two cosets of $H$ in $S_k$, say $g_1H$ and $g_2H$, we have that

$$ g_1H = (g_1 g_2^{-1}) g_2H = \mu_k(g) g_2H $$

which means, by definition of $f$, that $f$ is transitive.

\textbf{f is non-cyclic : } If $f$ were cyclic then according to Lemma~\ref{lemma cyclic homomorphism} 
and the definition of $f$, we would have that for each $i=1,\ldots,k-1$, $\mu_k(\sigma_i) = (i,i+1)$ would induce the same
permutation on the cosets of $H$ in $S_k$ by multiplication on the left. However, since $\mu_k(\sigma_1)=(1,2)\in H$ then $(1,2)$
sends the coset $H$ to itself while the element $\mu_k(\sigma_2)=(2,3)\not \in H$ sends the coset $H$ to $(2,3)H \neq H$. Hence
$f$ is non-cyclic.

Finally, note that if $k$ is even then $k \nmid  \frac{k(k-1)}{2}$, and so, for any even $k$, this example indeed refutes the conjecture
made in \cite{MaSu05} that there is no non-cyclic transitive homomorphism $B_k\to S_n$ 
for $k\not | \ n$.

\hfill $\bullet$
\end{exa}

We now give the following

\begin{conj}
Every good non-cyclic transitive homomorphism $B_k \to S_n$ is either canonical (if $n=k$), standard (if $\ k\neq n, \ k\ | \ n$) or
derived from a canonical or a standard homomorphism (if $k \not | \ n$). \hfill $\bullet$
\end{conj}

\end{document}